\documentclass[a4paper]{amsart}
\usepackage{graphicx}
\usepackage{amssymb}
\usepackage{amsmath}
\usepackage{amsthm}
\usepackage{amscd}
\usepackage[all,2cell]{xy}

\UseAllTwocells \SilentMatrices
\newtheorem{thm}{Theorem}[section]

\newtheorem{cor}[thm]{Corollary}
\newtheorem{lem}[thm]{Lemma}
\newtheorem{exm}[thm]{Example}

\newtheorem{prop}[thm]{Proposition}
\theoremstyle{definition}
\newtheorem{defn}[thm]{Definition}
\theoremstyle{remark}
\newtheorem{rem}[thm]{\bf Remark}
\numberwithin{equation}{section}

\begin{document}
\title[The dual actions, equiv. autoequiv. and stable tilting objects]{The dual actions, equivariant autoequivalences and stable tilting objects}

\author[Jianmin Chen, Xiao-Wu Chen, Shiquan Ruan] {Jianmin Chen, Xiao-Wu Chen$^*$, Shiquan Ruan}

\thanks{$^*$ the corresponding author}
\subjclass[2010]{18E30, 16S35, 58E40, 16D99}
\date{the revised version, \today}
\keywords{group action, equivariantization, autoequivalence,  tilting object, weighted projective line}%
\maketitle

\dedicatory{}%
\commby{}%

\begin{abstract}
For a finite abelian group action on a linear category, we study the dual action given by the character group acting on the category of equivariant objects. We prove that the groups of equivariant autoequivalences on these two categories are isomorphic. In the triangulated situation, this isomorphism implies that the classifications of stable tilting objects for these two categories are in a natural bijection. We apply these results to stable tilting complexes on weighted projective lines of tubular type.
\end{abstract}

\tableofcontents

\section{Introduction}

Weighted projective lines are introduced in \cite{GL87}, which provide a geometric approach to canonical algebras in the sense of \cite{Rin}. It is well known that weighted projective lines are related to   smooth projective curves via equivariantization; compare \cite{Po, Sch, Len16}. More precisely, let $X$ be a smooth projective curve, and  $G$ be a finite group of automorphisms on $X$ such that $X/G$ is isomorphic to the projective line. Then the category of $G$-equivariant coherent sheaves on $X$ is equivalent to the category of coherent sheaves on some weighted projective line, whose  weight structure  is given by  the ramification data of the quotient map $X \to X/G$.

The equivariantization with respect to various finite group actions is very important in the study of weighted projective lines and their derived categories. We are motivated by the classification of $\tau^2$-stable tilting complexes on weighted projective lines, where $\tau$ is the Auslander-Reiten translation; see \cite{Jas}. We observe that $\tau^2$ induces a cyclic group action on the derived category of a weighted projective line.  Then it is natural to ask how this $\tau^2$-action interacts with the tilting theory on the weighted projective line. More generally, given a triangulated category $\mathcal{T}$ with a finite group action, we are interested in  the interaction of the equivariantization and the tilting theory on $\mathcal{T}$. This leads to the study of stable tilting objects in $\mathcal{T}$, that is, those tilting objects fixed by the group action. We mention that stable tilting objects arise naturally in quite different setups; see \cite{Asa11, Asa13, Nov}.

Let us describe the main results of this work. For this, we explain first what  the classification of (stable) tilting objects really means.

Let $k$ be a field, and $\mathcal{T}$ be a $k$-linear triangulated category. The tilting objects and their endomorphism algebras are of great interest to representation theorists. We denote by ${\rm Tilt}(\mathcal{T})$ the set of isoclasses of basic tilting objects. The group ${\rm Aut}_\vartriangle(\mathcal{T})$ of triangle autoequivalences on $\mathcal{T}$ acts naturally on ${\rm Tilt}(\mathcal{T})$. Then the classification of basic tilting objects in $\mathcal{T}$ boils down to the orbit set ${\rm Tilt}(\mathcal{T})/{\rm Aut}_\vartriangle(\mathcal{T})$. Indeed, this orbit set is in a bijection to the set of isoclasses of  the endomorphism algebras of  those tilting objects; see Corollary \ref{cor:bij}.

Let $G$ be a finite abelian group, which acts on $\mathcal{T}$. There is an equivariant version of the classification problem. More precisely, we denote by ${\rm Tilt}^G(\mathcal{T})$ the subset of ${\rm Tilt}(\mathcal{T})$ formed by $G$-stable objects, that is, those basic tilting objects fixed by the $G$-action. We denote by ${\rm Aut}^G_\vartriangle(\mathcal{T})$ the centralizer of $G$ in ${\rm Aut}_\vartriangle(\mathcal{T})$. Then ${\rm Aut}^G_\vartriangle(\mathcal{T})$ acts on ${\rm Tilt}^G(\mathcal{T})$ naturally. We should be concerned with the corresponding orbit set.

However, it seems that the group ${\rm Aut}^G_\vartriangle(\mathcal{T})$ is not the right object to study in the equivariant situation. Instead, we take the group ${\rm Aut}_\vartriangle(\mathcal{T}; G)$ of $G$-\emph{equivariant} triangle autoequivalences on $\mathcal{T}$ into consideration. Moreover, we have a forgetful homomorphism
$${\rm Aut}_\vartriangle(\mathcal{T}; G)\longrightarrow {\rm Aut}^G_\vartriangle(\mathcal{T}),$$
 which is in general neither injective nor surjective. We will be concerned with the orbit set ${\rm Tilt}^G(\mathcal{T})/{\rm Aut}_\vartriangle(\mathcal{T}; G)$, which parameterizes  the classification of basic $G$-stable tilting objects in $\mathcal{T}$. This orbit set is in a bijection to a certain set of equivalence classes of $G$-crossed systems; see Proposition \ref{prop:tri-stable}.

We denote by $\mathcal{T}^G$ the category of $G$-equivariant objects; in nice situations, it is naturally triangulated \cite{Bal,Chen15,El2014}. The character group $\widehat{G}$ of $G$ acts naturally on $\mathcal{T}^G$, called the \emph{dual action}; see Definition \ref{defn:dual-act}. Then the classification of basic $\widehat{G}$-stable tilting objects in $\mathcal{T}^G$ is parameterized by the orbit set ${\rm Tilt}^{\widehat{G}}(\mathcal{T}^G)/{\rm Aut}_\vartriangle(\mathcal{T}^G; \widehat{G})$. Here, ${\rm Tilt}^{\widehat{G}}(\mathcal{T}^G)$ denotes the set of isoclasses of basic $\widehat{G}$-stable tilting objects in $\mathcal{T}^G$, which carries a natural action by the group ${\rm Aut}_\vartriangle(\mathcal{T}^G; \widehat{G})$ of $\widehat{G}$-equivariant triangle autoequivalences on $\mathcal{T}^G$.

The first main result claims that the two classifications are in a natural bijection; see Theorems \ref{thm:duality-iso-tri} and \ref{thm:tilting}.

\vskip 5pt

\noindent {\bf Theorem A}. \emph{Let $G$ be a finite abelian group, which splits over $k$. Let $\mathcal{T}$ be a triangulated category, which is Hom-finite and Krull-Schmidt.  Assume that $G$ acts on $\mathcal{T}$ by triangle autoequivalences. Consider the dual $\widehat{G}$-action on $\mathcal{T}^G$. Then there exist an isomorphism of groups
$${\rm Aut}_\vartriangle(\mathcal{T}; G) \stackrel{\sim}\longrightarrow {{\rm Aut}_\vartriangle(\mathcal{T}^G; \widehat{G})},$$
and a bijection  between the sets of isoclasses
$${\rm Tilt}^G(\mathcal{T})\stackrel{\sim}\longrightarrow  {\rm Tilt}^{\widehat{G}}(\mathcal{T}^G),$$
which are compatible. In particular, we have an induced bijection between the orbit sets
 $${\rm Tilt}^G(\mathcal{T})/{{\rm Aut}_\vartriangle(\mathcal{T}; G)}\stackrel{\sim}\longrightarrow {\rm Tilt}^{\widehat{G}}(\mathcal{T}^G)/{{\rm Aut}_\vartriangle(\mathcal{T}^G; \widehat{G})}.$$}

\vskip 5pt

The main ingredient of the proof is the duality theorem in \cite{El2014}, which states that the category of $\widehat{G}$-equivariant objects in $\mathcal{T}^G$ is naturally triangle equivalent to $\mathcal{T}$.  We mention that the duality theorem is implicit  in \cite{RR} with a completely different setup, and might be deduced from a general theorem in \cite{DGNO}.

Theorem A  may have potential application for  finite abelian group actions on varieties of higher dimension \cite{Pl}. In this paper, we apply it to weighted projective lines of tubular type. We obtain new insight on the classification of $\tau^2$-stable tilting complexes in \cite{Jas}, where the Auslander-Reiten translation  $\tau$ is given by the degree-shift with respect to the dualizing element $\vec{\omega}$ in the Picard group; see \cite{GL87}.

 Let $\mathbb{Y}$ be a weighted projective line of tubular type $(6,3,2)$ or $(4,4,2)$, where $p=6$ or $4$ is the order of $\vec{\omega}$, respectively. Assume that $a$ is a prime number, which divides $p$. Then the cyclic group $\mathbb{Z}(\frac{p}{a}\vec{\omega})$ has order $a$, and acts on the category ${\rm coh}\mbox{-}\mathbb{Y}$ of coherent sheaves on $\mathbb{Y}$ via degree-shift by $\frac{p}{a}\vec{\omega}$, or equivalently, by $\tau^{\frac{p}{a}}$.  Then the category  $({\rm coh}\mbox{-}\mathbb{Y})^{\mathbb{Z}(\frac{p}{a}\vec{\omega})}$ of equivariant sheaves is equivalent to the category ${\rm coh}\mbox{-}\mathbb{X}$ for another weighted projective line $\mathbb{X}$ of tubular type; see \cite{CC17, Len16}.

We denote by $C_{p, a}$  the character group  of $\mathbb{Z}(\frac{p}{a}\vec{\omega})$, which is also cyclic of order $a$. Then we have the dual $C_{p, a}$-action on $({\rm coh}\mbox{-}\mathbb{Y})^{\mathbb{Z}(\frac{p}{a}\vec{\omega})}$. On the other hand, there is an explicit automorphism $g_{p, a}$ on $\mathbb{X}$ of order $a$, which induces a $C_{p,a}$-action on ${\rm coh}\mbox{-}\mathbb{X}$; see Subsection~8.1. These actions extend to their bounded derived categories, and we will consider the corresponding equivariant derived categories.

The following second main result summarizes  Propositions \ref{4,4,2}, \ref{6,3,2 (I)} and \ref{6,3,2 (II)}.
\vskip 5pt

\noindent {\bf Theorem B}. \emph{Keep the notation as above. Then the following statements hold.
\begin{enumerate}
\item[(1)] There is an equivalence of triangulated categories
$$\mathbf{D}^b({\rm coh}\mbox{-}\mathbb{X}) \stackrel{\sim}\longrightarrow \mathbf{D}^b({\rm coh}\mbox{-}\mathbb{Y})^{\mathbb{Z}(\frac{p}{a}\vec{\omega})},$$
which is equivariant with respect to the above two $C_{p,a}$-actions.  Consequently, we have an equivalence of triangulated categories
    $$ \mathbf{D}^b({\rm coh}\mbox{-}\mathbb{X})^{C_{p,a}} \stackrel{\sim} \longrightarrow \mathbf{D}^b({\rm coh}\mbox{-}\mathbb{Y}).$$
    \item[(2)] There is a bijection between the sets of isoclasses
   \begin{equation*}
    \left\{
    \begin{aligned}
    &\tau^{\frac{p}{a}}\mbox{-stable tilting complexes} \\
    &\qquad  \mbox{ on } \mathbb{Y}
    \end{aligned}
    \right\}
    \stackrel{\iota} \longrightarrow
       \left\{
    \begin{aligned}
    & g_{p, a}\mbox{-stable tilting complexes } \\
    & \qquad \mbox{ on } \mathbb{X}.
    \end{aligned}
    \right\},
    \end{equation*}
    which is compatible with the actions by the equivariant triangle autoequivalence groups.
\end{enumerate}}

\vskip 5pt

We mention that the equivalences in (1) are based on results in  \cite{CC17}; see also \cite{Len16}. The bijection in (2) follows from Theorem A.

The paper is organized as follows.  Let $G$ be a finite group. In Section 2, we recall the notion of equivariant functors between two categories with $G$-actions. We observe that any group actions on module categories, that fix the regular modules, correspond to crossed systems; this observation is applied to $G$-stable objects. In Section 3, we prove that up to isomorphism, a cyclic group action is in a bijection to certain compatible pairs. For module categories, this is related to certain elements in the outer automorphism group of the algebra.

We recall the associated monads for a group action in Section 4. Assume now that $G$ is abelian. For a given $G$-action on a linear category $\mathcal{C}$, we have the dual action on the category $\mathcal{C}^G$ of $G$-equivariant objects by the character group $\widehat{G}$. In Theorem \ref{thm:duality}, we describe explicitly the equivalence between $\mathcal{C}$ and the category of $\widehat{G}$-equivariant objects in $\mathcal{C}^G$. This result is due to \cite{El2014}.  In Section 5, we prove that the groups of equivariant autoequivalences with respect to the given action and its dual action are naturally isomorphic; see Theorem \ref{thm:duality-iso}. We obtain an additive version of Theorem A in Proposition \ref{prop:stable-bij}, which claims a natural bijection between the sets of isoclasses of  basic stable objects.

We study triangle $G$-actions on a triangulated category $\mathcal{T}$ in Section 6. We obtain Theorem \ref{thm:duality-iso-tri}, which is analogous to Theorem \ref{thm:duality-iso} on equivariant triangle autoequivalences. In Section 7, we prove Theorem \ref{thm:tilting}, which claims that the classification of basic stable tilting objects for the given action and that for its dual action are in natural bijections.

In Section 8, we apply the results to the classifications of stable tilting complexes on weighted projective lines of different tubular types. In Section 9, we obtain an exact sequence, which involves the forgetful homomorphism and the second cohomological groups of $G$ with values in $k^*$.

We collect in Appendix A some identities for group actions and provide full proofs. In Appendix B, we recall the notion of strongly $\mathbf{K}$-standard category from \cite{CY}. This subtle notion plays a role in the study of stable tilting objects in Section~7.

\section{Group actions and equivariantization}

In this section, we recall basic notions on group actions and equivariantization. Any group action on a module category fixing the regular module corresponds to a  weak action on the given ring. Stable objects under the action naturally give rise to crossed systems.

\subsection{Group actions and equivariant functors} We will recall basic facts on group actions and equivariant functors. The details are found in \cite{DGNO}; compare \cite{De,RR,El2014}.

Let $G$ be an arbitrary group. We write $G$ multiplicatively and denote its unit by $e$. Let $\mathcal{C}$ be an arbitrary category.

A $G$-\emph{action} on $\mathcal{C}$ consists of the data $\{F_g, \varepsilon_{g, h}|\; g, h\in G\}$, where each $F_g\colon \mathcal{C}\rightarrow \mathcal{C}$ is an autoequivalence and each $\varepsilon_{g, h}\colon F_gF_h\rightarrow F_{gh}$ is a natural isomorphism such that the following $2$-cocycle condition holds
\begin{align}\label{equ:2-coc}
\varepsilon_{gh, k}\circ \varepsilon_{g,h}F_k=\varepsilon_{g, hk}\circ F_g\varepsilon_{h,k}
\end{align}
for all $g, h, k\in G$. This condition is equivalent to the following commutative diagram.
\[\xymatrix{F_gF_hF_k \ar[rr]^-{\varepsilon_{g,h}F_k} \ar[d]_-{F_g\varepsilon_{h,k}}&& F_{gh}F_k \ar[d]^-{\varepsilon_{gh, k}}\\
F_gF_{hk}\ar[rr]^-{\varepsilon_{g, hk}} && F_{ghk}
}\]
We observe that there exists a unique natural isomorphism $u\colon F_e\rightarrow {\rm Id}_\mathcal{C}$, called the \emph{unit} of the action, satisfying $\varepsilon_{e,e}=F_eu$; moreover, we have $F_{e}u=uF_e$ by (\ref{equ:2-coc}); see also Lemma \ref{lemA:G-action}.

The given $G$-action is \emph{strict} provided that each $F_g$ is an automorphism on $\mathcal{C}$ and that each isomorphism $\varepsilon_{g, h}$ is the identity, in which case the unit $u$ equals the identity. We observe that a strict $G$-action on $\mathcal{C}$ coincides with  a group homomorphism from $G$ to the automorphism group of $\mathcal{C}$.

Let us fix a $G$-action $\{F_g, \varepsilon_{g, h}|\; g, h\in G\}$ on $\mathcal{C}$. An object $X$ is \emph{$G$-stable} provided that $X$ is isomorphic to $F_g(X)$ for each $g\in G$. In this case, we also say that the $G$-action \emph{fixes} $X$. The full subcategory of $G$-stable objects does not behave well. For example, even if the category $\mathcal{C}$ is abelian, the subcategory of $G$-stable objects is not abelian in general.

A \emph{$G$-equivariant object} in $\mathcal{C}$ is a pair $(X, \alpha)$, where $X$ is an object in $\mathcal{C}$ and $\alpha$ assigns, for each $g\in G$, an isomorphism $\alpha_g\colon X\rightarrow F_g(X)$ subject to the relations
 \begin{align}\label{equ:rel}
 \alpha_{gh}=(\varepsilon_{g,h})_X \circ F_g(\alpha_{h}) \circ \alpha_g.
 \end{align}
 These relations imply that $\alpha_e=u^{-1}_X$. A morphism $f\colon (X, \alpha)\rightarrow (Y, \beta)$ between  two $G$-equivariant objects is a morphism $f\colon X\rightarrow Y$ in $\mathcal{C}$ such that
  $\beta_g\circ f=F_g(f)\circ \alpha_g$ for each $g\in G$. We denote by
$\mathcal{C}^G$    the category of $G$-equivariant objects. The \emph{forgetful functor}
   $U\colon \mathcal{C}^G\rightarrow \mathcal{C}$ is defined by $U(X, \alpha)=X$ and $U(f)=f$.

The following observation will be useful. \emph{Associated to} two given $G$-equivariant objects $(X, \alpha)$ and $(Y, \beta)$, the Hom set ${\rm Hom}_\mathcal{C}(X, Y)$ carries a (left) $G$-action by $g.f=\beta_g^{-1}\circ F_g(f)\circ \alpha_g$ for each $g\in G$ and $f\colon X\rightarrow Y$. Then by the very definition,  we have the following identity
\begin{align}\label{equ:fixed}
{\rm Hom}_{\mathcal{C}^G}((X, \alpha), (Y, \beta))={\rm Hom}_\mathcal{C}(X, Y)^G.
\end{align}
Here, for any set $S$ with a $G$-action, we denote by $S^G$ the subset of fixed elements.

For a given $G$-action on $\mathcal{C}$,  the process forming the category $\mathcal{C}^G$ of $G$-equivariant objects is known as the \emph{equivariantization}; see \cite[Section 4]{DGNO}.

Assume that $\mathcal{C}$ has a $G$-action $\{F_g, \varepsilon_{g, h}|\; g, h\in G\}$ and that $\mathcal{D}$ is another category with a $G$-action $\{F'_g, \varepsilon'_{g, h}|\; g, h\in G\}$.  Let $F\colon \mathcal{C}\rightarrow \mathcal{D}$ be a functor. The functor $F$ is \emph{$G$-equivariant} with respect to these two $G$-actions provided that there are natural isomorphisms $\delta_g\colon FF_g\rightarrow F'_gF$ of functors subject to the conditions
\begin{align}\label{equ:equi-fun}
\delta_{gh}\circ F\varepsilon_{g, h}=\varepsilon'_{g, h}F\circ (F'_g\delta_h\circ \delta_gF_h).
\end{align}
That is, the following diagram is required to be commutative.
\[\xymatrix{
FF_gF_h \ar[rr]^-{F\varepsilon_{g,h}} \ar[d]_-{F'_g\delta_h\circ \delta_g F_h} &&  F F_{gh}\ar[d]^-{\delta_{gh}} \\
F'_gF'_h F \ar[rr]^-{\varepsilon'_{g, h}F} && F'_{gh}F
}\]
Indeed, by a \emph{$G$-equivariant functor}, we really mean the data $(F, (\delta_g)_{g\in G})$, which will be abbreviated as $(F, \delta)$; compare \cite[Definition 4.8]{Asa11}. The composition of two equivariant functors
$$ \mathcal{C}\xrightarrow{(F, \delta)} \mathcal{D} \xrightarrow{(E, \partial)} \mathcal{E}$$
is defined to be $(EF, (\partial_gF\circ E\delta_g)_{g\in G})=(EF, \partial F\circ E\delta)$.

We say that the functor $F\colon \mathcal{C}\rightarrow \mathcal{D}$ is \emph{strictly $G$-equivariant} provided that $FF_g=F'_gF$ for each $g\in G$ and that $F\varepsilon_{g, h}=\varepsilon'_{g, h}F$ for each $g, h\in G$. In other words, $(F, ({\rm Id}_{FF_g})_{g\in G})$ is a $G$-equivariant functor. We will denote this strictly $G$-equivariant functor simply by $F$.

A $G$-equivariant functor $(F, \delta)\colon \mathcal{C}\rightarrow \mathcal{D}$ gives rise to a functor
\begin{align}\label{equ:fun-equi}
(F, \delta)^G\colon \mathcal{C}^G\longrightarrow \mathcal{D}^G
\end{align}
sending $(X, \alpha)$ to $(F(X), \tilde{\alpha})$, where $\tilde{\alpha}_g\colon F(X)\rightarrow F'_gF(X)$ equals $(\delta_g)_X\circ F(\alpha_g)$ for each $g\in G$. The functor $(F, \delta)^G$ acts on morphisms by $F$. This construction is compatible with the composition of equivariant functors.

Two $G$-equivariant functors $(F, \delta)\colon \mathcal{C}\rightarrow \mathcal{D}$ and $(F', \delta')\colon \mathcal{C}\rightarrow \mathcal{D}$ are \emph{isomorphic} provided that there is a natural isomorphism $\phi \colon F\rightarrow F'$ satisfying
\begin{align}\label{equ:iso-equi}
{F_g'} \phi \circ \delta_g={\delta_g'} \circ \phi F_g,
\end{align}
for each $g\in G$. In this case, the two functors $(F, \delta)^G$ and $(F',\delta')^G$ are naturally isomorphic.

The following fact is standard; compare \cite[Lemma 4.10]{Asa11}.

\begin{lem}\label{lem:G-equ}
Let $(F, \delta)\colon \mathcal{C}\rightarrow \mathcal{D}$ be a $G$-equivariant functor as above. Assume that $F$ is an equivalence of categories. Then the functor $(F,\delta)^G\colon \mathcal{C}^G\rightarrow \mathcal{D}^G$ is also an equivalence of categories.
 \end{lem}

\begin{proof}
Assume that $(X, \alpha)$ and $(Y, \beta)$ are two $G$-equivariant objects in $\mathcal{C}^G$. Then the set ${\rm Hom}_\mathcal{C}(X, Y)$ carries a $G$-action. Associated to $G$-equivariant objects $(F(X), \widetilde{\alpha})$ and $(F(Y), \widetilde{\beta})$  in $\mathcal{D}^G$, the set ${\rm Hom}_\mathcal{D}(F(X), F(Y))$ also carries a $G$-action. We observe that the bijection
$${\rm Hom}_\mathcal{C}(X, Y)\longrightarrow {\rm Hom}_\mathcal{D}(F(X), F(Y)),$$
sending $f$ to $F(f)$,  is compatible with these two $G$-actions. Then we have the induced bijection between the subsets of fixed elements. Applying (\ref{equ:fixed}), this bijection implies  the fully faithfulness of $(F,\delta)^G$.

 It remains to prove the denseness of $(F,\delta)^G$. Take an object $(Z, \gamma)$ in $\mathcal{D}^G$. We have an isomorphism $\theta\colon F(X)\rightarrow Z$ in $\mathcal{D}$ for some object $X$ in $\mathcal{C}$. For each $g\in G$, there is a unique isomorphism $\alpha_g\colon X\rightarrow F_g(X)$ satisfying $$F(\alpha_g)=(\delta_X)^{-1}\circ F'_g(\theta^{-1})\circ \gamma_g\circ \theta.$$
  One verifies that $(X, \alpha)$ is indeed a $G$-equivariant object and that $\theta\colon (F,\delta)^G(X, \alpha)\rightarrow (Z, \gamma)$ is a required isomorphism.
\end{proof}

Two $G$-actions $\{F_g, \varepsilon_{g,h}|\; g,h\in G\}$ and $\{F'_g, \varepsilon'_{g,h}|\; g,h\in G\}$ on $\mathcal{C}$ are \emph{isomorphic} provided that there are natural isomorphisms $\delta_g\colon F_g\rightarrow F'_g$ such that $({\rm Id}_\mathcal{C}, \delta)\colon \mathcal{C}\rightarrow \mathcal{C}$ is a $G$-equivariant functor, or equivalently, the following identities hold
\begin{align}\label{equ:iso-action}
\delta_{gh}\circ \varepsilon_{g, h}=\varepsilon'_{g, h}\circ (F'_g\delta_h\circ \delta_g F_h).
\end{align}
In other words, the following diagram is commutative
\[\xymatrix{
F_gF_h \ar[rr]^-{\varepsilon_{g,h}}\ar[d]_-{F'_g\delta_h\circ \delta_gF_h} && F_{gh}\ar[d]^-{\delta_{gh}}\\
F'_g F'_h \ar[rr]^-{\varepsilon'_{g,h}} && F'_{gh}
}\]
for any $g, h\in G$.

Up to isomorphism, we may adjust the autoequivalences  appearing in a $G$-action by any given natural isomorphisms. More precisely, the following statement is routine.

\begin{lem}\label{lem:change}
Let $\{F_g, \varepsilon_{g,h}|\; g,h\in G\}$ be a given $G$-action on $\mathcal{C}$. Assume that for each $g\in G$, there is an autoequivalence $F'_g$ on $\mathcal{C}$ with a natural isomorphism $\delta_g\colon F_g\rightarrow F'_g$. Then there exists a unique $G$-action $\{F'_g, \varepsilon'_{g,h}|\; g, h\in G\}$ on $\mathcal{C}$ satisfying (\ref{equ:iso-action}).

In particular, the two $G$-actions $\{F_g, \varepsilon_{g,h}|\; g,h\in G\}$  and $\{F'_g, \varepsilon'_{g,h}|\; g, h\in G\}$ are isomorphic.
\end{lem}

\begin{proof}
The uniqueness follows by (\ref{equ:iso-action}), since we have $$\varepsilon'_{g, h}=\delta_{gh}\circ \varepsilon_{g, h}\circ (F'_g\delta_h\circ \delta_gF_h)^{-1}.$$ It is routine to verify that these natural isomorphisms $\varepsilon'_{g,h}$ satisfy (\ref{equ:2-coc}).
\end{proof}

The following fact is standard.

\begin{lem}\label{lem:transport}
Let $F\colon \mathcal{C}\rightarrow \mathcal{D}$ be an equivalence of categories. Assume that $\mathcal{C}$ has a $G$-action $\{F_g, \varepsilon_{g,h}|\; g,h\in G\}$. Then there is a $G$-action $\{F'_g, \varepsilon'_{g,h}|\; g,h\in G\}$ on $\mathcal{D}$ with natural isomorphisms $\delta_g\colon FF_g\rightarrow F'_gF$ such that $(F, \delta)$ is a $G$-equivariant functor. Such a $G$-action on $\mathcal{D}$ is unique up to isomorphism.
\end{lem}

We will say that the $G$-action on $\mathcal{D}$ is \emph{transported} from the given one on $\mathcal{C}$.

\begin{proof}
We take a quasi-inverse $F^{-1}$ of $F$ with unit $a\colon {\rm Id}_\mathcal{C}\rightarrow F^{-1}F$. We may take $F'_g=FF_gF^{-1}$, $\varepsilon'_{g,h}=F\varepsilon_{g, h}F^{-1}\circ FF_ga^{-1}F_hF^{-1}$ and $\delta_g=FF_ga$.

For the uniqueness, we assume that there is another $G$-action $\{F''_g, \varepsilon''_{g, h}|\; g,h\in G\}$ on $\mathcal{D}$ with natural isomorphisms $\delta'_g\colon FF_g\rightarrow F''_g F$. There is a unique isomorphism $\partial_g\colon F'_g\rightarrow F''_g$ satisfying $\delta'_g=\partial_g F\circ \delta_g$; see Lemma \ref{lemA:Nat}. Then the two $G$-actions on $\mathcal{D}$ are isomorphic via the $G$-equivariant functor $({\rm Id}_\mathcal{D}, \partial)$.
\end{proof}

\subsection{Actions on module categories} Let $R$ be a ring with identity. We denote by ${\rm Mod}\mbox{-}R$ the category of right $R$-modules. By ${\rm mod}\mbox{-}R$ and ${\rm proj}\mbox{-}R$, we mean the full subcategories of finitely presented $R$-modules and finitely generated projective $R$-modules, respectively. For an $R$-module $M=M_R$, we usually denote the $R$-action by ``$.$". We denote by  $R=R_R$ the regular right $R$-module.

We will recall that $G$-actions on ${\rm Mod}\mbox{-}R$, which fix $R$, are in a bijection to weak $G$-actions on $R$.

We denote by ${\rm Aut}(R)$ the group of automorphisms on $R$, and by $R^\times$ the multiplicative group formed by invertible elements in $R$.

For an automorphism $\sigma\in {\rm Aut}(R)$ and an $R$-module $M$, the \emph{twisted module} ${^\sigma M}$ is defined as follows: ${^\sigma M}=M$ as an abelian group, and the new $R$-action ``$_\circ$" is given by $m_\circ r=m.\sigma(r)$. This gives rise to an automorphism $${^\sigma(-)}\colon {\rm Mod}\mbox{-}R\longrightarrow {\rm Mod}\mbox{-}R$$
 of categories, called  the \emph{twisting automorphism}. It acts on morphisms by the identity.

 We observe that there is an isomorphism $R\rightarrow {^\sigma R}$ of right $R$-modules, which sends $r$ to $\sigma(r)$. Moreover, for another automorphism $\sigma'\in {\rm Aut}(R)$, we have ${^{\sigma'}(^\sigma M)}={^{(\sigma\sigma')}M}$.

\begin{lem}\label{lem:mod} Keep the notation as above.
\begin{enumerate}
\item Any autoequivalence $F$ on ${\rm Mod}\mbox{-}R$ satisfying $F(R)\simeq R$ is isomorphic to $^\sigma(-)$ for some $\sigma\in {\rm Aut}(R)$.
\item The twisting automorphisms $^\sigma(-)$ and $^{\sigma'}(-)$ are isomorphic if and only if there exists $a\in R^\times$ with $\sigma'(x)=a^{-1}\sigma(x) a$ for all $x\in R$.
\item For a given natural isomorphism $\varepsilon\colon {^\sigma(-)}\rightarrow {^{\sigma'}(-)}$, there is a unique element $a\in R^\times$ satisfying
    \begin{align}
    \varepsilon_M(m)=m.a
    \end{align}
    for any $R$-module $M$ and $m\in M$. Here, the dot ``$.$" denotes the original $R$-action on $M$, not the one on ${^{\sigma}M}$ or ${^{\sigma'}M}$.
    \end{enumerate}
\end{lem}

The same results hold for the categories ${\rm mod}\mbox{-}R$ and ${\rm proj}\mbox{-}R$.

\begin{proof}
These statements are all well known. In (3), we observe that $a=\varepsilon_R(1)$. Moreover, the statement (2) is implied by (3).
\end{proof}

The following notion is standard; compare \cite[Section 1.4]{NV}.

\begin{defn}\label{defn:weakact}
By a \emph{weak $G$-action} on $R$, we mean a pair $(\rho, c)$, where $\rho\colon G\rightarrow {\rm Aut}(R)$ and $c\colon G\times G\rightarrow R^\times$ are maps subject to the conditions
\begin{enumerate}
\item[(WA1)] $\rho(gh)(x)=c(g, h)^{-1}\cdot \rho(g)(\rho(h)(x))\cdot c(g, h)$;
\item[(WA2)] $c(g, h)\cdot c(gh, k)=\rho(g)(c(h, k))\cdot c(g, hk),$
\end{enumerate}
for any $g, h, k\in G$ and $x\in R$. Here, we use the central dot to denote the multiplication in $R$.

Two weak $G$-actions $(\rho, c)$ and $(\rho', c')$ are \emph{isomorphic} provided that there is a map $\delta\colon G\rightarrow R^\times$ satisfying $\rho'(g)(x)=\delta(g)^{-1}\cdot \rho(g)(x)\cdot \delta(g)$ and $c'(g, h)=\delta(g)^{-1}\cdot \rho(g)(\delta(h)^{-1}) \cdot  c(g, h) \cdot \delta(gh)$. \hfill $\square$
\end{defn}

We observe from (WA1) that $\rho(e)(y)=c(e,e)\cdot y\cdot c(e,e)^{-1}$ for all $y\in R$. Moreover, by taking $h=e$ in (WA2), we infer that $c(e, k)=c(e,e)$ and $c(g, e)=\rho(g)(c(e,e))$.

In the literature, the triple  $(R, \rho, c)$ is called a \emph{$G$-crossed system}. The corresponding \emph{crossed product} $R\ast G$ is a ring which is defined as follows: $R\ast G$ is  a free left $R$-module with basis $\{\bar{g}|\; g\in G\}$, and its multiplication is given by
$$(r_1\bar{g})(r_2\bar{h})=(r_1\cdot \rho(g)(r_2)\cdot c(g, h))\overline{gh}.$$
We observe a ring embedding $R\rightarrow R\ast G$ sending $r$ to $(r\cdot c(e, e)^{-1})\bar{e}$. In particular, the identity of $R\ast G$ is $c(e,e)^{-1}\bar{e}$.

By a \emph{$G$-action} on $R$, we mean a weak $G$-action $(\rho, c)$ with $c(g, h)=1$ for all $g, h \in G$; then the map $\rho$ is a group homomorphism. In this case, the crossed product $R\ast G$ is called the \emph{skew group ring}.

For a given weak $G$-action $(\rho, c)$ on $R$, we consider the following natural isomorphism on ${\rm Mod}\mbox{-}R$
\begin{align}\label{equ:c}
c_{g, h}\colon {^{\rho(h^{-1})\rho(g^{-1})}(-)}\longrightarrow {^{\rho((gh)^{-1})}(-)}
\end{align}
such that $(c_{g, h})_M(m)=m.c(h^{-1}, g^{-1})$; compare (WA1) and Lemma \ref{lem:mod}(3). Indeed, this gives rise to a $G$-action $\{{^{\rho(g^{-1})}(-)}, c_{g, h}|\; g, h \in G\}$ on ${\rm Mod}\mbox{-}R$, where the condition (\ref{equ:2-coc}) follows from (WA2).

\begin{prop}\label{prop:act-mod}
There is a bijection from  the set of isoclasses of weak $G$-actions on $R$ to the set of isoclasses  of $G$-actions on ${\rm Mod}\mbox{-}R$
$$\{\mbox{weak } G\mbox{-actions on } R\}{/\simeq} \; \longleftrightarrow\;  \{G\mbox{-actions on {\rm Mod}-}R \mbox{ fixing }R\}{/\simeq},$$
which sends $(\rho, c)$ to $\{{^{\rho(g^{-1})}(-)}, c_{g, h}|\; g, h \in G\}$ on ${\rm Mod}\mbox{-}R$.

 Under this bijection, we have an isomorphism of categories
$$({\rm Mod}\mbox{-}R)^G\stackrel{\sim}\longrightarrow {\rm Mod}\mbox{-}R\ast G.$$
\end{prop}

The bijection holds for  ${\rm mod}\mbox{-}R$ and ${\rm proj}\mbox{-}R$. Moreover,  if the group $G$ is finite,  we have an isomorphism of categories
$$({\rm mod}\mbox{-}R)^G\stackrel{\sim}\longrightarrow {\rm mod}\mbox{-}R\ast G.$$

\begin{proof}
The above map is injective by Lemma \ref{lem:mod}(3). For the surjectivity, take any $G$-action $\{F_g, \varepsilon_{g,h}|\; g, h\in G\}$ on ${\rm Mod}\mbox{-}R$ that fixes $R$. By Lemmas \ref{lem:mod}(1) and \ref{lem:change}, we may assume that $F_g={^{\rho(g^{-1})}(-)}$ for some map $\rho\colon G\rightarrow {\rm Aut}(R)$. By Lemma \ref{lem:mod}(3), the natural isomorphisms $\varepsilon_{g, h}$ give rise to the map $c$ such that $\varepsilon_{g, h}$ coincide with $c_{g, h}$ in (\ref{equ:c}). Moreover, $(\rho, c)$ is a weak $G$-action on $R$, where (\ref{equ:2-coc}) implies (WA2). This proves the surjectivity.

The last isomorphism sends a $G$-equivariant $R$-module $(M, \alpha)$ to the $R\ast G$-module $M$, whose action is given by $m.(r\bar{g})=(\alpha_{g^{-1}})^{-1}(m.r)$. Here, the expression $m.r$ means the action on $M$ by the element $r\in R$. The inverse functor sends an $R\ast G$-module $X$ to $(X, \beta)$, where the underlying $R$-module structure on $X$ is given by $x.r=x.((r\cdot c(e, e)^{-1})\bar{e})$, and the isomorphism $\beta_g\colon X\rightarrow {^{\rho(g^{-1})}X}$ is defined by $\beta_g(x.\overline{g^{-1}})=x$ for each $x\in X$ and $g\in G$.
\end{proof}

\subsection{The stable objects and crossed systems} We will show that $G$-crossed systems arise naturally from $G$-stable objects; compare \cite{DLS}.

Let $\mathcal{C}$ be an additive category with a fixed $G$-action $\{F_g, \varepsilon_{g, h}|\; g,h\in G\}$. Recall that an object $T$ is $G$-stable provided that $T\simeq F_g(T)$ for each $g\in G$. We denote by ${\rm add}\; T$ the full subcategory consisting of direct summands of finite direct sums of copies of $T$. In this case, we obtain the restricted $G$-action on ${\rm add}\; T$.

We take a $G$-stable object $T$ and set $R={\rm End}_\mathcal{C}(T)$ to be its endomorphism ring. Choose
for each $g\in G$ an isomorphism $\alpha_g\colon T\rightarrow F_g(T)$. Then we have a ring automorphism $\rho(g)\in {\rm Aut}(R)$ such that
$$\rho(g)^{-1}(a)=(\alpha_{g^{-1}})^{-1}\circ F_{g^{-1}}(a)\circ \alpha_{g^{-1}}$$
for each $a\in R$. For $g,h \in G$, there is a unique element $c(g, h)\in R^\times$, or equivalently, a unique automorphism $c(g, h)$ of $T$,  satisfying
\begin{align}\label{equ:definec}
F_{(gh)^{-1}}(c(g, h))\circ \alpha_{(gh)^{-1}}=(\varepsilon_{h^{-1},g^{-1}})_T \circ F_{h^{-1}}(\alpha_{g^{-1}})\circ \alpha_{h^{-1}}.
\end{align}
This defines a  weak $G$-action $(\rho, c)$ on $R$. We observe that if we choose another family of isomorphisms $\beta_g\colon T\rightarrow F_g(T)$, then the resulting weak $G$-action on $R$ is isomorphic to $(\rho, c)$. In particular, we have a $G$-crossed system $(R, \rho, c)$.

The well-known functor
$${\rm Hom}_\mathcal{C}(T, -)\colon \mathcal{C}\longrightarrow {\rm Mod}\mbox{-}R$$
restricts to a fully faithful functor ${\rm add}\; T\stackrel{\sim}\longrightarrow {\rm proj}\mbox{-}R$; moreover, it is dense if $\mathcal{C}$ is idempotent complete. Here, we recall that an additive  category $\mathcal{C}$ is \emph{idempotent complete} provided that each idempotent $e\colon X\rightarrow X$ splits, that is, there exist morphisms $u\colon X\rightarrow Z$ and $v\colon Z\rightarrow X$ satisfying $e=v\circ u$ and ${\rm Id}_Z=u\circ v$.

For each object $X\in \mathcal{C}$ and $g\in G$, there is a natural isomorphism of $R$-modules
$$(\phi_g)_X\colon {\rm Hom}_\mathcal{C}(T, F_g(X))\longrightarrow {^{\rho(g^{-1})}{\rm Hom}_\mathcal{C}(T, X)},$$
such that $f=F_g((\phi_g)_X(f))\circ \alpha_g$ for each $f\colon T\rightarrow F_g(X)$.

Let $T$ and the resulting $G$-crossed system $(R, \rho, c)$ be as above. As in Proposition~\ref{prop:act-mod}, we consider the $G$-action $\{{^{\rho(g^{-1})}(-)},c_{g,h}|\; g,h\in G\}$ on ${\rm Mod}\mbox{-}R$ and its restricted $G$-action on ${\rm proj}\mbox{-}R$.

\begin{lem}\label{lem:stable-hom}
Keep the notation as above. Then the data $({\rm Hom}_\mathcal{C}(T, -), (\phi_g)_{g\in G})$ is a $G$-equivariant functor. In particular, if $\mathcal{C}$ is idempotent complete,  we have a $G$-equivariant equivalence
\begin{align}
({\rm Hom}_\mathcal{C}(T, -), \phi)\colon {\rm add}\; T\stackrel{\sim}\longrightarrow {\rm proj}\mbox{-}R.
\end{align}
\end{lem}

\begin{proof}
We just observe that (\ref{equ:definec}) implies (\ref{equ:equi-fun}) for the data  $({\rm Hom}_\mathcal{C}(T, -), (\phi_g)_{g\in G})$.  We omit the details.
\end{proof}

Let $f\colon R\rightarrow S$ be an isomorphism between two  rings. For a weak $G$-action $(\rho, c)$ on $R$, the isomorphism $f$ induces a weak $G$-action $f_*(\rho, c)=(\bar{\rho}, \bar{c})$ on $S$ as follows: $\bar{\rho}(g)=f\circ \rho(g)\circ f^{-1}$ and $\bar{c}=f\circ c$.

\begin{defn}\label{defn:equiv}
Two $G$-crossed systems $(R, \rho, c)$ and $(R', \rho', c')$ are \emph{equivalent} provided that there is an isomorphism $f\colon R\rightarrow R'$ of rings such that the two weak  $G$-actions  $f_*(\rho, c)$ and $(\rho',c') $ on $R'$ are isomorphic in the sense of Definition \ref{defn:weakact}. \hfill $\square$
\end{defn}

We assume that in the situation of Definition \ref{defn:equiv}, there is a map $\delta\colon G\rightarrow R'^\times$ giving the isomorphism between $f_*(\rho, c)=(\bar{\rho}, \bar{c})$ and $(\rho', c')$. For each right $R$-module $X$, we denote by $f_*(X)=X$ the corresponding $R'$-module, where the $R'$-action ``$_\circ$" is given by $x_\circ a'=x.f^{-1}(a')$. This gives rise to an isomorphism $f_*\colon {\rm Mod}\mbox{-}R\rightarrow {\rm Mod}\mbox{-}R'$ of categories. For each $g\in G$, there is a natural isomorphism of $R'$-modules
$$(\delta_g)_X\colon f_*{(^{\rho(g^{-1})}X)}\longrightarrow {^{\rho'(g^{-1})}f_*(X)}, \; x\mapsto x_\circ \delta(g^{-1}).$$
Here, we use $\rho'(g^{-1})(a')=\delta(g^{-1})^{-1} \cdot \bar{\rho}(g^{-1})(a') \cdot \delta(g^{-1})$ to verify this natural isomorphism; see Lemma \ref{lem:mod}(2).

In the following lemma, we consider the $G$-action $\{{^{\rho(g^{-1})}(-)},c_{g,h}|\; g,h\in G\}$ on  ${\rm Mod}\mbox{-}R$ and $G$-action $\{{^{\rho'(g^{-1})}(-)},c'_{g,h}|\; g,h\in G\}$ on ${\rm Mod}\mbox{-}R'$; see Proposition \ref{prop:act-mod}.

\begin{lem}\label{lem:fstar}
Keep the notation as above. Then we have a $G$-equivariant isomorphism of categories
$$(f_*, (\delta_g)_{g\in G})\colon {\rm Mod}\mbox{-}R\longrightarrow {\rm Mod}\mbox{-}R'.$$
\end{lem}

\begin{proof}
The condition (\ref{equ:equi-fun}) for the data $(f_*, (\delta_g)_{g\in G})$ follows directly from $$\bar{c}(h^{-1}, g^{-1})\cdot \delta(h^{-1}g^{-1})=\bar{\rho}(h^{-1})(\delta(g^{-1}))\cdot \delta(h^{-1})\cdot c'(h^{-1}, g^{-1}).$$
 This identity is one of  the properties of the map $\delta\colon G\rightarrow {R'}^\times$, which gives the isomorphism from $(\bar{\rho}, \bar{c})$ to $(\rho', c')$.
\end{proof}

The following observation will be used later.

\begin{lem}\label{lem:equiv-cross}
Let $T$ be a $G$-stable object  in $\mathcal{C}$ with the $G$-crossed system $(R, \rho, c)$ as above. Assume that $(F, \delta)$ is an equivariant autoequivalence on $\mathcal{C}$. Set $T'=F(T)$, which is also $G$-stable. Denote the corresponding $G$-crossed system by $(R', \rho', c')$ with $R'={\rm End}_\mathcal{C}(T')$. Then $(R, \rho, c)$ is equivalent to $(R', \rho', c')$.
\end{lem}

\begin{proof}
We take for each $g\in G$ the isomorphism $\delta_T\circ F(\alpha_g)\colon T'\rightarrow F_g(T')$ to define the weak $G$-action $(\rho', c')$ on $R'$. Recall that $R={\rm End}_\mathcal{C}(T)$ and $R'={\rm End}_\mathcal{C}(T')$.  Then the isomorphism $f\colon R\rightarrow R'$ of rings, defined by $f(a)=F(a)$, yields the required equivalence between the two $G$-crossed systems.
\end{proof}

\section{Cyclic group actions}

In this section, we study the action by a cyclic group. Indeed, actions by a cyclic group are classified by compatible pairs. For module categories, compatible pairs correspond to compatible elements in the outer automorphism groups of the algebras.

\subsection{Compatible pairs}

Let $\mathcal{C}$ be a category,  and let $d\geq 2$. Denote by $C_d=\{e, g, \cdots, g^{d-1} \}$ the cyclic group of order $d$. Then we have $g^ig^j=g^{[i+j]}$ for $0\leq i, j\leq d-1$, where $[i+j]=i+j$ if $i+j\leq d-1$ and $[i+j]=i+j-d$ otherwise.

The following example is taken from \cite[Example 2.2]{Chen16}.

\begin{exm}\label{exm:cyclic}
{\rm Let $F\colon \mathcal{C}\rightarrow \mathcal{C}$ be an autoequivalence with an isomorphism $c\colon F^d\rightarrow {\rm Id}_\mathcal{C}$ satisfying $Fc=cF$.

We construct a $C_d$-action $\{\bar{F}_{g^i}, \bar{\varepsilon}_{g^i, g^j}|\; 0\leq i, j\leq d-1\}$ on $\mathcal{C}$ as follows: $\bar{F}_{g^i}=F^i$, where $F^0={\rm Id}_\mathcal{C}$; the natural transformation $$\bar{\varepsilon}_{g^i, g^j}\colon \bar{F}_{g^i}\bar{F}_{g^j}\longrightarrow \bar{F}_{g^ig^j}=\bar{F}_{g^{[i+j]}}$$
is given by the identity if $i+j<d$, and by $F^{i+j-d}c$ otherwise. We mention that to verify (\ref{equ:2-coc}), one uses the condition $F^ic=cF^i$ for each $i\geq 0$.

The constructed $C_d$-action $\{\bar{F}_{g^i}, \bar{\varepsilon}_{g^i, g^j}|\; 0\leq i, j\leq d-1\}$ is said to be \emph{induced} from the pair $(F, c)$. It is a strict action if and only if $F$ is an automorphism  with $F^d={\rm Id}_\mathcal{C}$ and  $c$ equals the identity.} \hfill $\square$
\end{exm}

Let us consider an arbitrary $C_d$-action $\{F_{g^i}, \varepsilon_{g^i, g^j}|\; 0\leq i,j\leq d-1\}$ on $\mathcal{C}$. Set $F=F_g$. We define a natural isomorphism $\varepsilon^{(i)}\colon F^i\rightarrow F_{g^i}$ for each $i\geq 0$ as follows: $\varepsilon^{(0)}=u^{-1}$, $\varepsilon^{(1)}={\rm Id}_{F}$, $\varepsilon^{(2)}=\varepsilon_{g, g}$ and $\varepsilon^{(i+1)}=\varepsilon_{g^i, g}\circ \varepsilon^{(i)}F$ for each $i\geq 2$.  Consider the following composition
$$c\colon F^d \stackrel{\varepsilon^{(d)}} \longrightarrow F_{g^d}=F_e\stackrel{u}\longrightarrow {\rm Id}_{\mathcal{C}},$$
where $u$ is the unit of the given $C_d$-action.

\begin{lem}\label{lem:cyclic}
Keep the notation as above.   Then the following statements hold.
\begin{enumerate}
\item The natural isomorphism $c\colon F^d\rightarrow {\rm Id}_\mathcal{C}$ satisfies $Fc=cF$.
\item The given $C_d$-action $\{F_{g^i}, \varepsilon_{g^i, g^j}|\; 0\leq i,j\leq d-1\}$  is isomorphic to the $C_d$-action induced by the pair $(F, c)$.
\end{enumerate}
\end{lem}

\begin{proof}
(1) is contained in \cite[Lemma 2.1]{Chen16}, where we use the identity $\varepsilon^{(i+1)}=\varepsilon_{g, g^i}\circ F\varepsilon^{(i)}$.

(2) We use the notation in Example \ref{exm:cyclic}. For each $0\leq i\leq d-1$, we have the isomorphism $\varepsilon^{(i)} \colon \bar{F}_{g^i}=F^i \rightarrow F_{g^i}$. We claim that
$$\varepsilon^{([i+j])}\circ \bar{\varepsilon}_{g^i, g^j}=\varepsilon_{g^i, g^j}\circ (F_{g^i}\varepsilon^{(j)} \circ \varepsilon^{(i)}\bar{F}_{g^j})$$
holds for any $0\leq i, j\leq d-1$. Then we are done with the required isomorphism.

Indeed, the claim follows from Lemma \ref{lemA:twocases}(1) if $i+j\leq d-1$, and from Lemma \ref{lemA:twocases}(2) if $i+j\geq d$.
\end{proof}

The following terminology will be convenient for us.

\begin{defn}\label{defn:comp}
A \emph{compatible pair} $(F, c)$ of order $d$ consists of an endofunctor $F$ on $\mathcal{C}$ and a natural isomorphism $c\colon F^d\rightarrow {\rm Id}_\mathcal{C}$ satisfying $Fc=cF$.

Two compatible pairs $(F, c)$ and $(F', c')$ are defined to be \emph{isomorphic},  provided that there is a natural isomorphism $a\colon F\rightarrow F'$  satisfying $c=c'\circ a^d$. Here, the natural isomorphism $a^i\colon F^i\rightarrow F'^i$ is defined inductively by $a^1=a$ and $a^{i+1}=aF'^{i}\circ Fa^i$. \hfill $\square$
\end{defn}

We have the following classification result: the cyclic group actions are classified by compatible  pairs.

\begin{prop}\label{prop:cyclic-act}
There is a bijection between the sets of isomorphism classes
$$\{\mbox{compatible pairs of order } d\}{/\simeq} \; \longleftrightarrow\;  \{C_d\mbox{-actions on }\mathcal{C}\}{/\simeq},$$
sending a compatible pair to its induced $C_d$-action.
\end{prop}

\begin{proof}
The map is well defined, since isomorphic compatible pairs yield isomorphic actions. The surjectivity follows from Lemma \ref{lem:cyclic}.

For the injectivity,  we take two compatible pairs $(F, c)$ and $(F', c')$. Consider the induced $C_d$-actions $\{\bar{F}_{g^i}, \bar{\varepsilon}_{g^i, g^j}|\; 0\leq i, j\leq d-1\}$ and $\{\bar{F}'_{g^i}, \bar{\varepsilon}'_{g^i, g^j}|\; 0\leq i, j\leq d-1\}$. Assume that they are isomorphic. Then there is an isomorphism $\delta_i\colon \bar{F}_{g^i}=F^i\rightarrow \bar{F}'_{g^i}=F'^i$ for each $0\leq i\leq d-1$,  such that the following identity holds
$$\delta_{[i+j]}\circ \bar{\varepsilon}_{g^i, g^j}= \bar{\varepsilon}'_{g^i, g^j}\circ (F'^i\delta_j\circ \delta_i F^j).$$
Set $a=\delta_1\colon F\rightarrow F'$. We take $i=j=0$ to infer that $\delta_0$ equals the identity. The cases where $i+j\leq d-1$ imply that $\delta_i=a^i$. The case that $i+j=d$ yields that $c=c'\circ a^d$, where we use the fact that $a^d=F'^ia^j\circ a^i F^j$. Hence the two compatible pairs are isomorphic. Then we are done.
\end{proof}

Let $k$ be a field, and let $\mathcal{C}$ be a skeletally small $k$-linear category. Here, being skeletally small means  that the isoclasses of objects form a set. Recall by definition that the \emph{center} $Z(\mathcal{C})$ of the category  $\mathcal{C}$ consists of natural transformations $\lambda\colon {\rm Id}_\mathcal{C}\rightarrow {\rm Id}_\mathcal{C}$. Then $Z(\mathcal{C})$ is a commutative $k$-algebra, whose addition and multiplication are induced by the addition of morphisms and the composition of morphisms in $\mathcal{C}$, respectively.

For a $k$-linear autoequivalence $F\colon \mathcal{C}\rightarrow \mathcal{C}$, we denote by $[F]$ its isoclass. We denote by ${\rm Aut}_k(\mathcal{C})$ the group  formed by isoclasses of $k$-linear autoequivalences on $\mathcal{C}$, whose multiplication is induced by the composition of functors. In the literature, ${\rm Aut}_k(\mathcal{C})$ is called \emph{the group of autoequivalences} on $\mathcal{C}$.

\begin{lem}\label{lem:uniquec}
Let $d\geq 2$. Assume that $Z(\mathcal{C})=k$ and each element in $k$ has a $d$-th root. Let $[F]\in {\rm Aut}_k(\mathcal{C})$ with two isomorphisms $c\colon F^d\rightarrow {\rm Id}_\mathcal{C}$ and $c'\colon F^d\rightarrow {\rm Id}_\mathcal{C}$ satisfying $Fc=cF$ and $Fc'=c'F$ . Then the two compatible pairs $(F, c)$ and $(F, c')$ are isomorphic.
\end{lem}

\begin{proof}
Since $Z(\mathcal{C})=k$, there is a nonzero element $\lambda\in k$ such that $c'=\lambda c$. Assume that $\lambda=\mu^d$. Then the natural isomorphism $\mu F\colon F\rightarrow F$ implies  that the two compatible pairs are isomorphic.
\end{proof}

A $G$-action $\{F_g, \varepsilon_{g, h} |\; g, h\in G\}$ on $\mathcal{C}$ is \emph{$k$-linear}, provided that  each autoequivalence $F_g$ is $k$-linear.  Combining Proposition \ref{prop:cyclic-act} and Lemma \ref{lem:uniquec}, we have the following immediate consequence.

\begin{cor}\label{cor:cyclic-act}
Let $d\geq 2$. Assume that $Z(\mathcal{C})=k$ and each element in $k$ has a $d$-th root. Then there is a bijection between
\begin{enumerate}
\item the set of isoclasses of $k$-linear $C_d$-actions on $\mathcal{C}$, and
\item $\{[F]\in {\rm Aut}_k(\mathcal{C})\; |\;  \mbox{ there exists some compatible pair } (F, c) \mbox{ of order }d\}$,
\end{enumerate}
which sends a $C_d$-action $\{F_{g^i}, \varepsilon_{g^i, g^j}|\; 0\leq i,j\leq d-1\}$  on $\mathcal{C}$ to $[F_g]$. \hfill $\square$
\end{cor}

\subsection{Cyclic group actions on module categories}

Let $A$ be a finite dimensional $k$-algebra. We denote by ${\rm Aut}_k(A)$ the automorphism group of the algebra $A$. We say that a weak $G$-action $(\rho, c)$ on $A$ is \emph{$k$-linear} if $\rho$ takes values in ${\rm Aut}_k(A)$.

Recall that an automorphism $\sigma$ is inner provided that there is an invertible element $a\in A^\times$ satisfying $\sigma(x)=a^{-1}xa$. Inner automorphisms form a normal subgroup ${\rm Inn}_k(A)$ of ${\rm Aut}_k(A)$. The quotient group
$${\rm Out}_k(A)={\rm Aut}_k(A)/{{\rm Inn}_k(A)}$$
is called the \emph{outer automorphism group} of $A$, where the corresponding image of $\sigma\in {\rm Aut}_k(A)$ is denoted by $\bar{\sigma}$.

\begin{defn}
Let $d\geq 2$. An automorphism $\sigma\in {\rm Aut}_k(A)$ is \emph{$d$-compatible}, provided that there exists some element $a\in A^\times$ such that $\sigma(a)=a$ and $\sigma^d(x)=a^{-1}xa$ for all $x\in A$. \hfill $\square$
\end{defn}

For any inner automorphism $\delta$, we observe that $\sigma$ is $d$-compatible if and only if so is $\sigma\delta$. If $\sigma$ is $d$-compatible, we will also call the corresponding element $\bar{\sigma}\in {\rm Out}_k(A)$  \emph{$d$-compatible}. These elements are closed under conjugation.

\begin{lem}\label{lem:modA}
Assume that  the algebra $A$ is basic. Let $F$ be a $k$-linear autoequivalence on ${\rm mod}\mbox{-A}$ and $\sigma\in {\rm Aut}_k(A)$. Then the following statements hold.
\begin{enumerate}
\item We have $F(A)\simeq A$, and thus $F$ is isomorphic to $^{\sigma'}(-)$ for some $\sigma'\in {\rm Aut}_k(A)$.
\item There is an isomorphism  ${\rm Out}_k(A)\rightarrow {\rm Aut}_k({\rm mod}\mbox{-A})$ of groups, sending $\bar{\sigma}'$ to $[{^{\sigma'}(-)}]$, the isoclass of the twisting automorphism  $^{\sigma'}(-)$.
\item The automorphism $\sigma$ is $d$-compatible if and only if there is a compatible pair $({^\sigma(-)}, c)$ of order $d$.
\end{enumerate}
\end{lem}

\begin{proof}
Recall that any equivalence preserves basic projective generators. Then (1) follows from Lemma \ref{lem:mod}(1). Lemma \ref{lem:mod}(2)  implies the statement (2). For (3), we  just observe that a compatible isomorphism $c\colon ^{\sigma^d}(-)\rightarrow {\rm Id}_{{\rm mod}\mbox{-}A}$ yields the required element $a\in A^\times$; see Lemma \ref{lem:mod}(3).
\end{proof}

We denote by $Z(A)$ the center of $A$. It is well known that $Z(A)$ is isomorphic to $Z({\rm mod}\mbox{-}A)$.

\begin{prop}\label{prop:cyclic-act-alg}
Let $A$ be a finite dimensional basic $k$-algebra satisfying $Z(A)=k$, and $d\geq 2$. Assume that each element in $k$ has a $d$-th root. Then the following sets are in one-to-one correspondence to each other:
\begin{enumerate}
\item the set of isoclasses of $k$-linear $C_d$-actions on ${\rm mod}\mbox{-}A$;
\item the set of isoclasses  of $k$-linear weak $C_d$-actions on $A$;
\item the set of $d$-compatible elements in ${\rm Out}_k(A)$.
\end{enumerate}
Moreover, the bijection ``$(2)\Leftrightarrow (3)$" induces a bijection between the set of equivalence classes of $k$-linear $C_d$-crossed systems $(A, \rho, c)$ to the set of conjugacy  classes formed by $d$-compatible elements in ${\rm Out}_k(A)$.
\end{prop}

\begin{proof}
By Lemma \ref{lem:modA}(1), each $G$-action on ${\rm mod}\mbox{-}A$ fixes $A$. By Proposition \ref{prop:act-mod}, we have the correspondence between (1) and (2). By Lemma \ref{lem:modA}(2) and (3), the set of $d$-compatible elements is in a bijection to the set  in Corollary \ref{cor:cyclic-act}(2). Then Corollary \ref{cor:cyclic-act}  yields the correspondence between (1) and (3).

The bijection from (2) to (3) sends $(\rho, c)$ to the canonical image of $\rho(e)$ in ${\rm Out}_k(A)$. Then the final statement is immediate.
\end{proof}

\section{Associated monads and  dual actions}

In this section, we recall the comparison between the category of equivariant objects and module categories over certain monads. The strict action of the character group on the category of equivariant objects is called the dual action. We recall the details of a duality theorem from \cite{El2014}.

\subsection{Monads and adjoint pairs} In this subsection, we recall basic facts on monads and adjoint pairs. The standard reference is \cite[Chapter VI]{McL}.

Let $\mathcal{C}$ be a category. Recall that a \emph{monad} on  $\mathcal{C}$ is a triple $(M, \eta, \mu)$ consisting of an endofunctor $M\colon \mathcal{C}\rightarrow \mathcal{C}$ and two natural transformations, the \emph{unit} $\eta\colon{\rm Id}_\mathcal{C}\rightarrow M$  and the \emph{multiplication} $\mu \colon M^2\rightarrow M$, subject to the relations $\mu \circ M\mu =\mu \circ \mu M$ and $\mu\circ M\eta={\rm Id}_M=\mu\circ \eta M$. We suppress the unit and multiplication when they are understood, and denote the monad $(M, \eta, \mu)$ simply by $M$.

We will recall that each adjoint pair yields a monad.  Assume that $F\colon \mathcal{C}\rightarrow \mathcal{D}$ is a functor, which admits a right adjoint $U\colon \mathcal{D}\rightarrow \mathcal{C}$. We denote by $\eta\colon {\rm Id}_\mathcal{C}\rightarrow UF$ the \emph{unit} and $\epsilon\colon FU\rightarrow {\rm Id}_\mathcal{D}$ the \emph{counit}; they satisfy $\epsilon F\circ F\eta ={\rm Id}_F$ and $U\epsilon\circ \eta U={\rm Id}_U$. We denote this adjoint pair on $\mathcal{C}$ and $\mathcal{D}$ by the quadruple $(F, U; \eta, \epsilon)$. In other words, an adjoint pair really means the relevant quadruple. However, we suppress the unit and counit, when they are clear from the context.

The adjoint pair $(F, U;\eta, \epsilon)$  defines a monad $(M, \eta, \mu)$ on $\mathcal{C}$, where  $M=UF\colon \mathcal{C}\rightarrow \mathcal{C}$ and $\mu=U\epsilon F\colon M^2=UFUF\rightarrow U{\rm Id}_\mathcal{D}F=M$. The resulting monad $(M, \eta, \mu)$ on $\mathcal{C}$ is said to be \emph{defined} by the adjoint pair $(F, U; \eta, \epsilon)$. Indeed, as we will recall,  any monad is defined by a certain adjoint pair; see \cite[VI.2]{McL}.

For a monad $(M, \eta, \mu)$ on $\mathcal{C}$, an $M$-\emph{module} is a pair $(X, \lambda)$ consisting of an object $X$ in $\mathcal{C}$ and a morphism $\lambda \colon M(X)\rightarrow X$ subject to the conditions $\lambda\circ M\lambda =\lambda\circ \mu_X$ and $\lambda\circ \eta_X={\rm Id}_X$;  the object $X$ is said to be the \emph{underlying object} of the module. A morphism  $f\colon (X, \lambda)\rightarrow (X', \lambda')$  between two $M$-modules is a morphism $f\colon  X\rightarrow X'$ in $\mathcal{C}$ satisfying $f\circ \lambda=\lambda'\circ M(f)$.  Then we have the category  $M\mbox{-Mod}_\mathcal{C}$ of $M$-modules and the \emph{forgetful functor} $U_M\colon M\mbox{-Mod}_\mathcal{C}\rightarrow \mathcal{C}$.

We observe that each object $X$ in $\mathcal{C}$ gives rise to an $M$-module $F_M(X)=(M(X), \mu_X)$, the \emph{free $M$-module} generated by $X$. Indeed, this gives rise to the \emph{free module functor}  $F_M\colon \mathcal{C}\rightarrow M\mbox{-Mod}_\mathcal{C}$ sending $X$ to the free module $F_M(X)$, and a morphism $f\colon X\rightarrow Y$ to the morphism $M(f)\colon F_M(X)\rightarrow F_M(Y)$.

We have the adjoint pair $(F_M, U_M; \eta, \epsilon_M)$ on $\mathcal{C}$ and $M\mbox{-Mod}_\mathcal{C}$, where for an $M$-module $(X, \lambda)$, the counit $\epsilon_M$ is given such that $$(\epsilon_M)_{(X, \lambda)}=\lambda\colon F_M U_M(X,\lambda)=(M(X), \mu_X)\longrightarrow (X, \lambda).$$ The unit of the adjoint pair is given by the unit $\eta$ of the monad $M$, where we observe that $M=U_MF_M$. Moreover, the adjoint pair $(F_M, U_M; \eta, \epsilon_M)$ defines the given monad $(M, \eta, \mu)$ on $\mathcal{C}$.

For the given monad $(M, \eta, \mu)$, the above adjoint pair $(F_M, U_M; \eta, \epsilon_M)$ enjoys the following universal property: for any adjoint pair $(F, U; \eta, \epsilon)$ on $\mathcal{C}$ and $\mathcal{D}$ that defines $M$, there is a unique functor
$$K\colon \mathcal{D}\longrightarrow M\mbox{-Mod}_\mathcal{C}$$
satisfying $KF=F_M$ and $U_MK=U$; see \cite[VI.3]{McL}. This unique functor $K$ will be referred as the \emph{comparison functor} associated to the  adjoint pair $(F, U; \eta, \epsilon)$.

Indeed, the comparison functor $K\colon \mathcal{D}\rightarrow M\mbox{-Mod}_\mathcal{C}$ is described as follows:
 \begin{align}\label{equ:K}
 K(D)=(U(D), U\epsilon_D), \quad K(f)=U(f)
\end{align}
for an object $D$ and a morphism $f$ in $\mathcal{D}$. Here, we observe that $M=UF$ and that $(U(D), U\epsilon_D)$ is an $M$-module.

An \emph{isomorphism} $\iota\colon (M_1, \eta_1, \mu_1)\rightarrow (M_2, \eta_2, \mu_2)$ between two monads on $\mathcal{C}$ is a natural isomorphism $\iota\colon M_1\rightarrow M_2$  of functors satisfying $\iota\circ \eta_1=\eta_2$ and $\iota\circ \mu_1=\mu_2\circ (\iota M_2\circ M_1\iota)$. This isomorphism yields an isomorphism
\begin{align}
\iota^*\colon M_2\mbox{-Mod}_\mathcal{C}\longrightarrow M_1\mbox{-Mod}_\mathcal{C}
\end{align}
 between module categories, which sends an $M_2$-module $(X,\lambda)$ to the $M_1$-module $(X, \lambda\circ \iota_X)$, and acts on morphisms by the identity.

\subsection{The associated monads}

In what follows, we assume that $G$ is a finite group and that $\mathcal{C}$ is an additive category. We fix a $G$-action $\{F_g, \varepsilon_{g, h}|\; g, h\in G\}$ on $\mathcal{C}$. Then the category $\mathcal{C}^G$ of equivariant objects is additive, and the forgetful functor $U\colon \mathcal{C}^G\rightarrow \mathcal{C}$ is also additive.

  We recall that the \emph{induction functor} ${\rm Ind}\colon \mathcal{C}\rightarrow \mathcal{C}^G$ is defined as follows: for an object $X$, set $${\rm Ind}(X)=(\bigoplus_{h\in G} F_h(X), \varepsilon(X)),$$
  where for each $g\in G$, the isomorphism $\varepsilon(X)_g \colon \bigoplus_{h\in G} F_h(X)\rightarrow F_g(\bigoplus_{h\in G} F_h(X))$ is diagonally induced by the isomorphism $(\varepsilon_{g, g^{-1}h})_X^{-1}\colon F_h(X)\rightarrow F_gF_{g^{-1}h}(X)$. Here, to verify that ${\rm Ind}(X)$ is indeed an equivariant object, we need the $2$-cocycle condition (\ref{equ:2-coc}). The induction functor sends a morphism $\theta\colon X\rightarrow Y$ to ${\rm Ind}(\theta)=\bigoplus_{h\in G} F_h(\theta)\colon {\rm Ind}(X)\rightarrow {\rm Ind}(Y)$.

For an object $X$ in $\mathcal{C}$ and an object $(Y, \beta)$ in $\mathcal{C}^G$,
  a morphism ${\rm Ind}(X)\rightarrow (Y, \beta)$ is of the form $\sum_{h\in G}\theta_h\colon \bigoplus_{h\in G} F_h(X)\rightarrow Y$
  satisfying $F_g(\theta_h)=\beta_g\circ \theta_{gh}\circ (\varepsilon_{g, h})_X$ for any $g, h\in G$.  The  adjoint pair $({\rm Ind}, U; \eta, \epsilon)$ is given by the
  following natural isomorphism
\begin{align}\label{equ:adj1}
 {\rm Hom}_{\mathcal{C}^G} ({\rm Ind}(X), (Y, \beta))\stackrel{\sim}\longrightarrow {\rm Hom}_\mathcal{C}(X, U(Y, \beta))
 \end{align}
 sending the morphism $\sum_{h\in G} \theta_h$ to $\theta_e \circ u_X^{-1}\colon X\rightarrow Y$. The corresponding unit $\eta\colon {\rm Id}_\mathcal{C}\rightarrow U{\rm Ind}$ is
 given such that $\eta_X=(u_X^{-1}, 0, \cdots, 0)^t$, where `$t$' denotes the transpose;
 the counit $\epsilon\colon {\rm Ind}U\rightarrow {\rm Id}_{\mathcal{C}^G}$ is given such that
  $\epsilon_{(Y, \beta)}=\sum_{h\in G} \beta_h^{-1}$.

The monad $M=(U{\rm Ind}, \eta, \mu)$ defined by the  adjoint pair $({\rm Ind}, U; \eta, \epsilon)$ is computed as follows. The endofunctor $M=U{\rm Ind}\colon \mathcal{C}\rightarrow \mathcal{C}$ is given by $M(X)=\bigoplus_{h\in G} F_h(X)$ and $M(f)=\bigoplus_{h\in G} F_h(f)$ for an object $X$ and a morphism $f$ in $\mathcal{C}$. The multiplication $\mu\colon M^2\rightarrow M$ is given by
 $$\mu_X=U\epsilon_{{\rm Ind}(X)}\colon M^2(X)=\bigoplus_{h, g\in G} F_hF_g(X) \longrightarrow M(X)=\bigoplus_{h'\in G} F_{h'}(X)$$
 such that the corresponding entry
  $F_hF_g(X)\rightarrow F_{h'}(X)$ equals $\delta_{hg, h'}(\varepsilon_{h, g})_X$. Here, $\delta$ is the Kronecker symbol.

   Applying (\ref{equ:K}), we compute  that  the associated comparison functor
   $$K\colon \mathcal{C}^G\longrightarrow M\mbox{-Mod}_\mathcal{C}$$
    sends a $G$-equivariant object $(X, \alpha)$ to the $M$-module $(X, \lambda)$, where the module action is given by  $\lambda=\sum_{h\in G}(\alpha_h)^{-1}\colon M(X)=\bigoplus_{h\in G}F_h(X)\rightarrow X$. The functor $K$ acts on morphisms by the identity.

  The following result is standard; see  \cite[Proposition~3.11(2)]{El2014} and \cite[Proposition~3.1]{CCZ}.

  \begin{lem}\label{lem:K}
  Keep the notation as above. Then the comparison functor $K\colon \mathcal{C}^G\rightarrow M\mbox{-{\rm Mod}}_\mathcal{C}$  is an isomorphism of categories. \hfill $\square$
  \end{lem}

For an object $(Y, \beta)\in \mathcal{C}^G$ and an object $X\in \mathcal{C}$, any morphism $(Y, \beta)\rightarrow {\rm Ind}(X)$ is of the form $(\theta_h)_{h\in G}\colon Y\rightarrow \bigoplus_{h\in G}F_h(X)$ satisfying $\theta_{gh}=(\varepsilon_{g, h})_X\circ F_g(\theta_h)\circ \beta_g$ for each $g, h\in G$. The adjoint pair $(U, {\rm Ind}; \eta', \epsilon')$ is given by the following natural isomorphism
\begin{align}\label{equ:adj2}
{\rm Hom}_\mathcal{C}(U(Y, \beta), X)\stackrel{\sim}\longrightarrow {\rm Hom}_{\mathcal{C}^G}((Y, \beta), {\rm Ind}(X)),\end{align}
which sends a morphism $f\colon Y\rightarrow X$ to $(F_h(f)\circ \beta_h)_{h\in G}\colon (Y, \beta)\rightarrow {\rm Ind}(X)$. The unit $\eta'\colon {\rm Id}_{\mathcal{C}^G}\rightarrow {\rm Ind}U$ is given such that $\eta'_{(Y, \beta)}=(\beta_h)_{h\in G}$. The counit $\epsilon'\colon U{\rm Ind}\rightarrow {\rm Id}_\mathcal{C}$ is given such that $\epsilon'_{X}=(u_X, 0, \cdots, 0)\colon \bigoplus_{h\in G}F_h(X)\rightarrow X$.

The monad $N=({\rm Ind}U, \eta', \mu')$ defined by the adjoint pair $(U, {\rm Ind}; \eta', \epsilon')$ is computed as follows. The endofunctor $N={\rm Ind}U\colon \mathcal{C}^G\rightarrow \mathcal{C}^G$ is given by $N(Y, \beta)={\rm Ind}(Y)$ and $N(f)={\rm Ind}(f)$ for any object $(Y, \beta)$ and morphism $f$ in $\mathcal{C}^G$. The multiplication $\mu'\colon N^2\rightarrow N$ is given as follows
$$\mu'_{(Y, \beta)}\colon N^2(Y, \beta)=(\bigoplus_{g, h\in G} F_gF_h(Y), \varepsilon({\bigoplus_{h\in G}F_h(Y)}))\longrightarrow N(Y, \beta)=(\bigoplus_{h'\in G}F_{h'}(Y), \varepsilon(Y)),$$
where $\mu'_{(Y, \beta)}={\rm Ind}\epsilon'_{U(Y, \beta)}={\rm Ind}\epsilon'_Y$. Hence,  the corresponding entry $F_gF_h(Y)\rightarrow F_{h'}(Y)$ equals  $\delta_{g, h'}\delta_{h, e}F_g(u_Y)$, where $u\colon F_e\rightarrow {\rm Id}_\mathcal{C}$ is the unit of the action.

Applying (\ref{equ:K}), we infer that the associated comparison functor
$$K'\colon \mathcal{C}\longrightarrow N\mbox{-Mod}_{\mathcal{C}^G}$$
is given by $K'(X)=({\rm Ind}(X), {\rm Ind}\epsilon'_X)$ and $K'(f)={\rm Ind}(f)$ for any object $X$ and morphism $f$ in $\mathcal{C}$.

\begin{lem}\label{lem:K'}
Keep the notation as above. Assume that $\mathcal{C}$ is idempotent complete. Then the comparison functor $K'\colon \mathcal{C}\rightarrow N\mbox{-{\rm Mod}}_{\mathcal{C}^G}$ is an equivalence of categories.
\end{lem}

\begin{proof}
We observe that the counit $\epsilon'$ is a split epimorphism. Then the induction functor ${\rm Ind}$ is separable; for example, see \cite[Lemma 2.2]{Chen15}. Then the statement follows from \cite[Corollary 3.6]{Chen15}; compare \cite[Theorem 4.4]{DGNO} and \cite[Proposition 3.11(4)]{El2014}.
\end{proof}

We mention that the adjoint pairs (\ref{equ:adj1}) and (\ref{equ:adj2}) may be found in \cite[Lemma~4.6(ii)]{DGNO}. The monad $M$ on $\mathcal{C}$ and the monad $N$ on $\mathcal{C}^G$ are said to be \emph{associated} to the given $G$-action.

The following observation will be used later.  Let $(F, \delta)\colon \mathcal{C}\rightarrow \mathcal{C}$ be a $G$-equivariant endofunctor, and  $(F, \delta)^G\colon \mathcal{C}^G\rightarrow \mathcal{C}^G$ be the induced endofunctor on $\mathcal{C}^G$; see (\ref{equ:fun-equi}). For each object $X$, there is an isomorphism
{\footnotesize \begin{align*}
\xi_X = \bigoplus_{h\in G} (\delta_g)_X\colon (F,\delta)^G {\rm Ind}(X)=(\bigoplus_{h\in G}FF_h(X), \widetilde{\varepsilon(X)}) \longrightarrow {\rm Ind}F(X)=(\bigoplus_{h\in G} F_hF(X), \varepsilon(FX)).
\end{align*}}
Here, to see that $\xi_X$ is a morphism in $\mathcal{C}^G$, we use the following direct consequence of (\ref{equ:equi-fun})
$$(\varepsilon_{g, g^{-1}h}F)^{-1}\circ \delta_h=F_g\delta_{g^{-1}h}\circ (\delta_gF_{g^{-1}h}\circ (F\varepsilon_{g, g^{-1}h})^{-1}).$$
Indeed, this gives rise to a natural isomorphism of functors
\begin{align}\label{equ:xi}
\xi \colon (F, \delta)^G {\rm Ind} \longrightarrow {\rm Ind}F.
\end{align}

We say that a natural number $n$ is \emph{invertible} in the category $\mathcal{C}$ provided that for each morphism $f\colon X\rightarrow Y$, there is a unique morphism $g\colon X\rightarrow Y$ satisfying $f=ng$. This unique morphism is denoted by $\frac{1}{n}f$. In case that $\mathcal{C}$ is skeletally small, $n$ is invertible in $\mathcal{C}$ if and only if $n$ is invertible in its center $Z(\mathcal{C})$. We denote by $|G|$ the order of the finite group $G$.

The following observation is well known.

\begin{lem}\label{lem:separable}
Assume that $|G|$ is invertible in $\mathcal{C}$. Then the counit $\epsilon\colon {\rm Ind} U\rightarrow {\rm Id}_{\mathcal{C}^G}$ is a split epimorphism and the unit $\eta'\colon {\rm Id}_{\mathcal{C}^G}\rightarrow {\rm Ind} U$ is a split monomorphism.
\end{lem}

\begin{proof}
The section of $\epsilon_{(Y, \beta)}$ is $\frac{1}{|G|}(\beta_h)_{h\in G}\colon (Y, \beta)\rightarrow {\rm Ind}(Y)$, and the retraction of $\eta'_{(Y, \beta)}$ is $\frac{1}{|G|}\sum_{h\in G}(\beta_h)^{-1}\colon {\rm Ind}(Y)\rightarrow (Y, \beta)$.
\end{proof}

\subsection{The dual action and the double-dual action} Let $k$ be a field. In this subsection, we assume that the additive category $\mathcal{C}$ is $k$-linear. Let $G$ be a finite group. We fix a $k$-linear $G$-action $\{F_g, \varepsilon_{g, h} |\; g, h\in G\}$ on $\mathcal{C}$. In this case, the category $\mathcal{C}^G$ of equivariant objects is naturally $k$-linear.

We denote by $\widehat{G}={\rm Hom}(G, k^*)$ the \emph{character group} of $G$. For a character $\chi$ on $G$ and an object $(X, \alpha)\in \mathcal{C}^G$, we have another equivariant object $(X, \chi\otimes \alpha)$, where for each $g\in G$, the isomorphism $(\chi\otimes\alpha)_g\colon X\rightarrow F_g(X)$ equals $\chi(g^{-1})\alpha_g$. Set $F_\chi(X, \alpha)=(X, \chi\otimes \alpha)$. This yields an automorphism
\begin{align}\label{equ:defn-Fchi}
F_\chi \colon \mathcal{C}^G\longrightarrow \mathcal{C}^G,
\end{align}
which acts on morphisms by the identity.  We observe that $F_{\chi}F_{\chi'}=F_{\chi \chi'}$ for any $\chi,\chi'\in \widehat{G}$. In other words, we have a strict $\widehat{G}$-action on $\mathcal{C}^G$, which is $k$-linear.

The following  terminology will be justified by Theorem \ref{thm:duality}.

\begin{defn}\label{defn:dual-act}
For the given $G$-action $\{F_g, \varepsilon_{g, h} |\; g, h\in G\}$ on $\mathcal{C}$, we call the above strict $\widehat{G}$-action on $\mathcal{C}^G$ the \emph{dual action} by $\widehat{G}$. \hfill $\square$
\end{defn}

We consider the category $(\mathcal{C}^G)^{\widehat{G}}$ of $\widehat{G}$-equivariant objects in $\mathcal{C}^G$, with respect to the dual $\widehat{G}$-action. For an object $X\in \mathcal{C}$ and a character $\chi \in \widehat{G}$, we recall that
$${\rm Ind}(X)=(\bigoplus_{h\in G}F_h(X), \varepsilon(X)) \mbox{  and  }  F_\chi{\rm Ind}(X)=(\bigoplus_{h\in G}F_h(X), \chi\otimes\varepsilon(X)).$$
Therefore, we have  a canonical isomorphism in $\mathcal{C}^G$
$${\rm can}(X)_\chi=\bigoplus_{h\in G}\chi(h)\colon {\rm Ind}(X)\longrightarrow F_\chi{\rm Ind}(X).$$
This gives rise to a $\widehat{G}$-equivariant object $({\rm Ind}(X), {\rm can}(X))$ in $\mathcal{C}^G$. Furthermore, we have a well-defined functor
\begin{align*}
\Theta \colon \mathcal{C}\longrightarrow (\mathcal{C}^G)^{\widehat{G}}
\end{align*}
such that $\Theta(X)=({\rm Ind}(X), {\rm can}(X))$ and $\Theta(f)={\rm Ind}(f)$ for any morphism $f$ in $\mathcal{C}$.

There is a canonical evaluation homomorphism
\begin{align}\label{equ:ev}
{\rm ev}\colon G\longrightarrow \widehat{\widehat{G}}={\rm Hom}(\widehat{G}, k^*)
\end{align}
 given by ${\rm ev}(g)(\chi)=\chi(g)$. Then we obtain a strict $G$-action on the category $(\mathcal{C}^G)^{\widehat{G}}$, which is called the \emph{double-dual action} by $G$. More precisely, each $g\in G$ gives rise to an automorphism
 \begin{align}\label{equ:doubleFg}
 \widehat{F}_g\colon (\mathcal{C}^G)^{\widehat{G}}\longrightarrow (\mathcal{C}^G)^{\widehat{G}},
 \end{align}
 which sends $(A, \alpha)$ to $(A, {\rm ev}(g)\otimes \alpha)$ and acts on morphisms by the identity. Here, $A$ is an object in $\mathcal{C}^G$, and $\alpha_\chi\colon A\rightarrow F_\chi(A)$ is an isomorphism for each $\chi\in \widehat{G}$. We have $({\rm ev}(g)\otimes\alpha)_\chi=\chi(g^{-1})\alpha_\chi$.

For an object $X\in \mathcal{C}$ and $g\in G$, we have that
\begin{center}$\Theta F_g(X)=({\rm Ind}(F_gX), {\rm can}(F_gX))$ and $\widehat{F}_g\Theta(X)=({\rm Ind}(X), {\rm ev}(g)\otimes {\rm can}(X))$.\end{center}
 Recall that
 $$ {\rm Ind}F_g(X)=(\bigoplus_{h\in G}F_hF_g(X), \varepsilon(F_gX)) \mbox{  and  } {\rm Ind}(X)=(\bigoplus_{h\in G}F_{h}(X),  \varepsilon(X)).$$
Then we have   an isomorphism in $\mathcal{C}^G$
$$(\partial_g)_X=\bigoplus_{h\in G} (\varepsilon_{g, h})_X \colon {\rm Ind}F_g(X)\longrightarrow {\rm Ind}(X).$$
 Moreover, this isomorphism gives rise to an isomorphism
$$(\partial_g)_X\colon \Theta F_g(X)\longrightarrow \widehat{F}_g\Theta(X)$$ in $(\mathcal{C}^G)^{\widehat{G}}$. Furthermore, this yields a natural isomorphism of functors
$$\partial_g\colon \Theta F_g\longrightarrow \widehat{F}_g\Theta.$$

\begin{lem}\label{lem:double-dual}
Consider the given $G$-action on $\mathcal{C}$ and the double-dual $G$-action on $(\mathcal{C}^G)^{\widehat{G}}$. Then  $$(\Theta, (\partial_g)_{g\in G}) \colon \mathcal{C}\longrightarrow (\mathcal{C}^G)^{\widehat{G}}
$$
is a $G$-equivariant functor  with respect to these $G$-actions.
\end{lem}

\begin{proof}
Since the $G$-action on $(\mathcal{C}^G)^{\widehat{G}}$ is strict, it suffices to prove that $\partial_{gh}\circ \Theta\varepsilon_{g, h}=\widehat{F}_g\partial_h\circ \partial_gF_h$. Indeed, it follows directly from (\ref{equ:2-coc}).
\end{proof}

\subsection{The duality theorem} In this subsection, let  $G$ be a finite abelian group. We assume that $G$ \emph{splits} over $k$, that is, the group algebra $kG$ is isomorphic to a product of $k$. In particular, the characteristic of $k$ does not divide the order $|G|$ of $G$. In this case, the character group $\widehat{G}$ is isomorphic to $G$ and the evaluation homomorphism (\ref{equ:ev}) is an isomorphism. In particular, we are in the situation, provided that the field $k$ is algebraically closed of characteristic zero.

The following duality theorem is essentially due to \cite[Theorem 4.2]{El2014}, which may be deduced from the general result \cite[Theorem 4.4]{DGNO}; compare \cite[Subsection 5.1]{RR}. We make the equivalence explicit for our purpose.

\begin{thm}\label{thm:duality}
Keep the notation as above. Let $G$ be a finite abelian group, which splits over $k$. Assume that $\mathcal{C}$ is idempotent complete. Then the above functor
$$(\Theta, (\partial_g)_{g\in G})\colon \mathcal{C}\longrightarrow (\mathcal{C}^G)^{\widehat{G}}$$
is a $G$-equivariant equivalence.
\end{thm}

For the proof, we recall from Lemma \ref{lem:K'} that the monad $N$ on $\mathcal{C}^G$ is associated to the given $G$-action. We denote by $\widehat{M}$ the monad on $\mathcal{C}^G$, which is associated to the dual $\widehat{G}$-action on $\mathcal{C}^G$; that is, $\widehat{M}$ is the monad in Lemma \ref{lem:K} applied to the dual $\widehat{G}$-action.

The following observation may be deduced from the dual of  \cite[Proposition~4.1]{El2014}.

\begin{lem}
Keep the same assumptions as above. Then the two monads $N$ and $\widehat{M}$ on $\mathcal{C}^G$ are isomorphic.
\end{lem}

\begin{proof}
For an object $(X, \alpha)\in \mathcal{C}^G$, we have $N(X, \alpha)={\rm Ind}(X)=(\bigoplus_{h\in G}F_h(X), \varepsilon(X))$ and $\widehat{M}(X, \alpha)=\bigoplus_{\chi\in \widehat{G}} F_\chi(X, \alpha)=\bigoplus_{\chi\in \widehat{G}} (X, \chi\otimes \alpha)$.

 We construct an isomorphism
\begin{align}\label{equ:iso-monad}
f\colon \widehat{M}(X, \alpha)\longrightarrow N(X, \alpha)
 \end{align}
 as follows. The restriction of $f$ to $F_\chi(X, \alpha)$ is given by
 $$(\chi(h^{-1})\alpha_h)_{h\in G}\colon X\longrightarrow \bigoplus_{h\in G}F_h(X);$$
 compare the adjunction (\ref{equ:adj2}). The $\chi$-th component of $f^{-1}$  is given by $$\frac{1}{|G|}\sum_{h\in G}\chi(h)\alpha_h^{-1}\colon \bigoplus_{h\in G}F_h(X)\longrightarrow X.$$
  Here, we use the well-known orthogonality of characters.

The isomorphism $f$ is natural in $(X, \alpha)$. Then the two endofunctors $\widehat{M}$ and $N$ are isomorphic. Using the explicit calculation in Subsection 4.2, it is direct to verify that $f$ is compatible with the units and the multiplications. In other words, we have the required isomorphism of monads.
\end{proof}

\noindent \emph{Proof of Theorem \ref{thm:duality}.}\quad In view of Lemma \ref{lem:double-dual}, it suffices to show that $\Theta$ is an equivalence. We  observe that $\Theta$ equals the composition of the following  sequence of equivalences, which are all explicitly given
$$\mathcal{C}\stackrel{K'}\longrightarrow N\mbox{-Mod}_{\mathcal{C}^G}\stackrel{f^*}\longrightarrow \widehat{M}\mbox{-Mod}_{\mathcal{C}^G}\stackrel{\widehat{K}^{-1}}\longrightarrow (\mathcal{C}^G)^{\widehat{G}}.$$
Here, the equivalence $K'$ is given in Lemma \ref{lem:K'} and the isomorphism $\widehat{K}\colon  (\mathcal{C}^G)^{\widehat{G}} \rightarrow \widehat{M}\mbox{-Mod}_{\mathcal{C}^G}$ is given as in Lemma \ref{equ:K}, which is applied to  the dual $\widehat{G}$-action on $\mathcal{C}^G$. The isomorphism $f\colon \widehat{M}\rightarrow N$ of monads is given in (\ref{equ:iso-monad}). \hfill $\square$

\vskip 5pt

The following example will be used later.

\begin{exm}\label{exm:equi-fun}
{\rm  Let $G$ be a finite group, which might be non-abelian. Let $\mathcal{C}$ be an additive category with a $G$-action $\{F_g, \varepsilon_{g, h}|\; g, h\in G\}$.

Assume that $a\in G$ is a central element. For each $g\in G$, we have a natural isomorphism $$(c_a)_g=(\varepsilon_{g, a})^{-1}\circ \varepsilon_{a, g}\colon F_aF_g\longrightarrow F_gF_a.$$
Here, we use the fact that $ag=ga$. By applying (\ref{equ:2-coc}) three times, we infer (\ref{equ:equi-fun}) with $F=F_a$ and $\varepsilon'=\varepsilon$. In other words, $(F_a, c_a)=(F_a, ((c_a)_g)_{g \in G})$ is a $G$-equivariant endofunctor on $\mathcal{C}$.

We claim that the corresponding endofunctor $(F_a, c_a)^G$ on $\mathcal{C}^G$ is isomorphic to the identity functor. More precisely, we have a natural isomorphism
$$\psi\colon {\rm Id}_{\mathcal{C}^G}\longrightarrow (F_a,c_a)^G$$
such that $\psi_{(X, \alpha)}=\alpha_a\colon (X, \alpha) \rightarrow (F_a(X), \widetilde{\alpha})$.  Here, to verify that $\alpha_a$ is indeed a morphism in $\mathcal{C}^G$, we use the identity $(\varepsilon_{a, g})_X\circ F_a(\alpha_g)\circ \alpha_a=(\varepsilon_{g, a})_X\circ F_g(\alpha_a)\circ \alpha_g$, since both equal $\alpha_{ag}=\alpha_{ga}$ by (\ref{equ:rel}).

We assume further that $\mathcal{C}$ is $k$-linear and that the given $G$-action is $k$-linear. Recall that the dual $\widehat{G}$-action on $\mathcal{C}^G$ sends $\chi$ to the automorphism $F_\chi$. We observe that $(F_a,c_a)^G$ becomes  a strictly $\widehat{G}$-equivariant functor.  We view the identity functor ${\rm Id}_{\mathcal{C}^G}$ also as a strictly $\widehat{G}$-equivariant functor.

In general, the above isomorphism $\psi$ is not an isomorphism between $\widehat{G}$-equivariant functors. Indeed, $F_\chi \psi\neq \psi F_\chi$ provided that $\chi(a)\neq 1$.} \hfill $\square$
\end{exm}

\section{An isomorphism between groups of equivariant autoequivalences}

In this section, we prove that for an action by a finite abelian group, the equivariantization yields an isomorphism between the two groups of equivariant autoequivalences. This isomorphism induces a bijection between the orbit sets of isoclasses of stable objects.

Throughout, we work over a fixed field $k$. So we omit the subindex $k$, and require that all functors are $k$-linear. We assume that $\mathcal{C}$ is a skeletally small $k$-linear additive category.

\subsection{Groups of equivariant autoequivalences}

By ${\rm Aut}(\mathcal{C})={\rm Aut}_k(\mathcal{C})$, we denote the group of isoclasses of $k$-linear autoequivalences on $\mathcal{C}$. For a $k$-linear autoequivalence $F$, we denote by $[F]$ its isoclass. However, by abuse of notation, we sometimes view $F$ as an element in ${\rm Aut}(\mathcal{C})$.

The following fact is standard.

\begin{lem}\label{lem:auto}
Let $\mathcal{C}$ and $\mathcal{D}$ be $k$-linear categories with a given  $k$-linear equivalence $\Theta\colon \mathcal{C}\rightarrow \mathcal{D}$. Then there is a unique isomorphism
$$\Theta_*\colon{\rm Aut}(\mathcal{C})\longrightarrow {\rm Aut}(\mathcal{D})$$
of groups such that for each $k$-linear autoequivalence $F$ on $\mathcal{C}$, there is a $k$-linear autoequivalence $\Theta_*(F)$ on $\mathcal{D}$ with an  isomorphism  $\Theta F\rightarrow \Theta_*(F)\Theta$. \hfill $\square$
\end{lem}

Let $G$ be a finite group with a $k$-linear $G$-action  $\{F_g, \varepsilon_{g, h}|\; g, h\in G\}$ on $\mathcal{C}$. The subgroup
$${\rm Aut}^G(\mathcal{C})=\{[F]\in {\rm Aut}(\mathcal{C})\; |\; [FF_g]=[F_gF] \mbox{ for each }g\in G\}$$
 of ${\rm Aut}(\mathcal{C})$ is of interest. However, it does not behave well. Instead, we will consider the group of equivariant autoequivalences.

Recall that a $G$-equivariant endofunctor on $\mathcal{C}$ means the data $(F, \delta)=(F, (\delta_g)_{g\in G})$, where $F\colon \mathcal{C}\rightarrow \mathcal{C}$ is an endofunctor and the natural isomorphisms $\delta_g\colon FF_g\rightarrow F_gF$ satisfy (\ref{equ:equi-fun}). The composition of two equivariant endofunctors $(F, \delta)$ and $(F', \delta')$ is defined to be $(FF', (\delta_gF'\circ F\delta'_g)_{g\in G})$.

 The isoclass  of the $G$-equivariant endofunctor $(F, \delta)$ will be denoted by $[F, \delta]$.  We denote by ${\rm Aut}(\mathcal{C}; G)$ the isoclass group of $k$-linear $G$-equivariant autoequivalences on $\mathcal{C}$, whose multiplication is induced by the composition of equivariant functors. We call ${\rm Aut}(\mathcal{C}; G)$ the \emph{group of $G$-equivariant autoequivalences} on $\mathcal{C}$.

We have the following group homomorphism
\begin{align}\label{equ:forg}
\phi\colon {\rm Aut}(\mathcal{C}; G)\longrightarrow {\rm Aut}^G(\mathcal{C}), \; [F, \delta]\mapsto [F],
\end{align}
called \emph{forgetful homomorphism} associated to the given $G$-action. In general, it is neither injective nor surjective.

We observe the following equivariant version of Lemma \ref{lem:auto}.  Let $\mathcal{C}$ and $\mathcal{D}$ be two $k$-linear categories with $k$-linear $G$-actions $\{F_g, \varepsilon_{g, h}|\; g, h\in G\}$ and $\{F'_g, \varepsilon'_{g, h}|\; g,h\in G\}$, respectively. Assume that we are given a $k$-linear $G$-equivariant equivalence $(\Theta, \partial)\colon \mathcal{C}\rightarrow \mathcal{D}$. We have that $\Theta_*([F_g])=[F'_g]$ by the isomorphism $\partial_g$. Hence, we have a restricted isomorphism $\Theta_*\colon {\rm Aut}^G(\mathcal{C})\rightarrow {\rm Aut}^G(\mathcal{D})$.

 Take any $G$-equivariant $k$-linear autoequivalence $(F, \delta)$ on $\mathcal{C}$. Then there is a $k$-linear autoequivalence $\Theta_*(F)$ on $\mathcal{D}$ with an isomorphism $\gamma\colon \Theta F\rightarrow \Theta_*(F) \Theta$. There are unique natural isomorphisms $\delta'_g\colon \Theta_*(F)F'_g\rightarrow F'_g \Theta_*(F)$ satisfying the following identity
 $$ \delta'_g\Theta \circ \Theta_*(F)\partial_g \circ \gamma F_g = F'_g\gamma \circ \partial_g F\circ \Theta \delta_g.$$
 Moreover, $(\Theta_*(F), (\delta'_g)_{g\in G})$ is a $G$-equivariant endofunctor on $\mathcal{D}$ and $\gamma$ yields an isomorphism between equivariant functors
$$\gamma\colon (\Theta, \partial)(F, \delta)\longrightarrow (\Theta_*(F), \delta')(\Theta, \partial).$$
We put $(\Theta, \partial)_*(F, \delta)=(\Theta_*(F), \delta')$. This gives rise to the upper row isomorphism of the following commutative diagram
\begin{align}\label{equ:commdiag}
\xymatrix{
{\rm Aut}(\mathcal{C}; G) \ar[d]_-{\phi} \ar[rr]^-{(\Theta, \partial)_*} && {\rm Aut}(\mathcal{D}; G) \ar[d]^-{\phi}\\
{\rm Aut}^G(\mathcal{C}) \ar[rr]^-{\Theta_*} && {\rm Aut}^G(\mathcal{D}),
}\end{align}
where the vertical maps are the forgetful homomorphisms.

For the given $G$-action $\{F_g, \varepsilon_{g, h}|\; g, h\in G\}$ on $\mathcal{C}$,  we consider the following normal subgroup of ${\rm Aut}(\mathcal{C}; G)$
 $${\rm Act}(\mathcal{C}; G)=\{[F, \delta]\in {\rm Aut}(\mathcal{C}; G)\; |\; F \mbox{ is isomorphic to } F_a \mbox{ for some } a\in G\}.$$
 Here, ``Act" stands for action.

\begin{cor}\label{cor:Act}
Let  $(\Theta, \partial)\colon \mathcal{C}\rightarrow \mathcal{D}$ be the above  $k$-linear $G$-equivariant equivalence. Then the above isomorphism $(\Theta, \partial)_*$ restricts to an isomorphism
$$(\Theta, \partial)_*\colon {\rm Act}(\mathcal{C}; G)\longrightarrow {\rm Act}(\mathcal{D}; G)$$
\end{cor}

\begin{proof}
It follows from the commutative diagram (\ref{equ:commdiag}) and the fact that $\Theta_*([F_a])=F'_a$ for each $a\in G$.
\end{proof}

\subsection{An isomorphism theorem}

We fix a $k$-linear $G$-action $\{F_g, \varepsilon_{g, h}|\; g, h\in G\}$ on $\mathcal{C}$, and consider the category $\mathcal{C}^G$ of $G$-equivariant objects. Recall that the character group $\widehat{G}$ has a strict action on $\mathcal{C}^G$ by sending each character $\chi$ to the automorphism $F_\chi$ on $\mathcal{C}^G$; see Subsection 4.3. We then have the corresponding group ${\rm Aut}(\mathcal{C}^G; \widehat{G})$ of $\widehat{G}$-equivariant autoequivalences on $\mathcal{C}^G$ and its subgroup ${\rm Act}(\mathcal{C}^G; \widehat{G})$.

 For a given $G$-equivariant $k$-linear autoequivalence $(F, \delta)$ on $\mathcal{C}$, we have by Lemma \ref{lem:G-equ} the autoequivalence $(F, \delta)^G\colon \mathcal{C}^G\rightarrow \mathcal{C}^G$.  We observe that $F_\chi (F, \delta)^G=(F, \delta)^GF_\chi$ for each $\chi\in \widehat{G}$. In other words,  the autoequivalence $(F,\delta)^G$ is a strictly $\widehat{G}$-equivariant endofunctor. More precisely, the data $((F, \delta)^G, ({\rm Id}_{(F, \delta)^GF_\chi})_{\chi\in \widehat{G}})$ form a strictly $\widehat{G}$-equivariant functor, which will be abbreviated as $(F, \delta)^G$.

  Hence, the following group homomorphism is well defined
\begin{align}\label{equ:(-)^G}
(-)^G\colon {\rm Aut}(\mathcal{C};G) \longrightarrow {\rm Aut}(\mathcal{C}^G; \widehat{G}), \; [F, \delta]\mapsto [(F, \delta)^G].
\end{align}
We refer to $(-)^G$ as the \emph{equivariantization homomorphism} associated to the given $G$-action on $\mathcal{C}$.

\begin{thm}\label{thm:duality-iso}
Let $G$ be a finite abelian group, which splits over $k$. Keep the notation as above. Assume further that $\mathcal{C}$ is idempotent complete. Then the above equivariantization homomorphism $(-)^G$ is an isomorphism.
\end{thm}

In particular, each $\widehat{G}$-equivariant autoequivalence on $\mathcal{C}^G$ is isomorphic to a strictly $\widehat{G}$-equivariant autoequivalence.

\begin{proof}
We recall that $G$ is identified with $\widehat{\widehat{G}}$ via the evaluation map (\ref{equ:ev}). Then we have the double-dual $G$-action on $(\mathcal{C}^G)^{\widehat{G}}$. It is a strict action by sending $g\in G$ to the automorphism $\widehat{F}_g$ on $(\mathcal{C}^G)^{\widehat{G}}$; see (\ref{equ:doubleFg}). Recall from Theorem \ref{thm:duality} the $G$-equivariant equivalence $(\Theta, (\partial_g)_{g\in G})\colon \mathcal{C}\rightarrow (\mathcal{C}^G)^{\widehat{G}}$. Here, we use the assumption that $\mathcal{C}$ is idempotent complete.

Similar as (\ref{equ:(-)^G}), we have the equivariantization homomorphism associated to the dual $\widehat{G}$-action on $\mathcal{C}^G$
$$(-)^{\widehat{G}}\colon {\rm Aut}(\mathcal{C}^G, \widehat{G})\longrightarrow {\rm Aut}((\mathcal{C}^G)^{\widehat{G}}; G).$$
We claim that there is a commutative triangle of group homomorphisms.
\begin{align}\label{equ:triangle}
\xymatrix{ {\rm Aut}(\mathcal{C}; G) \ar[rd]_-{(-)^G}\ar[rr]^-{(\Theta, \partial)_*} && {\rm Aut}((\mathcal{C}^G)^{\widehat{G}}; G)\\
& {\rm Aut}(\mathcal{C}^G; \widehat{G}) \ar[ru]_-{(-)^{\widehat{G}}}
}
\end{align}

Recall from (\ref{equ:commdiag}) that $(\Theta,\partial)_*$ is an isomorphism. It follows from the claim that $(-)^G$ is injective and that $(-)^{\widehat{G}}$ is surjective. Applying the same argument to the dual $\widehat{G}$-action,  we infer that $(-)^{\widehat{G}}$ is also injective. This implies the required statements.

For the claim, we take an arbitrary  $G$-equivariant autoequivalence $(F, \delta)$ on $\mathcal{C}$. For an object $X$ in $\mathcal{C}$, we have
\begin{center}$\Theta F(X)=({\rm Ind}(FX), {\rm can}(FX))$ and $((F,\delta)^G)^{\widehat{G}}\Theta(X)=((F,\delta)^G {\rm Ind}(X), \widetilde{{\rm can}(X)})$,
\end{center}
where $\widetilde{{\rm can}(X)}_\chi=(F,\delta)^G({\rm can}(X)_\chi)$ for each $\chi \in \widehat{G}$. It follows that the isomorphism $\xi$ in (\ref{equ:xi}) yields a natural isomorphism of functors
$$\xi\colon ((F,\delta)^G)^{\widehat{G}} \Theta\longrightarrow \Theta F.$$
Here, we use the explicit construction of `can' in Subsection 4.3 to verify that ${\rm can}(FX)_\chi\circ \xi_X=F_\chi(\xi_X)\circ \widetilde{{\rm can}(X)}_\chi$ for each $\chi\in\widehat{G}$.

  Moreover, the above isomorphism $\xi$ is an isomorphism of $G$-equivariant functors, that is, the following identity holds
   $$\widehat{F}_g\xi\circ ((F,\delta)^G)^{\widehat{G}}\partial_g=(\partial_g F\circ \Theta\delta_g)\circ \xi F_g.$$
   Here, we use the explicit form of $\partial_g$ and $\xi$. The above identity follows, since $(F, \delta)$ is a $G$-equivariant endofunctor on $\mathcal{C}$, that is, (\ref{equ:equi-fun}) holds with $F'_g=F_g$ and $\varepsilon'_{g, h}=\varepsilon_{g, h}$.  By the very definition of $(\Theta, \partial)_*$,  we conclude from the isomorphism $\xi$ that $(\Theta, \partial)_*(F, \delta)$  is isomorphic to the strictly $G$-equivariant functor $((F, \delta)^G)^{\widehat{G}}$. This proves the  claim.
\end{proof}

We are indebted to Hideto Asashiba for the remark below.

\begin{rem}
One might deduce Theorem \ref{thm:duality-iso} from the following conjectural observation: there is a  $2$-equivalence between the $2$-category of small idempotent complete $k$-linear categories with $G$-actions and the $2$-category of small idempotent complete $k$-linear categories with $\widehat{G}$-actions. This observation might be expected from  \cite[Theorem 4.4]{DGNO}; compare \cite[Theorem 7.5]{Asa17}. \hfill $\square$
\end{rem}

Recall that $Z(\mathcal{C})$ is the center of $\mathcal{C}$; it is a commutative $k$-algebra. The following  well-known fact will be used later; compare Lemma \ref{lemA:Nat}.

\begin{lem}\label{lem:nat}
Let $F\colon \mathcal{C}\rightarrow \mathcal{C}$ be a $k$-linear autoequivalence. Then any natural morphism $F\rightarrow F$ is of the form $\lambda F=F\lambda'$ for some uniquely determined $\lambda, \lambda'\in Z(\mathcal{C})$. Moreover, if $\lambda$ belongs to $k$, then $\lambda=\lambda'$. \hfill $\square$
\end{lem}

 To each character $\chi\in \widehat{G}$, we associate a $G$-equivariant endofunctor
 $$({\rm Id}_\mathcal{C}, (\chi(g)^{-1}{\rm Id}_{F_g})_{g\in G})\colon \mathcal{C}\longrightarrow \mathcal{C},$$
  where $\chi(g)^{-1}{\rm Id}_{F_g}\colon {\rm Id}_\mathcal{C} F_g\rightarrow F_g{\rm Id}_\mathcal{C}$ is a natural isomorphism induced by the multiplication of $\chi(g)^{-1}$.  We observe the following identity
  \begin{align}\label{equ:Fchi}
  ({\rm Id}_\mathcal{C}, (\chi(g)^{-1}{\rm Id}_{F_g})_{g\in G} )^G=F_\chi,
  \end{align}
  where the left hand side is by the construction (\ref{equ:fun-equi}) and $F_\chi$ is defined in (\ref{equ:defn-Fchi}).

Recall the normal subgroup of ${\rm Aut}(\mathcal{C}^G;\widehat{G})$
$${\rm Act}(\mathcal{C}^G;\widehat{G})=\{[H, \delta]\in {\rm Aut}(\mathcal{C}^G;\widehat{G})\; |\; H \mbox{ is isomorphic to } F_\chi \mbox{ for some }\chi \in \widehat{G} \}. $$

\begin{prop}\label{prop:Act}
Keep the assumptions in Theorem \ref{thm:duality-iso}. Assume further that $Z(\mathcal{C})=k=Z(\mathcal{C}^G)$.  Then the equivariantization homomorphism $(-)^G$ restricts to an isomorphism
$$(-)^G\colon {\rm Act}(\mathcal{C}; G)\longrightarrow {\rm Act}(\mathcal{C}^G;  \widehat{G}).$$
\end{prop}

\begin{proof}
Take an element $[F_a, \delta]$ from ${\rm Act}(\mathcal{C}; G)$ for some $a\in G$. Recall the $G$-equivariant functor $(F_a, c_a)$ from Example \ref{exm:equi-fun}. By Lemma \ref{lem:nat} there is a unique nonzero scalar $\chi(g)\in k$ such that $\delta_g=\chi(g)^{-1} (c_a)_g$ for each $g\in G$. Indeed, one infers that $\chi$ is a character of $G$. In other words, we have the following identity of $G$-equivariant functors
$$(F_a, \delta)=({\rm Id}_\mathcal{C}, (\chi(g)^{-1}{\rm Id}_{F_g})_{g\in G})\; (F_a, c_a).$$
Hence, the following identity holds
$$(F_a, \delta)^G=({\rm Id}_\mathcal{C}, (\chi(g)^{-1}{\rm Id}_{F_g})_{g\in G})^G \; (F_a, c_a)^G.$$
We recall (\ref{equ:Fchi}), and that the functor $(F_a, c_a)^G$ is isomorphic to ${\rm Id}_{\mathcal{C}^G}$; see Example \ref{exm:equi-fun}. It follows that the underlying functor of the strictly $\widehat{G}$-equivariant functor $(F_a, \delta)^G$ is isomorphic to $F_\chi$. In other words, we have that $[(F_a, \delta)^G]$ lies in ${\rm Act}(\mathcal{C}^G; \widehat{G})$.

We have shown that $({\rm Act}(\mathcal{C}; G))^G\subseteq {\rm Act}(\mathcal{C}^G; \widehat{G})$. By the same argument,  we have $({\rm Act}(\mathcal{C}^G; \widehat{G}))^{\widehat{G}}\subseteq {\rm Act}((\mathcal{C}^G)^{\widehat{G}}; G)$. Now we are done by the commutative triangle (\ref{equ:triangle}) and Corollary \ref{cor:Act}.
\end{proof}

\subsection{Bijections between  stable objects}

 In this subsection, we assume that the $k$-linear category $\mathcal{C}$ is idempotent complete and Hom-finite. Here, the Hom-finiteness means that ${\rm Hom}_\mathcal{C}(X, Y)$ is a finite dimensional $k$-space for any objects $X, Y\in \mathcal{C}$. In particular, the category $\mathcal{C}$ is Krull-Schmidt. An object $M$ in $\mathcal{C}$ is \emph{basic} if each of its indecomposable direct summands has multiplicity one. In this case, the endomorphism algebra ${\rm End}_\mathcal{C}(M)$ is basic.

Let $G$ be a finite abelian group. We fix a $k$-linear $G$-action $\{F_g, \varepsilon_{g, h}|\; g, h\in G\}$ on $\mathcal{C}$. Recall that an object $M$ is $G$-stable provided that $F_g(M)\simeq M$ for each $g\in G$. We denote by ${\rm Stab}^G(\mathcal{C})$ the set of isoclasses of basic $G$-stable objects. It carries a natural ${\rm Aut}(\mathcal{C}; G)$-action by $[F, \delta].M=F(M)$. Here, using the isomorphisms $\delta_g$'s, we observe that $F(M)$ is again $G$-stable.  We are interested in the orbit set ${\rm Stab}^G(\mathcal{C})/{{\rm Aut}(\mathcal{C}; G)}$.

We consider the dual $\widehat{G}$-action on the category $\mathcal{C}^G$ of $G$-equivariant objects. Then ${\rm Stab}^{\widehat{G}}(\mathcal{C}^G)$ denotes the set of isoclasses of basic $\widehat{G}$-stable objects in $\mathcal{C}^G$, which has a natural action by ${\rm Aut}(\mathcal{C}^G; \widehat{G})$. Recall that ${\rm Ind}\colon \mathcal{C}\rightarrow \mathcal{C}^G$ is the induction functor.

The following result shows that basic stable objects in $\mathcal{C}$ and in $\mathcal{C}^G$ are in  a natural bijection, which extends the corresponding result for cyclic group actions on module categories; see \cite[Corollary~13]{Hub} in slightly different terminologies.

\begin{prop}\label{prop:stable-bij}
 Let $G$ be a finite abelian group, which splits over $k$. Then there is a bijection
 $$\iota\colon {\rm Stab}^G(\mathcal{C})\longrightarrow {\rm Stab}^{\widehat{G}}(\mathcal{C}^G), \quad M\mapsto \iota(M),$$
 where the object $\iota(M)$ is uniquely determined by ${\rm add}\; \iota(M)={\rm add}\; {\rm Ind}(M)$. Moreover, $\iota$ induces a bijection on orbit sets
 $${\rm Stab}^G(\mathcal{C})/{{\rm Aut}(\mathcal{C}; G)}\longrightarrow {\rm Stab}^{\widehat{G}}(\mathcal{C}^G)/{{\rm Aut}(\mathcal{C}^G; \widehat{G})}.$$
\end{prop}

\begin{proof}
The map $\iota$ is well defined, since ${\rm Ind}(M)$ is always $\widehat{G}$-stable; see Subsection 4.3. It is injective, since ${\rm add}\; U\iota(M)={\rm add} \; M$. Here, $U$ is the forgetful functor.  For the surjectivity, let $(Y, \alpha)$ be a basic $\widehat{G}$-stable object in $\mathcal{C}^G$. Then $F_\chi(Y, \alpha)\simeq (Y, \alpha)$ for each $\chi\in \widehat{G}$. By the isomorphism (\ref{equ:iso-monad}), we infer that ${\rm add}\; {\rm Ind}(Y)={\rm add}\; (Y, \alpha)$. Take a basic object $Y_0$ such that ${\rm add}\; Y_0={\rm add}\; Y$. Since $Y$ is clearly $G$-stable, so is $Y_0$. We observe that $\iota(Y_0)=(Y, \alpha)$. Then we are done.

For the second bijection, it suffices to show that $\iota$ is compatible with the equivariantization isomorphism in Theorem \ref{thm:duality-iso}. Indeed, for a $G$-equivariant autoequivalence $(F, \delta)$ on $\mathcal{C}$ and a $G$-stable object $M$, we recall the isomorphism $\xi_M\colon (F, \delta)^G {\rm Ind}(M)\rightarrow {\rm Ind}F(M)$ in (\ref{equ:xi}). Then we have the following isomorphism
$$\iota F(M)\simeq (F, \delta)^G \iota(M).$$
Since the equivariantization isomorphism sends $[F, \delta]$ to $[(F, \delta)^G]$, we infer the required compatibility.
\end{proof}

\begin{rem}\label{rem:bij}
(1) A basic $G$-stable object $M$ is \emph{$G$-indecomposable} if it is not decomposable into the direct sum of two $G$-stable objects. Then the bijection $\iota$ induces a bijection between $G$-indecomposable objects and $\widehat{G}$-indecomposable objects.

 Denote by ${\rm ind}\; \mathcal{C}$ the set of   isomorphism classes of indecomposable objects in $\mathcal{C}$. Then $G$ acts on ${\rm ind}\; \mathcal{C}$. There is a bijection between the orbit set ${({\rm ind}\; \mathcal{C})}/G$ and the set of $G$-indecomposables, sending a $G$-orbit to its direct sum. Consequently, we obtain a bijection between ${({\rm ind}\; \mathcal{C})}/G$ and ${({\rm ind}\; \mathcal{C}^G)}/\widehat{G}$.

(2) For a basic $G$-stable object $M$, we observe that ${\rm End}_{\mathcal{C}^G}(\iota(M))$ is Morita equivalent to ${\rm End}_{\mathcal{C}^G}({\rm Ind}(M))$. We claim that ${\rm End}_{\mathcal{C}^G}({\rm Ind}(M))$ is isomorphic to ${\rm End}_\mathcal{C}(M)\ast G$, where the crossed product is with respect to the weak $G$-action on ${\rm End}_\mathcal{C}(M)$ described in Subsection 2.3; compare \cite[Proposition 3.1.1]{DLS}. In particular, we have chosen an isomorphism $\alpha_g\colon M\rightarrow F_g(M)$ for each $g\in G$.

The claim follows immediately from the following isomorphisms
$${\rm End}_{\mathcal{C}^G}({\rm Ind}(M))\stackrel{\sim}\longrightarrow \bigoplus_{g\in G} {\rm Hom}_\mathcal{C}(M, F_g(M)) \stackrel{\sim}\longleftarrow  {\rm End}_\mathcal{C}(M)\ast G,$$
where the left isomorphism is by the adjunction (\ref{equ:adj1}) and the right one sends $a\overline{g^{-1}}$ to ${\rm inc}_g \circ F_g(a)\circ \alpha_g$. Here, ${\rm inc}_g\colon F_g(T)\rightarrow \bigoplus_{h\in G}F_h(T)$ is the inclusion, and $a$ belongs to ${\rm End}_\mathcal{C}(M)$. \hfill $\square$
\end{rem}

\section{Groups of  equivariant triangle  autoequivalences}

In this section, we sketch a triangle version of Theorem \ref{thm:duality-iso}. We first recall the notion of a triangle action by a group on a (pre-)triangulated category.

\subsection{Triangle actions}

Let $\mathcal{T}$ be a pre-triangulated category with the translation functor $\Sigma$. Here, by a pre-triangulated category we mean a triangulated category,  which possibly does not satisfy the octahedral axiom.

A triangle functor between two pre-triangulated categories is a pair $(F, \omega)$, where $F\colon \mathcal{T}\rightarrow \mathcal{T}'$ is an additive functor and $\omega\colon F\Sigma\rightarrow \Sigma'F$ is a natural isomorphism such that any exact triangle $X\rightarrow Y\rightarrow Z\stackrel{f}\rightarrow \Sigma(X)$ in $\mathcal{T}$ is sent to an exact triangle $F(X)\rightarrow F(Y)\rightarrow F(Z)\xrightarrow{\omega_X\circ F(f)} \Sigma'F(X)$ in $\mathcal{T}'$. We call $\omega$ the \emph{connecting isomorphism} for $F$. The composition and natural isomorphisms of triangle functors respect the connecting isomorphisms. For example, $({\rm Id}_\mathcal{T}, {\rm Id}_\Sigma)$ is a triangle functor, which is simply denoted by ${\rm Id}_\mathcal{T}$. More generally, the connecting isomorphism  in a triangle functor $(F, \omega)$ is \emph{trivial} provided that $F\Sigma=\Sigma'F$ and that $\omega={\rm Id}_{F\Sigma}$ is the identity transformation. In this case, the triangle functor $(F, {\rm Id}_{F\Sigma})$ is simply denoted by $F$.

Let $\mathcal{T}$ be a pre-triangulated category and $G$ a finite group. A \emph{triangle $G$-action} $\{(F_g, \omega_g), \varepsilon_{g, h} |\; g, h\in G \}$  on $\mathcal{T}$ consists of triangle autoequivalences $(F_g, \omega_g)$ on $\mathcal{T}$ and natural isomorphisms
$$\varepsilon_{g, h}\colon (F_g, \omega_g)(F_h, \omega_h)\longrightarrow (F_{gh}, \omega_{gh})$$
 of triangle functors subject to the condition (\ref{equ:2-coc}). Since the isomorphism $\varepsilon_{g, h}$ respects the connecting isomorphisms, we have the condition
\begin{align}\label{equ:tr-ac}
 \omega_{gh}\circ \varepsilon_{g, h}\Sigma=\Sigma \varepsilon_{g,h}\circ (\omega_gF_h \circ F_g \omega_h).
\end{align}

We consider the category $\mathcal{T}^G$ of $G$-equivariant objects in $\mathcal{T}$; it is an additive category. We observe  by (\ref{equ:tr-ac}) that
$$(\Sigma, (\omega_g^{-1})_{g\in G})\colon \mathcal{T}\longrightarrow \mathcal{T}$$
 is a $G$-equivariant endofunctor. In particular, we have a well-defined endofunctor
 $$(\Sigma, (\omega_{g}^{-1})_{g\in G})^G\colon \mathcal{T}^G\longrightarrow \mathcal{T}^G,$$
 which is abbreviated as $\Sigma^G$; it is an autoequivalence; see Lemma \ref{lem:G-equ}.  For an object $(X, \alpha)$ in $\mathcal{T}^G$, we have $\Sigma^G(X, \alpha)=(\Sigma(X), \tilde{\alpha})$ with $\tilde{\alpha}_g=(\omega_g)_X^{-1}\circ \Sigma(\alpha_g)$ for each $g\in G$. The autoequivalence $\Sigma^G$ acts on morphisms by $\Sigma$.

The following basic result is essentially due to \cite[Corollary 4.3]{Bal}, which is made explicit in \cite[Lemma 4.4]{Chen15} and \cite[Theorem 6.9]{El2014}.

\begin{lem}\label{lem:pretri}
Assume that the pre-triangulated category $\mathcal{T}$ has a triangle $G$-action as above. Suppose that $\mathcal{T}$ is idempotent complete and that $|G|$ is invertible in $\mathcal{T}$. Then $\mathcal{T}^G$ is pre-triangulated  with $\Sigma^G$ its translation functor and satisfying the following condition: a triangle $(X, \alpha)\rightarrow (Y, \beta)\rightarrow (Z, \gamma)\rightarrow \Sigma^G(X, \alpha)$ is exact if and only if the corresponding triangle of underlying objects is exact in $\mathcal{T}$. \hfill $\square$
\end{lem}

 In the case of Lemma \ref{lem:pretri}, both the forgetful functor $U\colon \mathcal{T}^G\rightarrow \mathcal{T}$ and the induction functor ${\rm Ind}\colon \mathcal{T}\rightarrow \mathcal{T}^G$ are triangle functors, where the connecting isomorphism for $U$ is trivial and the one for ${\rm Ind}$ is induced from $\omega_g$'s.

Indeed, in most cases,  the category $\mathcal{T}^G$ is triangulated; see \cite[Corollary 6.10]{El2014}. In general, we do not know whether  $\mathcal{T}^G$ is triangulated under the assumption that $\mathcal{T}$ is triangulated.

Let $k$ be a field.  Let $\mathcal{T}$ be a $k$-linear pre-triangulated category, which is skeletally small. Recall that the \emph{triangle center}  of $\mathcal{T}$ is given by $Z_\vartriangle(\mathcal{T})=\{\lambda\in Z(\mathcal{T})\; |\; \Sigma \lambda=\lambda \Sigma\}$; it is a $k$-subalgebra of the center $Z(\mathcal{T})$. We observe that the triangle center is bijective to the set consisting of natural transformations $\lambda\colon ({\rm Id}_\mathcal{T}, {\rm Id}_{\Sigma})\rightarrow ({\rm Id}_\mathcal{T}, {\rm Id}_{\Sigma})$ as triangle functors.

We denote by ${\rm Aut}_\vartriangle(\mathcal{T})$ the \emph{group of triangle autoequivalences} on $\mathcal{T}$, whose elements are the isoclasses $[F, \omega]$ of $k$-linear triangle autoequivalences $(F, \omega)$ of $\mathcal{T}$ and whose multiplication is given by the composition of triangle functors. The  homomorphism
$${\rm Aut}_\vartriangle(\mathcal{T})\longrightarrow {\rm Aut}(\mathcal{T})$$
 sending $[F, \omega]$ to $[F]$ is in general neither injective nor surjective.

\subsection{A triangle version of Theorem \ref{thm:duality-iso}}

Let $k$ be a field and let $\mathcal{T}$ be a $k$-linear pre-triangulated category, which is skeletally small. Let $G$ be a finite abelian group,  which splits over $k$. We fix a $k$-linear triangle $G$-action $\{(F_g, \omega_g), \varepsilon_{g, h} |\; g, h\in G \}$ on $\mathcal{T}$.

We denote  by ${\rm Aut}_\vartriangle^G(\mathcal{T})$ the subgroup of ${\rm Aut}_\vartriangle(\mathcal{T})$ consisting of elements $[F, \omega]$ which commute with each $[F_g, \omega_g]$. The group of equivariant triangle autoequivalences behaves better.

A $G$-equivariant triangle endofunctor $((F, \omega), \delta)$ on $\mathcal{T}$ consists of a triangle endofunctor $(F, \omega)$ and natural isomorphisms
$$\delta_g\colon (F, \omega) (F_g, \omega_g) \longrightarrow (F_g, \omega_g) (F, \omega)$$
between triangle functors subject to (\ref{equ:equi-fun}). We denote by ${\rm Aut}_\vartriangle(\mathcal{T}; G)$ the \emph{group of equivariant triangle autoequivalences} on $\mathcal{T}$, whose elements are the isoclasses $[(F, \omega), \delta]$ of equivariant triangle autoequivalences. Then we have the \emph{forgetful homomorphism} associated to the given triangle $G$-action
$$\phi\colon {\rm Aut}_\vartriangle(\mathcal{T}; G)\longrightarrow {\rm Aut}_\vartriangle^G(\mathcal{T}), \;  [(F, \omega), \delta]\mapsto [F, \omega].$$
We denote by ${\rm Act}_\vartriangle(\mathcal{T}; G)$ the normal subgroup of
${\rm Aut}_\vartriangle(\mathcal{T}; G)$ consisting of those elements $[(F, \omega), \delta]$ such that $(F, \omega)$ is isomorphic to $(F_a, \omega_a)$ for some $a\in G$.

We consider the dual $\widehat{G}$-action on $\mathcal{T}^G$, which sends $\chi$ to the automorphism $F_\chi$ in (\ref{equ:defn-Fchi}). We observe that $F_\chi$ is a triangle functor with the trivial connecting isomorphism. In particular, the dual $\widehat{G}$-action on $\mathcal{T}^G$ is a triangle action. We then have the corresponding groups  ${\rm Aut}_\vartriangle(\mathcal{T}^G; \widehat{G})$ and ${\rm Act}_\vartriangle(\mathcal{T}^G; \widehat{G})$.

Let $((F, \omega), \delta)\colon \mathcal{T}\rightarrow \mathcal{T}$ be an equivariant triangle autoequivalence. Recall that the autoequivalence $(F, \delta)^G$ on $\mathcal{T}^G$ is strictly $\widehat{G}$-equivariant. We have a natural isomorphism
$$\omega^G\colon (F, \delta)^G \Sigma^G\longrightarrow \Sigma^G (F, \delta)^G$$
given by $(\omega^G)_{(X, \alpha)}=\omega_X$ for each object $(X, \alpha)\in \mathcal{T}^G$. Moreover, the pair $((F, \delta)^G, \omega^G)$ is a triangle endofunctor on $\mathcal{T}^G$. Hence, we have a strictly $\widehat{G}$-equivariant triangle autoequivalence $(((F, \delta)^G, \omega^G), {\rm Id})$ on $\mathcal{T}^G$, which will be abbreviated as $((F, \omega), \delta)^G$.

This gives to the triangle version of the \emph{equivariantization homomorphism} associated to the given  triangle $G$-action
$$(-)^G\colon {\rm Aut}_\vartriangle(\mathcal{T}; G)\longrightarrow {\rm Aut}_\vartriangle(\mathcal{T}^G; \widehat{G}), \; [(F, \omega), \delta]\mapsto [((F, \omega), \delta)^G].$$

\begin{thm}\label{thm:duality-iso-tri}
Let $G$ be a finite abelian group, which splits over $k$.  Assume that the pre-triangulated category $\mathcal{T}$ is idempotent complete with the given $k$-linear triangle $G$-action. Then the above equivariantization homomorphism $(-)^G$ is an isomorphism.

Moreover, if $Z_\vartriangle(\mathcal{T})=k=Z_\vartriangle(\mathcal{T}^G)$, this equivariantization homomorphism restricts to an isomorphism
$$(-)^G\colon {\rm Act}_\vartriangle(\mathcal{T}; G)\longrightarrow {\rm Act}_\vartriangle(\mathcal{T}^G; \widehat{G}).$$
\end{thm}

\begin{proof}
We mention that the triangle version of Theorem \ref{thm:duality} holds. More precisely, we observe that the $G$-equivariant equivalence in Theorem \ref{thm:duality} is a triangle functor, whose connecting isomorphism is induced by the isomorphisms $\omega_g$'s. Then the argument in the proof of Theorem \ref{thm:duality-iso} and Proposition  \ref{prop:Act} carries over to the triangle case.
\end{proof}

\section{Stable tilting objects and bijections}

In this section, we study tilting objects that are stable under a given triangle action. If the acting group is finite abelian, there is a bijection between the sets of isoclasses of stable tilting objects for the given action and for the dual action.

Throughout, let $k$ be a field, and let $\mathcal{T}$ a $k$-linear pre-triangulated category, which is assumed to be idempotent complete and Hom-finite.

\subsection{Tilting objects} We  recall the basic facts on tilting objects. For an object $M$ in $\mathcal{T}$,  we denote by ${\rm thick}\langle M\rangle $ the smallest thick triangulated subcategory of $\mathcal{T}$ containing $M$. Here, the thickness means being closed under direct summands.

\begin{defn}
An object $T\in \mathcal{T}$ is called \emph{tilting} provided that the following conditions are fulfilled.
\begin{enumerate}
\item[(T1)] The object $T$ satisfies ${\rm Hom}_\mathcal{T}(T, \Sigma^i(T))=0$ for $i\neq 0$;
\item[(T2)] ${\rm thick}\langle T \rangle=\mathcal{T}$. \hfill $\square$
\end{enumerate}
\end{defn}

The motivating example is as follows. Let $A$ be a finite dimensional algebra. We denote by $\mathbf{K}^b({\rm proj}\mbox{-}A)$ the bounded homotopy category of projective $A$-modules. We view a module $X$ as a stalk complex concentrated in degree zero, still denoted by $X$.  Then the regular module $A$ is a tilting object in $\mathbf{K}^b({\rm proj}\mbox{-}A)$.

Following \cite[Subsection 8.7]{Kel}, the pre-triangulated category $\mathcal{T}$ is \emph{algebraic} provided that there is a $k$-linear Frobenius category $\mathcal{E}$ with a $k$-linear triangle equivalence $\mathcal{T}\rightarrow \underline{\mathcal{E}}$. Here, $\underline{\mathcal{E}}$ is the stable category modulo projective objects; it is a $k$-linear triangulated category.

The following basic fact is well known, which is essentially contained in \cite[Theorem 8.5]{Kel}; for a sketchy proof, we refer to \cite[Lemma 3.1]{Chen11}.

\begin{prop}\label{prop:tilting}
Let $\mathcal{T}$ be an algebraic triangulated category. Assume that $T$ is a tilting object. Then there is a $k$-linear triangle equivalence
$$(\Psi, u)\colon \mathcal{T}\longrightarrow \mathbf{K}^b({\rm proj}\mbox{-}{\rm End}_\mathcal{T}(T)),$$
which restricts to the equivalence ${\rm Hom}_\mathcal{T}(T, -)\colon {\rm add}\; T\rightarrow {\rm proj}\mbox{-}{\rm End}_\mathcal{T}(T)$.  \hfill $\square$
\end{prop}

We denote by ${\rm Tilt}(\mathcal{T})$ the  set of isoclasses  of basic tilting objects in $\mathcal{T}$. This set
carries an action by ${\rm Aut}_\vartriangle(\mathcal{T})$ as follows: $[F, \omega].T=F(T)$ for any $k$-linear triangle autoequivalence $(F, \omega)$.  We are interested in the orbit set ${\rm Tilt}(\mathcal{T})/{\rm Aut}_\vartriangle(\mathcal{T})$, for which we have another interpretation.

We denote by ${\rm Alg}(\mathcal{T})$ the set of isoclasses of finite dimensional basic $k$-algebras, which are isomorphic to ${\rm End}_\mathcal{T}(T)$ for some basic tilting object $T$. The following map is surjective by definition
\begin{align}\label{equ:alg}
{\rm Tilt}(\mathcal{T})/{\rm Aut}_\vartriangle(\mathcal{T})\longrightarrow {\rm Alg}(\mathcal{T}),\quad  T\mapsto {\rm End}_\mathcal{T}(T).
\end{align}

The following observation shows that this map is bijective, provided that $\mathcal{T}$ is algebraic.

\begin{cor}\label{cor:bij}
Let $\mathcal{T}$ be an algebraic triangulated category.  Assume that $T$ and $T'$ are two tilting objects. Then there is a $k$-linear triangle autoequivalence $(F, \omega)$ on $\mathcal{T}$ satisfying $F(T)\simeq T'$ if and only if the algebras ${\rm End}_\mathcal{T}(T)$ and ${\rm End}_\mathcal{T}(T')$ are isomorphic. Consequently, the map (\ref{equ:alg}) is bijective.
\end{cor}

\begin{proof}
The ``only if" part is clear, since $F$ induces the isomorphism. Conversely, an isomorphism between these algebras yields an isomorphism $\mathbf{K}^b({\rm proj}\mbox{-}{\rm End}_\mathcal{T}(T))\rightarrow \mathbf{K}^b({\rm proj}\mbox{-}{\rm End}_\mathcal{T}(T'))$ of triangulated categories, which sends ${\rm End}_\mathcal{T}(T)$ to ${\rm End}_\mathcal{T}(T')$. Combining it with the equivalences in Proposition \ref{prop:tilting} applied to $T$ and $T'$, we obtain the required triangle autoequivalence on $\mathcal{T}$.
\end{proof}

\subsection{Stable tilting objects} Let $G$ be a finite group, whose order $|G|$ is invertible in $k$. We fix a $k$-linear triangle $G$-action $\{(F_g, \omega_g), \varepsilon_{g,h}|\; g,h\in G\}$ on $\mathcal{T}$.

We denote by ${\rm Tilt}^G(\mathcal{T})\subseteq {\rm Tilt}(\mathcal{T})$ the subset consisting of basic $G$-stable tilting objects. There is an ${\rm Aut}_\vartriangle(\mathcal{T}; G)$-action on ${\rm Tilt}^G(\mathcal{T})$ as follows: $[(F, \omega), \delta].T=F(T)$. By the isomorphisms $\delta_g$'s, the tilting object $F(T)$ is indeed $G$-stable. We are interested in the orbit set ${\rm Tilt}^G(\mathcal{T})/{{\rm Aut}_\vartriangle(\mathcal{T}; G)}$.  We mention that stable tilting objects are studied in \cite{DLS, Asa11, Asa13,  Jas, ZH, Nov} in quite different setups.

\begin{rem}\label{rem:surj}
We observe that ${\rm Aut}_\vartriangle^G(\mathcal{T})$ also acts on ${\rm Tilt}^G(\mathcal{T})$. Indeed, if the forgetful homomorphism $\phi\colon {{\rm Aut}_\vartriangle(\mathcal{T}; G)}\rightarrow {\rm Aut}_\vartriangle^G(\mathcal{T})$ is surjective, then the orbit sets ${\rm Tilt}^G(\mathcal{T})/{{\rm Aut}_\vartriangle(\mathcal{T}; G)}$ and ${\rm Tilt}^G(\mathcal{T})/{{\rm Aut}_\vartriangle^G(\mathcal{T})}$ coincide.  This surjectivity holds provided that $G$ is cyclic, $Z_\vartriangle(\mathcal{T})=k$ and that the ground field $k$ is nice enough; see Section 9 for details.  \hfill $\square$
\end{rem}

We denote by ${\rm Alg}(\mathcal{T}; G)$ the set of equivalence classes of $G$-crossed systems $(A, \rho, c)$ with $A\in {\rm Alg}(\mathcal{T})$; see Definition \ref{defn:equiv}.

 In what follows, we will construct a $G$-equivariant version of the map (\ref{equ:alg}). For a basic $G$-stable tilting object $T$, we choose for each $g\in G$ an isomorphism $\alpha_g\colon T\rightarrow F_g(T)$. Set $A={\rm End}_\mathcal{T}(T)$. In Subsection 2.3, we already obtain a $G$-crossed system $(A, \rho, c)$ from these isomorphisms. By Lemma \ref{lem:equiv-cross}, the following map is well defined
\begin{align}\label{equ:alg-G}
{\rm Tilt}^G(\mathcal{T})/{{\rm Aut}_\vartriangle(\mathcal{T}; G)}\longrightarrow {\rm Alg}(\mathcal{T}; G), \quad T\mapsto (A, \rho, c).
\end{align}

The condition in the following definition is subtle.

\begin{defn}
Let $\mathcal{T}$ be an algebraic triangulated category as above. We say that $\mathcal{T}$ is \emph{tilt-standard} provided that it has a tilting object $T$ such that the additive category ${\rm add}\; T$ is strongly $\mathbf{K}$-standard in the sense of Definition \ref{defn:sstand}. \hfill $\square$
\end{defn}

In this situation, by Proposition \ref{prop:homot}, the additive category ${\rm add}\; T'$ is strongly $\mathbf{K}$-standard for any tilting object $T'$.

\begin{exm}
{\rm Let $\mathcal{H}$ be a hereditary abelian category such that its bounded derived category $\mathbf{D}^b(\mathcal{H})$ is Hom-finite. Assume that $\mathbf{D}^b(\mathcal{H})$ has a tilting object $T$, which is usually called a \emph{tilting complex} on $\mathcal{H}$. Then $\mathbf{D}^b(\mathcal{H})$  is tilt-standard.

Indeed, the endomorphism algebra $A={\rm End}_{\mathbf{D}^b(\mathcal{H})}(T)$ is piecewise hereditary, in particular, triangular. Then ${\rm proj}\mbox{-}A$ is strongly $\mathbf{K}$-standard; see Appendix B. By the equivalence ${\rm Hom}_{\mathbf{D}^b(\mathcal{H})}(T, -)\colon {\rm add}\; T\rightarrow {\rm proj}\mbox{-}A$, we are done.} \hfill $\square$
\end{exm}

\begin{prop}\label{prop:tri-stable}
Assume that $\mathcal{T}$ is tilt-standard.  Then the map (\ref{equ:alg-G}) is injective.
\end{prop}

\begin{rem}\label{rem:non-sur}
We mention that the map (\ref{equ:alg-G}) is not surjective in general. Indeed, we consider the Grothendieck group $K_0(\mathcal{T})$ of the triangulated category $\mathcal{T}$,  and the induced $G$-action on it. On the algebra side, we have the $G$-action on the $K_0$-group $K_0(A)$. More precisely, for $g\in G$ and $P\in {\rm proj}\mbox{-}A$ with its class $[P]\in K_0(A)$, the action is given by $g.[P]=[{^{\rho(g^{-1})}P}]$. Since $K_0(A)$ is naturally isomorphic to $K_0(\mathbf{K}^b({\rm proj}\mbox{-}A))$,  we might identify $K_0(A)$ with $K_0(\mathcal{T})$ via the equivalence in Proposition \ref{prop:tilting}. We observe that if $(A, \rho, c)$ lies in the image of (\ref{equ:alg-G}), then the two $G$-actions on $K_0(\mathcal{T})$ and $K_0(A)$ are necessarily compatible. \hfill $\square$
\end{rem}

In the cyclic group case, Proposition \ref{prop:tri-stable} provides a practical  way for classifying the  stable tilting objects. For $d\geq 2$ and an algebra $A$, we denote by ${\rm Out}(A)_d$ the subset of ${\rm Out}(A)$ consisting of $d$-compatible elements; see Subsection 3.2. Since those elements are closed under conjugation, we denote by $\overline{{\rm Out}(A)_d}$ the set of conjugation classes contained in ${\rm Out}(A)_d$ .

In view of Proposition \ref{prop:cyclic-act-alg}, we have the following immediate consequence of Proposition \ref{prop:tri-stable}.

\begin{cor}\label{cor:cyclc-cross}
Let $G=C_d=\langle g\rangle$ be the cyclic group of order $d$. Assume that $Z_\vartriangle(\mathcal{T})=k$ and that each element in $k$ has a $d$-th root. Keep the assumptions as above. Then there is an injective map
$${\rm Tilt}^{C_d}(\mathcal{T})/{{\rm Aut}_\vartriangle(\mathcal{T}; C_d)}\longrightarrow \{(A, \sigma)\; |\; A\in {\rm Alg}(\mathcal{T}), \sigma\in \overline{{\rm Out}(A)_d},   K_0\mbox{-}{\rm compatible} \}.$$
Here, the $K_0$-compatiblity of  $(A, \sigma)$ means that the action of $\sigma$ on $K_0(A)$ is isomorphic to the $g$-action on $K_0(\mathcal{T})$; see Remark \ref{rem:non-sur}. \hfill $\square$
\end{cor}

Let $T$ be a $G$-stable tilting object as above with $A={\rm End}_\mathcal{T}(T)$. We have obtained the $G$-crossed system $(A, \rho, c)$. For any autoquivalence $H$ on ${\rm proj}\mbox{-}A$, it naturally induces a triangle autoequivalence on $\mathbf{K}^b({\rm proj}\mbox{-}A)$, which has a trivial connecting isomorphism. The induced triangle equivalence is still denoted by $H$. Consequently, from the crossed system,  we obtain a triangle $G$-action $\{{^{\rho(g^{-1})}(-)},c_{g,h}|\; g,h\in G\}$ on  $\mathbf{K}^b({\rm proj}\mbox{-}A)$; see Subsection 2.3.

\begin{lem}
Assume that $\mathcal{T}$ is tilt-standard. Keep the notation as above, and recall the equivalence $(\Psi, u)$ in Proposition \ref{prop:tilting}. Then for each $g\in G$, there exists a natural isomorphism $\delta_g\colon (\Psi, u) (F_g, \omega_g)\rightarrow {^{\rho(g^{-1})}(-)} (\Psi, u)$ of triangle functors, such that
$$ ((\Psi, u), \delta)\colon \mathcal{T}\longrightarrow \mathbf{K}^b({\rm proj}\mbox{-}A)$$
is a $G$-equivariant triangle equivalence.
\end{lem}

\begin{proof}
We might use the equivalence $(\Psi, u)$ to transport the triangle action from $\mathcal{T}$ to $\mathbf{K}^b({\rm proj}\mbox{-}A)$; compare Lemma \ref{lem:transport}. More precisely, for each $g\in G$, there is a triangle autoequivalence $(H_g, v_g)$ on   $\mathbf{K}^b({\rm proj}\mbox{-}A)$ together with a natural isomorphism $\delta'_g\colon (\Psi, u) (F_g, \omega_g)\rightarrow (H_g, v_g) (\Psi, u)$. However, we observe that the restriction $H_g$ on ${\rm proj}\mbox{-}A$ is isomorphic to the twisting automorphism $^{\rho(g^{-1})}(-)$; compare Lemma~\ref{lem:stable-hom}. Since ${\rm add}\; T$ and thus ${\rm proj}\mbox{-}A$ are strongly $\mathbf{K}$-standard, we infer from Lemma~\ref{lem:appB} that $(H_g, v_g)$ is isomorphic to $^{\rho(g^{-1})}(-)$, the natural extension of the twisting automorphism on $\mathbf{K}^b({\rm proj}\mbox{-}A)$. So we replace $H_g$ by $^{\rho(g^{-1})}(-)$, and adjust $\delta_g$ suitably; see Lemma~\ref{lem:change}.  Then the result follows immediately.
\end{proof}

\noindent \emph{Proof of Proposition \ref{prop:tri-stable}.} \; We take two $G$-stable tilting objects $T$ and $T'$, which yield two $G$-crossed systems $(A, \rho, c)$ and $(A', \rho', c')$. To show the injectivity, we assume that $f\colon A\rightarrow A'$ is an algebra isomorphism such that $f_*(\rho, c)$ and $(\rho', c')$ are isomorphic; see Definition \ref{defn:equiv}.

The rows of the following commutative diagram are obtained by applying the previous lemma to $T$ and to $T'$, respectively. The right hand column is an $G$-equivariant triangle isomorphism induced by Lemma \ref{lem:fstar}.
\[\xymatrix{
\mathcal{T}\ar@{.>}[d]_{((F, \omega), \partial)} \ar[rr]^-{((\Psi, u), \delta)} && \ar[d]^-{f_*}\mathbf{K}^b({\rm proj}\mbox{-}A)\\
\mathcal{T} \ar[rr]^-{((\Psi', u'), \delta')} && \mathbf{K}^b({\rm proj}\mbox{-}A')
}\]
Then there is a unique $[(F, \omega), \partial]\in {\rm Aut}_\vartriangle(\mathcal{T}; G)$ making the diagram commute. Since $f_*(A)\simeq A'$, we infer that $F(T)\simeq T'$. We are done. \hfill $\square$

\subsection{Bijections between stable tilting objects}

Let $G$ be a finite abelian group, which splits over $k$.  We fix a $k$-linear triangle $G$-action $\{(F_g, \omega_g), \varepsilon_{g,h}|\; g,h\in G\}$ on $\mathcal{T}$. The induction functor ${\rm Ind}\colon \mathcal{T}\rightarrow \mathcal{T}^G$ is a triangle functor.  Consider the dual $\widehat{G}$-action on $\mathcal{T}^G$. So we  have the set ${\rm Tilt}^{\widehat{G}}(\mathcal{T}^G)$ of isoclasses  of basic $\widehat{G}$-stable tilting objects in $\mathcal{T}^G$, which carries an action by ${\rm Aut}_\vartriangle(\mathcal{T}^G; \widehat{G})$.

We have the following triangle analogue  of Proposition \ref{prop:stable-bij}, where the first bijection is related to \cite[Theorem 4.1]{Nov}.

\begin{thm}\label{thm:tilting}
 Let $G$ be a finite abelian group, which splits over $k$. Then there is a bijection
 $$\iota\colon {\rm Tilt}^G(\mathcal{T})\longrightarrow {\rm Tilt}^{\widehat{G}}(\mathcal{T}^G), T\mapsto \iota(T)$$
 such that ${\rm add}\; \iota(T)={\rm add}\; {\rm Ind}(T)$. Moreover, $\iota$ induces a bijection on orbit sets
 $${\rm Tilt}^G(\mathcal{T})/{{\rm Aut}_\vartriangle(\mathcal{T}; G)}\longrightarrow {\rm Tilt}^{\widehat{G}}(\mathcal{T}^G)/{{\rm Aut}_\vartriangle(\mathcal{T}^G; \widehat{G})}.$$
\end{thm}

\begin{proof}
To see that $\iota$ is well defined, it suffices to  claim that the basic object $\iota(T)$ is tilting in $\mathcal{T}^G$. Indeed, by the adjunction (\ref{equ:adj1}) we have
$${\rm Hom}_{\mathcal{T}^G}({\rm Ind}(T), (\Sigma^G)^n {\rm Ind}(T))\simeq \bigoplus_{h\in G} {\rm Hom}_\mathcal{T}(T, \Sigma^n F_h(T))=0,$$
for each $n\neq 0$, where we use $T\simeq F_h(T)$ and the condition (T1) for $T$. It remains to show that ${\rm thick}\langle {\rm Ind}(T)\rangle =\mathcal{T}^G$. For this end, we consider the full subcategory $$\mathcal{N}=\{X\in \mathcal{T}\; |\; {\rm Ind}(X)\in {\rm thick}\langle {\rm Ind}(T)\rangle\}.$$
Since ${\rm Ind}$ is a triangle functor, it follows that $\mathcal{N}$ is a thick triangulated subcategory, which certainly contains $T$. We conclude that $\mathcal{N}=\mathcal{T}$. Take an arbitrary object $(Y, \alpha)$ in $\mathcal{T}^G$. By Lemma \ref{lem:separable}, $(Y, \alpha)$ is a direct summand of ${\rm Ind}(Y)$, while the latter  lies in ${\rm thick}\langle {\rm Ind}(T)\rangle$ by the previous conclusion. Hence $(Y, \alpha)$ lies in ${\rm thick}\langle {\rm Ind}(T)\rangle$.  Then we are done with (T2) for $\iota(T)$.

By Proposition \ref{prop:stable-bij}, the map $\iota$ is injective. For the surjectivity, let $(Y, \alpha)$ be a basic $\widehat{G}$-stable tilting object in $\mathcal{T}^G$. As in the proof of Proposition \ref{prop:stable-bij}, there is a basic $G$-stable object $Y_0$ such that $\iota(Y_0)=(Y, \alpha)$. We claim that $Y_0$ is a tilting object in $\mathcal{T}$. This will complete the proof of the bijectivity of $\iota$.

By reversing the argument above, we infer that $Y_0$ satisfies (T1). For (T2), we consider the full subcategory $$\mathcal{M}=\{(X, \beta)\in \mathcal{T}^G\; |\; U(X, \beta)=X\in {\rm thick}\langle Y_0\rangle \},$$
which is a thick triangulated subcategory containing $(Y, \alpha)$. We conclude that $\mathcal{M}=\mathcal{T}^G$. Then for an arbitrary object $Z$ in $\mathcal{T}$, ${\rm Ind}(Z)$ lies in $\mathcal{M}$. It follows that $U{\rm Ind}(Z)$ lies in ${\rm thick}\langle Y_0\rangle$, which forces that so does $Z$.  This proves (T2) for $Y_0$ and thus the claim.

For the second bijection, it suffices to claim that $\iota$ is compatible with the equivariantization isomorphism in Theorem \ref{thm:duality-iso-tri}. This is analogous to the last paragraph in the proof of Proposition \ref{prop:stable-bij}. We omit the details.
\end{proof}

\section{Weighted projective lines of tubular type}

In this section, we apply the bijections on stable tilting objects to weighted projective lines of tubular type. The obtained bijection relates the stable tilting complexes on weighted projective lines of  different tubular types.

We assume that $k$ is algebraically closed, whose characteristic is different from $2$ or $3$.

\subsection{Graded automorphisms} We will fix the notation for this section. Let $H$ be an abelian group, which is written additively. Let $R=\oplus_{h\in H}R_h$ be a finitely generated $H$-graded commutative algebra. We denote by ${\rm gr}\mbox{-}R$ the abelian category of finitely generated $H$-graded right $R$-modules. Such modules are denoted by $M=\oplus_{h\in H}M_h$. We denote by ${\rm qgr}\mbox{-}R={\rm gr}\mbox{-}R/{\rm gr}_0\mbox{-}R$ the quotient abelian category, where ${\rm gr}_0\mbox{-}R$ denotes the Serre subcategory formed by finite dimensional modules.

For an element $w\in H$, we denote by $(w)\colon {\rm gr}\mbox{-}R\rightarrow {\rm gr}\mbox{-}R$ the \emph{degree-shift automorphism}. For a graded module $M$, we obtain a shifted module $M(w)=M$,  which is graded by $M(w)_h=M_{h+w}$. The degree-shift automorphism acts on morphisms by the identity. Moreover, $(w)$ descends to an automorphism on ${\rm qgr}\mbox{-}R$, which will be still denoted by $(w)$.

By a \emph{graded automorphism} on $R$, we mean a pair $(g, \psi)$, where $\psi\colon H\rightarrow H$ is a group automorphism and $g\colon R\rightarrow R$ is an algebra automorphism such that $g(R_h)=R_{\psi(h)}$ for each $h\in H$. It induces the \emph{twisting automorphism}  $^g(-)$ on both ${\rm gr}\mbox{-}R$ and ${\rm qgr}\mbox{-}R$, where we suppress $\psi$ in the notation. More precisely, for a graded module $M$, the twisted module ${^gM}=M$ is graded by $(^gM)_h=M_{\psi(h)}$, whose $R$-action is given by $m_\circ r=m.g(r)$ for $r\in R$.

It is well known that weighted projective lines \cite{GL87} of tubular types are of weight types $(2,2,2,2)$, $(3,3,3)$, $(4,4,2)$ and $(6,3,2)$; compare \cite{Rin}. For the type $(2,2,2,2)$, there is a parameter $\lambda\in k$, which is not $0$ or $1$. Following \cite{GL87}, we list their homogeneous coordinate algebras explicitly as follows.
\begin{align*}
S(2,2,2,2; \lambda)&=k[X_1, X_2, X_3, X_4]/(X_3^2-(X_2^2-X_1^2), X_4^2-(X_2^2-\lambda X_1^2));\\
S(3,3,3)&=k[Y_1, Y_2, Y_3]/(Y_3^3-(Y_2^3-Y_1^3));\\
S(4,4,2)&=k[Z_1, Z_2, Z_3]/(Z_3^2-(Z_2^4-Z_1^4));\\
S(6,3,2)&=k[U_1, U_2, U_3]/(U_3^2-(U_2^3-U_1^6)).
\end{align*}
Moreover, we will use letters in the lower case to represent their images in the quotient algebras. For example, $y_i$ will represent the image of $Y_i$ in $S(3,3,3)$.

The algebra $S(2,2,2,2; \lambda)$ is graded by the abelian group $L(2,2,2,2)$, which is generated  by $\vec{x}_1, \vec{x}_2, \vec{x}_3$ and $\vec{x}_4$ subject to the relations $2\vec{x}_1=2\vec{x}_2=2\vec{x}_3=2\vec{x}_4$. More precisely, we have ${\rm deg} \; x_i=\vec{x}_i$. The category ${\rm coh}\mbox{-}\mathbb{X}(2,2,2,2; \lambda)$ of coherent sheaves on the corresponding weighted projective line $\mathbb{X}(2,2,2,2; \lambda)$ is identified with ${\rm qgr}\mbox{-} S(2,2,2,2; \lambda)$.

 Similarly, the algebra $S(3,3,3)$ is graded by the ablian group $L(3,3,3)$, which is generated by $\vec{y}_1, \vec{y}_2$ and $\vec{y}_3$ with the relations $3\vec{y}_1=3\vec{y}_2=3\vec{y}_3$; here, this common value is denoted by $\vec{c}$. The grading is given by ${\rm deg}\;  y_i=\vec{y}_i$. The category ${\rm coh}\mbox{-}\mathbb{X}(3,3,3)$ of coherent sheaves on the corresponding weighted projective line $\mathbb{X}(3,3,3)$ is identified with ${\rm qgr}\mbox{-} S(3,3,3)$. The Auslander-Reiten translation $\tau$ on  ${\rm coh}\mbox{-}\mathbb{X}(3,3,3)$ is given by the degree-shift $(\vec{\omega})$, where $\vec{\omega}=\vec{c}-\vec{y}_1-\vec{y}_2-\vec{y}_3$ is the \emph{dualizing element}. Similar remarks hold for $\mathbb{X}(4,4,2)$ and $\mathbb{X}(6,3,2)$. For example, the Auslander-Reiten translation $\tau$  of ${\rm coh}\mbox{-}\mathbb{X}(4,4,2)$ is given by $(\vec{\omega})$, where $\vec{\omega}=\vec{c}-\vec{z}_1-\vec{z}_2-\vec{z}_3$ in $L(4,4,2)$.

 In what follows, we fix two elements $\sqrt{-1}$ and $\epsilon$ in $k$, where $\epsilon^2-\epsilon+1=0$ is satisfied. In Table 1, we consider the following graded automorphisms, where the automorphisms $\psi_i$'s  on the grading groups are naturally induced by $g_i$'s.

{\renewcommand\arraystretch{1.5}
 \begin{table}[h]
\caption{The graded automorphisms on the coordinate algebras}
\begin{tabular}{|c|c|}
  \hline
  $(g_1, \psi_1)$ on $S(2,2,2,2; -1)$  & $x_1\mapsto \sqrt{-1}x_1, \; x_2\mapsto x_2,\;  x_3\mapsto x_4, \; x_4\mapsto x_3$ \\
  \hline
  $(g_2, \psi_2)$ on $S(2,2,2,2; \epsilon)$  & $x_1\mapsto x_2,\;  x_2\mapsto x_3, \; x_3\mapsto \sqrt{-1}x_1,\;  x_4\mapsto \sqrt{-1}\epsilon x_4$ \\
  \hline
  $(g_3, \psi_3)$ on $S(3,3,3)$  & $y_1\mapsto y_2,\;  y_2\mapsto y_1,\;  y_3\mapsto \epsilon y_3$ \\
    \hline
\end{tabular}
\end{table}}

We denote by $C_2$ and $C_3$ the cyclic groups of order two and three, respectively. We observe that the degree-shift automorphism $^{g_1^2}(-)$ is isomorphic to the identity functor. Moreover, there is a natrual isomorphism
$$c\colon {^{g_1^2}(-)}\longrightarrow {\rm Id}_{{\rm coh}\mbox{-}\mathbb{X}(2,2,2,2; -1)}$$
such that $({^{g_1}(-)}, c)$ is a compatible pair of order two; see Definition \ref{defn:comp}. In more details, for a graded module $M$, the isomorphism
$$c_M\colon {^{g_1^2}(M)}\longrightarrow M$$
is given such that $(c_M)_{\vec{l}}=\gamma(\vec{l}){\rm Id}_{M_{\vec{l}}}$ for each $\vec{l}\in L(2,2,2,2)$. Here, the group homomorphism $\gamma\colon L(2,2,2,2)\rightarrow \{\pm 1\}$ is defined by $\gamma(\vec{x}_1)=-1$ and $\gamma(\vec{x}_2)=\gamma(\vec{x}_3)=\gamma(\vec{x}_4)=1$.

By the compatible pair  $({^{g_1}(-)}, c)$ of order two,  we obtain a $k$-linear $C_2$-action on ${\rm coh}\mbox{-}\mathbb{X}(2,2,2,2; -1)$, which is uniquely \emph{determined} by $(g_1, \psi_1)$; see Corollary \ref{cor:cyclic-act}. Here, we observe that $Z({\rm coh}\mbox{-}\mathbb{X}(2,2,2,2; -1))=k$. So the action is independent of the choice of $c$.

In a similar way, we obtain a $k$-linear $C_3$-action on ${\rm coh}\mbox{-}\mathbb{X}(2,2,2,2; \epsilon)$ determined by $(g_2, \psi_2)$, and a $k$-linear $C_2$-action on ${\rm coh}\mbox{-}\mathbb{X}(3,3,3)$ by $(g_3, \psi_3)$. These actions extend naturally to their bounded derived categories.

\subsection{Relating different tubular types} In this subsection, we use the equivariantization to relate weighted projective lines of different tubular types. We apply Theorems \ref{thm:duality-iso-tri} and \ref{thm:tilting}  to relate certain stable tilting complexes on them.

We explain the relation between $\mathbb{X}(4,4,2)$ and $\mathbb{X}(2,2,2,2;-1)$ in some detail. The dualizing element $\vec{\omega}$ in $L(4,4,2)$ has order $4$. We consider the cyclic group $\mathbb{Z}(2\vec{\omega})$ of order two, which has a strict action on ${\rm coh}\mbox{-}\mathbb{X}(4,4,2)$ by the degree-shift. We will consider the category $({\rm coh}\mbox{-}\mathbb{X}(4,4,2))^{\mathbb{Z}(2\vec{\omega})}$ of $\mathbb{Z}(2\vec{\omega})$-equivariant sheaves.  We identify $C_2$ with the character group $\widehat{\mathbb{Z}(2\vec{\omega})}$. Thus we obtain the dual $C_2$-action on $({\rm coh}\mbox{-}\mathbb{X}(4,4,2))^{\mathbb{Z}(2\vec{\omega})}$; see Subsection 4.3. Recall the $C_2$-action on ${\rm coh}\mbox{-}\mathbb{X}(2,2,2,2;-1)$ determined by the graded automorphism $(g_1, \psi_1)$.

\begin{prop}\label{4,4,2}
Keep the notation as above. Then the following statements hold.
\begin{enumerate}
\item[(1)] There is an equivalence of categories
$${\rm coh}\mbox{-}\mathbb{X}(2,2,2,2; -1)\stackrel{\sim}\longrightarrow ({\rm coh}\mbox{-}\mathbb{X}(4,4,2))^{\mathbb{Z}(2\vec{\omega})},$$
which is equivariant with respect to the above two $C_2$-actions.
    \item[(2)] There is an equivalence of categories
    $$ ({\rm coh}\mbox{-}\mathbb{X}(2,2,2,2; -1))^{C_2} \stackrel{\sim} \longrightarrow {\rm coh}\mbox{-}\mathbb{X}(4,4,2).$$
    \item[(3)] There is a bijection between the sets of isoclasses
   \begin{equation*}
    \left\{
    \begin{aligned}
    &\tau^2\mbox{-stable tilting complexes} \\
    &\qquad  \mbox{ on } \mathbb{X}(4,4,2)
    \end{aligned}
    \right\}
    \stackrel{\iota} \longrightarrow
       \left\{
    \begin{aligned}
    & g_1\mbox{-stable tilting complexes } \\
    & \qquad \mbox{ on } \mathbb{X}(2,2,2,2;-1).
    \end{aligned}
    \right\}
    \end{equation*}
    Moreover, for two $\tau^2$-stable tilting complexes $T_1$ and $T_2$, there is a triangle autoquivalence $F$ on $\mathbf{D}^b({\rm coh}\mbox{-}\mathbb{X}(4,4,2))$ with $F(T_1)\simeq T_2$ if and only if there is  a triangle autoequivalence $F'$ on  $\mathbf{D}^b({\rm coh}\mbox{-}\mathbb{X}(2,2,2,2;-1))$ satisfying that $F'(\iota(T_1))\simeq \iota(T_2)$ and that $F'$ commutes with $^{g_1}(-)$.
\end{enumerate}
\end{prop}

Here, a tilting complex $T$ on $\mathbb{X}(4,4,2)$ is \emph{$\tau^2$-stable} provided that $\tau^2(T)\simeq T$; see \cite{Jas}. This is equivalent to being $\mathbb{Z}(2\vec{\omega})$-stable, since $\tau$ is given by $(\vec{\omega})$. Similarly, a tilting complex $T'$ on $\mathbb{X}(2,2,2,2;-1)$ is \emph{$g_1$-stable} if $T'\simeq {^{g_1}{T'}}$, that is, it is $C_2$-stable with respect to the $C_2$-action determined by $(g_1, \psi_1)$.

\begin{proof}
In this proof, we set $\mathbb{X}=\mathbb{X}(2,2,2,2;-1)$ and $\mathbb{Y}=\mathbb{X}(4,4,2)$. We denote by $\mathcal{O}_\mathbb{X}$ and $\mathcal{O}_\mathbb{Y}$ their structure sheaves.

The equivalence in (1) is obtained explicitly in \cite[Proposition~3.2]{CC17}, which will be denoted by  $G$. We observe that $G(\mathcal{O}_\mathbb{X})={\rm Ind}(\mathcal{O}_\mathbb{Y})$, where ${\rm Ind}\colon {\rm coh}\mbox{-}\mathbb{X}(4,4,2) \rightarrow ({\rm coh}\mbox{-}\mathbb{X}(4,4,2))^{\mathbb{Z}(2\vec{\omega})}$ is the induction functor; consult the proof of \cite[Proposition~2.4]{CC17}. In particular, the dual $C_2$-action fixes ${\rm Ind}(\mathcal{O}_\mathbb{Y})$; see Subsection 4.3.

Transporting the dual $C_2$-action on $({\rm coh}\mbox{-}\mathbb{X}(4,4,2))^{\mathbb{Z}(2\vec{\omega})}$ via a quasi-inverse of $G$, we obtain the $C_2$-action on ${\rm coh}\mbox{-}\mathbb{X}(2,2,2,2; -1)$. It follows that the transported action fixes $\mathcal{O}_\mathbb{X}$, that is, given by an automorphism of $\mathbb{X}$; see \cite[Proposition 3.1]{LM}.

By investigating the action on simple sheaves, it is not hard to see that the transported action coincides with the one determined  by $(g_1, \psi_1)$, proving (1). More precisely, by the same argument, we infer that the transported action fixes the simple sheaves concentrated on the ideals $(x_1)$ and $(x_2)$, respectively. Moreover, it swaps the simple sheaves on $(x_3)$ and $(x_4)$; compare \cite[Proposition 3.1]{LM}.

The statement (2) follows from (1) and  Theorem \ref{thm:duality}. We mention that the equivalences in (1) and (2) extend naturally to their bounded derived categories and the corresponding equivariantizations; see \cite[Proposition 4.5]{Chen15}. Moreover, as the Auslander-Reiten translation, the degree-shift automorphism $(\vec{\omega})$ lies in the center of ${\rm Aut}_\vartriangle(\mathbf{D}^b({\rm coh}\mbox{-}\mathbb{Y}))$. It follows that ${\rm Aut}^{\mathbb{Z}(2\vec{\omega})}_\vartriangle(\mathbf{D}^b({\rm coh}\mbox{-}\mathbb{Y}))={\rm Aut}_\vartriangle(\mathbf{D}^b({\rm coh}\mbox{-}\mathbb{Y}))$.    In view of Remark \ref{rem:surj}, (3) follows immediately from Theorem \ref{thm:tilting}.
\end{proof}

There is a similar relation between $\mathbb{X}(2,2,2,2;\epsilon)$ and $\mathbb{X}(6,3,2)$. The dualizing element $\vec{\omega}=\vec{c}-\vec{u}_1-\vec{u}_2-\vec{u}_3$ in $L(6,3,2)$ has order $6$. The cyclic subgroup $\mathbb{Z}(2\vec{\omega})$ has order $3$, which strictly acts on ${\rm coh}\mbox{-}\mathbb{X}(6,3,2)$ by the degree-shift. We consider the category $({\rm coh}\mbox{-}\mathbb{X}(6,3,2))^{\mathbb{Z}(2\vec{\omega})}$ of $\mathbb{Z}(2\vec{\omega})$-equivariant sheaves. Identifying $C_3$ as the character group $\widehat{\mathbb{Z}(2\vec{\omega})}$, we have the dual $C_3$-action on $({\rm coh}\mbox{-}\mathbb{X}(6,3,2))^{\mathbb{Z}(2\vec{\omega})}$. On the other hand, we have the $C_3$-action on ${\rm coh}\mbox{-}\mathbb{X}(2,2,2,2;\epsilon)$ determined by the graded automorphism $(g_2,\psi_2)$.

The following result is analogous to Proposition \ref{4,4,2}. We omit the details.

\begin{prop}\label{6,3,2 (I)}
Keep the notation as above. Then the following statements hold.
\begin{enumerate}
\item[(1)] There is an equivalence of categories
$${\rm coh}\mbox{-}\mathbb{X}(2,2,2,2; \epsilon)\stackrel{\sim}\longrightarrow ({\rm coh}\mbox{-}\mathbb{X}(6,3,2))^{\mathbb{Z}(2\vec{\omega})},$$
which is equivariant with respect to the above two $C_3$-actions.
    \item[(2)] There is an equivalence of categories
    $$ ({\rm coh}\mbox{-}\mathbb{X}(2,2,2,2; \epsilon))^{C_3} \stackrel{\sim} \longrightarrow {\rm coh}\mbox{-}\mathbb{X}(6,3,2).$$
    \item[(3)] There is a bijection between the sets of isoclasses
      \begin{equation*}
    \left\{
    \begin{aligned}
    &\tau^2\mbox{-stable tilting complexes} \\
    & \qquad \mbox{ on } \mathbb{X}(6,3,2)
    \end{aligned}
    \right\}
    \stackrel{\iota} \longrightarrow
       \left\{
    \begin{aligned}
    & g_2\mbox{-stable tilting complexes } \\
    & \qquad \mbox{ on } \mathbb{X}(2,2,2,2;\epsilon).
    \end{aligned}
    \right\}
    \end{equation*}
    Moreover, for two $\tau^2$-stable tilting complexes $T_1$ and $T_2$, there is a triangle autoquivalence $F$ on $\mathbf{D}^b({\rm coh}\mbox{-}\mathbb{X}(6,3,2))$ with $F(T_1)\simeq T_2$ if and only if there is  a triangle autoequivalence $F'$ on  $\mathbf{D}^b({\rm coh}\mbox{-}\mathbb{X}(2,2,2,2;\epsilon))$ satisfying that $F'(\iota(T_1))\simeq \iota(T_2)$ and that $F'$ commutes with $^{g_2}(-)$. \hfill $\square$
\end{enumerate}
\end{prop}

In $L(6,3,2)$, the cyclic group $\mathbb{Z}(3\vec{\omega})$ has order two. Identifying $C_2$ as the character group $\widehat{\mathbb{Z}(3\vec{\omega})}$, we obtain the dual $C_2$-action on $({\rm coh}\mbox{-}\mathbb{X}(6,3,2))^{\mathbb{Z}(3\vec{\omega})}$. On the other hand, we have the $C_2$-action on ${\rm coh}\mbox{-}\mathbb{X}(3,3,3)$ determined by the graded automorphism $(g_3, \psi_3)$.

\begin{prop}\label{6,3,2 (II)}
Keep the notation as above. Then the following statements hold.
\begin{enumerate}
\item[(1)] There is an equivalence of categories
$${\rm coh}\mbox{-}\mathbb{X}(3,3,3)\stackrel{\sim}\longrightarrow ({\rm coh}\mbox{-}\mathbb{X}(6,3,2))^{\mathbb{Z}(3\vec{\omega})},$$
which is equivariant with respect to the above two $C_2$-actions.
    \item[(2)] There is an equivalence of categories
    $$ ({\rm coh}\mbox{-}\mathbb{X}(3,3,3))^{C_2} \stackrel{\sim} \longrightarrow {\rm coh}\mbox{-}\mathbb{X}(6,3,2).$$
    \item[(3)] There is a bijection between the sets of isoclasses
      \begin{equation*}
    \left\{
    \begin{aligned}
    &\tau^3\mbox{-stable tilting complexes} \\
    & \qquad \mbox{ on } \mathbb{X}(6,3,2)
    \end{aligned}
    \right\}
    \stackrel{\iota} \longrightarrow
       \left\{
    \begin{aligned}
    & g_3\mbox{-stable tilting complexes } \\
    & \qquad \mbox{ on } \mathbb{X}(3,3,3).
    \end{aligned}
    \right\}
    \end{equation*}
     Moreover, for two $\tau^3$-stable tilting complexes $T_1$ and $T_2$, there is a triangle autoquivalence $F$ on $\mathbf{D}^b({\rm coh}\mbox{-}\mathbb{X}(6,3,2))$ with $F(T_1)\simeq T_2$ if and only if there is  a triangle autoequivalence $F'$ on  $\mathbf{D}^b({\rm coh}\mbox{-}\mathbb{X}(3,3,3))$ satisfying that $F'(\iota(T_1))\simeq \iota(T_2)$ and that $F'$ commutes with $^{g_3}(-)$. \hfill $\square$
\end{enumerate}
\end{prop}

\begin{rem}
(1) In Propositions \ref{4,4,2}-\ref{6,3,2 (II)}, the equivalences in (2) might be deduced from a general result \cite[Proposition 3]{Len16}. Here, we emphasize that they are dual to the equivalences in (1). This duality allows us to apply Theorems \ref{thm:duality-iso-tri} and \ref{thm:tilting}.

(2) Propositions \ref{4,4,2} and \ref{6,3,2 (I)} relate the classification of $\tau^2$-stable tilting complexes in \cite{Jas} to that of $g_i$-stable tilting complexes on weighted projective lines of weight type $(2,2,2,2)$. Such a relation is implicitly indicated in \cite[p.30]{Jas} and \cite[Remark~3.9(3)]{CC17}.

Indeed, based on the classification of tilting complexes for the tubular type $(2,2,2,2)$ in \cite[Chapter 10]{Mel} (see also \cite[Example 3.3]{Sko} and  \cite[Figure 1]{BPe}),  these relation might provide a different approach to the classification of $\tau^2$-stable tilting complexes. In view of Corollary \ref{cor:cyclc-cross}, the latter classification boils down to certain elements in the outer automorphism groups of algebras. Proposition \ref{6,3,2 (II)} indicates that the classification of $\tau^3$-stable tilting complexes on $\mathbb{X}(6,3,2)$ might also be of interest. \hfill $\square$
\end{rem}

\section{Forgetful and obstruction homomorphisms}

In this section, we investigate when the forgetful homomorphism  (\ref{equ:forg}) is surjective. It turns out that the obstruction homomorphism from the autoequivalence group to the second cohomological group plays a role.

Throughout, we work over a fixed field $k$. Let $\mathcal{C}$ be a $k$-linear category such that its center $Z(\mathcal{C})=k$.  Let $G$ be a group with a fixed $k$-linear $G$-action $\{F_g, \varepsilon_{g, h}|\; g, h\in G\}$ on $\mathcal{C}$.

\subsection{The obstruction homomorphism}  Take a $k$-linear autoequivalence $F\colon \mathcal{C}\rightarrow \mathcal{C}$ with $[F]\in {\rm Aut}^G(\mathcal{C})$, that is, $FF_g$ is isomorphic to $F_g F$ for each $g\in G$. In general, this does not give rise to a $G$-equivariant functor. We choose a natural isomorphism $\delta_g\colon FF_g\rightarrow F_gF$ for each $g\in G$.

For any  $g, h\in G$, we claim that there is a unique nonzero scalar $\sigma_F(g, h)\in k^*$ satisfying
\begin{align}\label{equ:sigma}
\delta_{gh}\circ F\varepsilon_{g, h}=\sigma_F(g, h)  \varepsilon_{g, h}F\circ F_g\delta_h\circ \delta_gF_h.
\end{align}
Indeed, we apply Lemma \ref{lem:nat} to the automorphism $\delta_{gh}\circ F\varepsilon_{g, h}\circ (\varepsilon_{g, h}F\circ F_g\delta_h\circ \delta_gF_h)^{-1}$ of the autoequivalence $F_{gh}F\colon \mathcal{C}\rightarrow \mathcal{C}$.

The following result implies that $\sigma_F$ defines a $2$-cocycle of $G$ with values in $k^*$.

\begin{lem}
Keep the notation and assumptions as above. Then we have
$$\sigma_F(gh, f) \sigma_F(g, h) = \sigma_F(g, hf)\sigma_F(h, f) $$
for any $g, h, f\in G$.
\end{lem}

\begin{proof}
By (\ref{equ:2-coc}) we have $\delta_{ghf}\circ F\varepsilon_{gh, f}\circ F\varepsilon_{g, h}F_f=\delta_{ghf}\circ F\varepsilon_{g, hf}\circ FF_g\varepsilon_{h, f}$. Applying (\ref{equ:sigma}) twice to the left hand side, we obtain $$\sigma_F(gh, f)\sigma_F(g, h) (\varepsilon_{gh, f} F\circ \varepsilon_{g, h}F_fF)\circ (F_gF_h\delta_f\circ F_g\delta_hF_f\circ \delta_gF_hF_f).$$ Similarly, we have that the right hand side equals $$\sigma_F(g, hf)\sigma_F(h, f) (\varepsilon_{g, hf}F\circ F_g\varepsilon_{h,f}F)\circ (F_gF_h\delta_f\circ F_g\delta_hF_f\circ \delta_gF_hF_f).$$
Applying (\ref{equ:2-coc}) again, we infer the required identity.
\end{proof}

The $2$-cocycle $\sigma_F$ depends on the choice of the isomorphisms $\delta_g$'s. Take another set of isomorphisms $\delta'_g\colon FF_g\rightarrow F_gF$, which yields another $2$-cocycle $\sigma'_F$. By Lemma~\ref{lem:nat} there is a unique $\lambda(g)\in k^*$ with $\delta'_g=\lambda(g)\delta_g$ for each $g\in G$. By comparing (\ref{equ:sigma}) for $\sigma_F(g, h)$ and for $\sigma'_F(g, h)$, we infer that
\begin{align}\label{equ:sigma-2}
\sigma'_F(g, h)=\sigma_F(g, h) \lambda(gh)\lambda(g)^{-1}\lambda(h)^{-1}.
\end{align}
It follows that the cohomological class $[\sigma_F]\in H^2(G, k^*)$ is independent of the choice of the isomorphisms $\delta_g$'s. Moreover, the class $[\sigma_F]$ is trivial if and only if $F$ lifts to a $G$-equivariant endofunctor.

\begin{lem}
Keep the notation and assumptions as above. Then the following two statements hold.
\begin{enumerate}
\item Assume that $\gamma \colon F\rightarrow F'$ is a natural isomorphism between two $k$-linear autoequivalences. Then $[\sigma_F]=[\sigma_{F'}]$.
\item For two $k$-linear autoequivalences $F_1, F_2$ on $\mathcal{C}$, we have $[\sigma_{F_1F_2}]=[\sigma_{F_1}\sigma_{F_2}]$.
\end{enumerate}
\end{lem}

\begin{proof}
(1) Take natural isomorphisms $\delta_g\colon FF_g\rightarrow F_gF$ and set $$\delta'_g=F_g\gamma\circ \delta_g\circ \gamma^{-1}F_g\colon F'F_g\longrightarrow F_gF'.$$ It follows that
\begin{align*}
\delta'_{gh}\circ F'\varepsilon_{g, h}&=\sigma_F(g, h)F_{gh}\gamma \circ \varepsilon_{g, h}F\circ F_g\delta_h\circ \delta_gF_h\circ \gamma^{-1}F_gF_h,\\
&=\sigma_F(g, h)\varepsilon_{g, h}F'\circ F_g\delta'_h\circ \delta'_gF_h. \end{align*}
Then we infer that $\sigma_{F'}(g, h)=\sigma_{F}(g, h)$.

(2) Write $F=F_1F_2$. For each $g\in G$, we take natural isomorphisms $\delta_g^1\colon F_1 F_g\rightarrow F_gF_1$ and $\delta_g^2\colon F_2F_g\rightarrow F_gF_2$. Set $\delta_g=\delta_g^1F_2\circ F_1\delta_g^2\colon FF_g\rightarrow F_gF$. It follows by direct calculation that $\sigma_F(g,h)=\sigma_{F_1}(g,h)\sigma_{F_2}(g,h)$.
\end{proof}

The above lemma implies that the following group homomorphism is well defined
$$\sigma\colon {\rm Aut}^G(\mathcal{C})\longrightarrow H^2(G, k^*), \; [F]\mapsto [\sigma_F].$$
We will call $\sigma$ the \emph{obstruction homomorphism} of the given $G$-action.

We observe a group homomorphism
$$\rho\colon \widehat{G}\longrightarrow {\rm Aut}(\mathcal{C}; G)$$
  sending $\chi$ to the $G$-equivariant functor $[{\rm Id}_\mathcal{C}, (\chi(g)F_g)_{g\in G}]$; compare Subsection 5.2.

The following result implies that the forgetful homomorphism (\ref{equ:forg})  is surjective if and only if $\sigma$ is trivial. We mention that the result extends \cite[Theorem 6(1)]{Pl}.

\begin{prop}\label{prop:longex}
Let $\{F_g, \varepsilon_{g, h}|\; g, h\in G\}$ be a $k$-linear $G$-action on $\mathcal{C}$. Assume that $Z(\mathcal{C})=k$. Then there is an exact sequence of groups
\begin{align}
  1\longrightarrow \widehat{G}\stackrel{\rho}\longrightarrow {\rm Aut}(\mathcal{C}; G)\stackrel{\phi}\longrightarrow {\rm Aut}^G(\mathcal{C})\stackrel{\sigma}\longrightarrow H^2(G, k^*),
\end{align}
where $\phi$ is the forgetful homomorphism and $\sigma$ is the obstruction homomorphism of the given $G$-action, respectively.
\end{prop}

\begin{proof}
The homomorphism $\rho$ is clearly injective. On the other hand, if a $G$-equivariant endofunctor lies in the kernel of $\phi$, then it is isomorphic to $({\rm Id}_\mathcal{C}, (\gamma_g)_{g\in G})$. Now applying Lemma \ref{lem:nat}, we infer that $\gamma_g=\chi(g)F_g$ for some character $\chi$. This proves the exactness at ${\rm Aut}(\mathcal{C}; G)$.

As mentioned above, the class $[\sigma_F]$ is trivial if and only if $(F, (\delta_g)_{g\in G})$ is a $G$-equivariant functor for some natural isomorphisms $\delta_g$'s. Then the exactness at ${\rm Aut}^G(\mathcal{C})$ follows.
\end{proof}

\subsection{An isomorphism of groups} In this subsection, we assume that the  group $G$ is finite abelian, and splits over $k$. We assume further that $Z(\mathcal{C}^G)=k$. Applying Proposition \ref{prop:longex} to the dual $\widehat{G}$-action on $\mathcal{C}^G$, we obtain an exact sequence
\begin{align}
  1\longrightarrow G \stackrel{\rho} \longrightarrow {\rm Aut}(\mathcal{C}^G; \widehat{G}) \stackrel{\widehat{\phi}} \longrightarrow {\rm Aut}^{\widehat{G}}(\mathcal{C}^G) \stackrel{\sigma}\longrightarrow H^2(\widehat{G}, k^*).
\end{align}
Here, we identify by (\ref{equ:ev}) $G$ with the character group of $\widehat{G}$, and $\widehat{\phi}$ denotes the forgetful homomorphism.

From the given $G$-action, we have the group homomorphism $G\rightarrow {\rm Aut}(\mathcal{C})$ sending $g$ to $[F_g]$, whose image lies in ${\rm Aut}^G(\mathcal{C})$. Moreover, its image is a normal subgroup of ${\rm Aut}^G(\mathcal{C})$. The corresponding quotient group is denoted by ${\rm Aut}^G(\mathcal{C})/G$, where we abuse $G$ with its image.  For the dual $\widehat{G}$-action on $\mathcal{C}^G$, we have the quotient group ${\rm Aut}^{\widehat{G}}(\mathcal{C}^G)/{\widehat{G}}$.

\begin{prop}\label{prop:ano-iso-group}
Let $G$ be a finite abelian group, which splits over $k$. Assume that $\mathcal{C}$ is idempotent complete satisfying that $Z(\mathcal{C})=k=Z(\mathcal{C}^G)$. We assume further that the forgetful homomorphisms $\phi$ and $\widehat{\phi}$ are both surjective.  Then there is a unique isomorphism ${\rm Aut}^G(\mathcal{C})/G \stackrel{\sim}\longrightarrow {\rm Aut}^{\widehat{G}}(\mathcal{C}^G)/{\widehat{G}}$ of groups, which fills into the following commutative diagram
\[\xymatrix{
{\rm Aut}(\mathcal{C}; G) \ar[d] \ar[rr]^-{(-)^G} && {\rm Aut}(\mathcal{C}^G; \widehat{G}) \ar[d]\\
 {\rm Aut}^G(\mathcal{C})/G\ar@{.>}[rr]^-{\sim} && {\rm Aut}^{\widehat{G}}(\mathcal{C}^G)/{\widehat{G}},
}\]
where the upper row is the equivariantization isomorphism in Theorem \ref{thm:duality-iso}, and the vertical homomorphisms are induced by the forgetful homomorphisms.
\end{prop}

\begin{proof}
By the very definition of ${\rm Act}(\mathcal{C}; G)$ in Subsection 5.1, the forgetful homomorphism  induces an isomorphism
$${\rm Aut}(\mathcal{C}; G)/{{\rm Act}(\mathcal{C}; G)}\stackrel{\sim}\longrightarrow {\rm Aut}^G(\mathcal{C})/G.$$
The same argument yields an isomorphism
$${\rm Aut}(\mathcal{C}^G; \widehat{G})/{{\rm Act}(\mathcal{C}^G; \widehat{G})}\stackrel{\sim}\longrightarrow {\rm Aut}^{\widehat{G}}(\mathcal{C}^G)/{\widehat{G}}.$$
Now the required isomorphism follows  Theorem \ref{thm:duality-iso} and Proposition \ref{prop:Act}.
\end{proof}

\begin{rem}
(1) Let us mention a simple case where Proposition \ref{prop:ano-iso-group} applies. Let $G$ be a cyclic group of order $d$ with $d\geq 2$. We assume that the field $k$ has a primitive $d$-th root of unity, and each element has a $d$-th root. Then $G$ splits over $k$, and $H^2(G, k^*)=0=H^2(\widehat{G}, k^*)$. We obtain an isomorphism between ${\rm Aut}^G(\mathcal{C})/G$ and ${\rm Aut}^{\widehat{G}}(\mathcal{C}^G)/{\widehat{G}}$. Take $F=F_g$ for a generator $g$ of $G$. Then ${\rm Aut}^G(\mathcal{C})$ is just the centralizer of $[F]$ in ${\rm Aut}(\mathcal{C})$.

(2) There is a triangle version of Proposition \ref{prop:ano-iso-group}. More precisely, let $\mathcal{T}$ be a pre-triangulated category with a triangle $G$-action such that $Z_\vartriangle(\mathcal{T})=k=Z_\vartriangle(\mathcal{T}^G)$. Assume that the forgetful homomorphisms $\phi\colon {\rm Aut}_\vartriangle(\mathcal{T}; G) \rightarrow {\rm Aut}_\vartriangle^G(\mathcal{T})$ and $\widehat{\phi}\colon {\rm Aut}_\vartriangle(\mathcal{T}^G; \widehat{G})\rightarrow {\rm Aut}_\vartriangle^{\widehat{G}}(\mathcal{T}^G)$ are both surjective. Then we have an isomorphism
$${\rm Aut}_\vartriangle^G(\mathcal{T})/G \stackrel{\sim}\longrightarrow {\rm Aut}_\vartriangle^{\widehat{G}}(\mathcal{T}^G)/{\widehat{G}},$$
which is compatible with the equivariantization isomorphism in Theorem \ref{thm:duality-iso-tri}. \hfill $\square$
\end{rem}

\begin{appendix}

\section{Identities for group actions}

We collect some identities for an arbitrary group action, and provide complete proofs.

Let  $G$ be a group, which is written multiplicatively and whose unit is denoted by $e$.  Let $\mathcal{C}$ be a category. We denote by ${\rm Id}_\mathcal{C}\colon \mathcal{C}\rightarrow \mathcal{C}$ the identity endofunctor.

For two functors $F\colon \mathcal{C}\rightarrow \mathcal{C}'$ and $F'\colon \mathcal{C}'\rightarrow \mathcal{C}''$, their composition is denoted by $F'F\colon \mathcal{C}\rightarrow \mathcal{C}''$. Let $F'_1\colon \mathcal{C}'\rightarrow \mathcal{C}''$  and $F'_2\colon \mathcal{C}'\rightarrow \mathcal{C}''$ be two functors.  We denote by ${\rm Nat}(F', F_1')$ the class of natural transformations from $F'$ to $F_1'$. For two natural transformations $\eta\colon F'\rightarrow F'_1$ and $\delta \colon F_1'\rightarrow F'_2$, we denote by $\delta\circ \eta\colon F'\rightarrow F'_2$ their composition. Let $\eta\colon F'\rightarrow F'_1$ be a natural transformation and $F''\colon \mathcal{C}''\rightarrow \mathcal{C}'''$ be a functor. We denote by $F''\eta F\colon F''F'F\rightarrow F''F'_1F$  the induced natural transformation.

\begin{lem}\label{lemA:rule}
Let $\eta\colon F_1\rightarrow F_2$ be a natural transformation between functors from $\mathcal{C}$ to $\mathcal{C}'$, and let $\eta'\colon F_1'\rightarrow F'_2$ be a natural transformation between functors from  $\mathcal{C}'$ to $\mathcal{C}''$. Then we have $\eta'F_2\circ F'_1\eta =F'_2\eta \circ \eta' F_1$.
\end{lem}

\begin{proof}
The identity follows from the naturality of $\eta'$.
\end{proof}

The following fact is standard.

\begin{lem}\label{lemA:Nat}
Assume that $F\colon \mathcal{C}\rightarrow \mathcal{C}'$  and $F''\colon \mathcal{C}''\rightarrow \mathcal{C}'''$  are two equivalences of categories. Then the map ${\rm Nat}(F', F_1')\rightarrow {\rm Nat}(F''F'F, F''F_1'F)$ sending $\eta$ to $F''\eta F$ is bijective. Moreover, $\eta$ is a natural isomorphism if and only if so is $F''\eta F$. \hfill $\square$
\end{lem}

Recall that a \emph{$G$-action} $\{F_g, \varepsilon_{g, h}|\; g, h \in G\}$ on $\mathcal{C}$ consists of an  autoequivalence $F_g\colon \mathcal{C}\rightarrow \mathcal{C}$ for each $g\in G$, and a natural isomorphism  $\varepsilon_{g, h}\colon F_gF_h\rightarrow F_{gh}$ for each pair $g, h$ of elements in $G$, which are subject to the following conditions
\begin{align}\label{equA:2-cocycle}
\varepsilon_{gh, k}\circ \varepsilon_{g, h}F_k=\varepsilon_{g, hk}\circ F_g\varepsilon_{h, k}
\end{align}
for all $g, h, k$ in $G$.

We deduce from Lemma \ref{lemA:Nat} that there is a unique isomorphism $u\colon F_e\rightarrow {\rm Id}_\mathcal{C}$ such that $F_e u=\varepsilon_{e,e}$. We call $u$ the \emph{unit} of this $G$-action.

\begin{lem}\label{lemA:G-action}
For a given $G$-action $\{F_g, \varepsilon_{g, h}|\; g, h \in G\}$ on $\mathcal{C}$, the following statements hold.
\begin{enumerate}
\item For each pair $g, k$ in $G$, we have $\varepsilon_{g, e}F_k=F_g\varepsilon_{e, k}$.
\item For each $k\in G$, we have $\varepsilon_{e, k}=uF_k$. In particular, we have $\varepsilon_{e, e}=uF_e$.
\item For each $g\in G$, we have $\varepsilon_{g,e}=F_g u$.
\end{enumerate}
\end{lem}

\begin{proof}
We take $h=e$ in (\ref{equA:2-cocycle}) and notice that $\varepsilon_{g, k}$ is an isomorphism. Then we infer (1). Taking $g=e$ in (1), we have $F_euF_k=\varepsilon_{e,e}F_k=F_e\varepsilon_{e, k}$.  Applying Lemma \ref{lemA:Nat}, we infer (2). Taking $k=e$ in (1) and applying (2), we have $\varepsilon_{g, e}F_e=F_g\varepsilon_{e, e}=F_guF_e$. We now deduce (3) from Lemma \ref{lemA:Nat}.
\end{proof}

We will describe an extension of (\ref{equA:2-cocycle}).  For $n\geq 3$ and $g_1, g_2,\cdots, g_n\in G$, we define inductively a natural isomorphism
\begin{align}
\varepsilon_{g_1, g_2, \cdots, g_n}\colon F_{g_1}F_{g_2}\cdots F_{g_n}\longrightarrow F_{g_1g_2\cdots g_n}
\end{align}
by $\varepsilon_{g_1, g_2, \cdots, g_n}=\varepsilon_{g_1\cdots g_{n-1}, g_n}\circ \varepsilon_{g_1, g_2, \cdots, g_{n-1}}F_{g_n}$. In particular, we have the following isomorphism
\begin{align}\label{equA:varepsilon}
\varepsilon_g^{(n)}=\varepsilon_{g, g, \cdots, g}\colon F_g^n=F_g F_g \cdots F_g \longrightarrow F_{g^n}.
\end{align}
Moreover, we define $\varepsilon_{g}^{(2)}=\varepsilon_{g, g}$ and $\varepsilon_{g}^{(1)}={\rm Id}_{F_g}$. By convention, we set $\varepsilon_{g}^{(0)}=u^{-1}$, where $u$ is the unit of the $G$-action and $F_g^0={\rm Id}_\mathcal{C}$.

Recall from \cite[Section 4]{DGNO} that a $G$-action on $\mathcal{C}$ is the same as  a monoidal functor from $\underline{G}$ to the category of endofunctors on $\mathcal{C}$. Here,  $\underline{G}$ is the category whose objects are elements in $G$ and whose morphisms are just the identities on objects; it  has a canonical monoidal structure induced from the multiplication of $G$. Then the following proposition might be deduced from the coherence theorem \cite[Theorem~1.6]{Eps}; compare \cite[Theorem~5.4]{Shi}.

\begin{prop}\label{propA:gen-2coc}
Let $n\geq 3$, $m\geq 1$ and $i\geq 2$ satisfying $m+i\leq n$. Then we have
{\small \begin{align*}
\varepsilon_{g_1, g_2, \cdots, g_n}= \varepsilon_{g_1, \cdots, g_m, g_{m+1}\cdots g_{m+i}, g_{m+i+1}, \cdots, g_n}\circ F_{g_1}\cdots F_{g_m}\varepsilon_{g_{m+1}, \cdots, g_{m+i}} F_{g_{m+i+1}} \cdots F_{g_n}.
\end{align*}}
\end{prop}

\vskip 5pt

\begin{proof}
We use induction on $n$. The case that $n=3$ is due to (\ref{equA:2-cocycle}). Assume that $n\geq 4$. We divide into three cases.

For the first case, we assume that $m+i<n$, that is, $F_{g_n}$ does appear in the right hand side. Then we are done by the following identity:
{\tiny \begin{align*}
&\varepsilon_{g_1, \cdots, g_m, g_{m+1}\cdots g_{m+i}, g_{m+i+1}, \cdots, g_n}\circ F_{g_1}\cdots F_{g_m}\varepsilon_{g_{m+1}, \cdots, g_{m+i}} F_{g_{m+i+1}} \cdots F_{g_n}\\
= &\varepsilon_{g_1\cdots g_{n-1}, g_n} \circ \varepsilon_{g_1,\cdots ,g_m, g_{m+1}\cdots g_{m+i}, g_{m+i+1}, \cdots, g_{n-1}} F_{g_n}\circ F_{g_1}\cdots F_{g_m}\varepsilon_{g_{m+1}, \cdots, g_{m+i}} F_{g_{m+i+1}} \cdots F_{g_n}\\
=&\varepsilon_{g_1\cdots g_{n-1}, g_n}\circ \varepsilon_{g_1, \cdots, g_{n-1}}F_{g_n}\\
=&\varepsilon_{g_1, g_2,\cdots g_n}.
\end{align*}}Here, the second equality uses the induction hypothesis.

For the second case, we assume that $m+i=n$ and $m\geq 2$. Then we are done by the following identity:
 \begin{align*}
&\varepsilon_{g_1, \cdots, g_{m}, g_{m+1}\cdots g_n}\circ F_{g_1}\cdots F_{g_m}\varepsilon_{g_{m+1}, \cdots, g_n}\\
=&\varepsilon_{g_1\cdots g_m, g_{m+1}\cdots g_n}\circ \varepsilon_{g_1, \cdots, g_m}F_{g_{m+1}\cdots g_n}\circ F_{g_1}\cdots F_{g_m}\varepsilon_{g_{m+1}, \cdots, g_n}\\
=&\varepsilon_{g_1\cdots g_m, g_{m+1}\cdots g_n}\circ F_{g_1\cdots g_m}\varepsilon_{g_{m+1}, \cdots, g_n}\circ \varepsilon_{g_1,\cdots, g_m}F_{g_{m+1}}\cdots F_{g_n}\\
=&\varepsilon_{g_1\cdots g_m, g_{m+1}, \cdots, g_n} \circ \varepsilon_{g_1,\cdots, g_m}F_{g_{m+1}}\cdots F_{g_n}\\
=& \varepsilon_{g_1, g_2, \cdots, g_n}.
\end{align*}
Here, the second equality uses Lemma \ref{lemA:rule}, and the third uses the fact that $m\geq 2$ and then the induction hypothesis. The last equality uses the first case.

For the final case, we assume that $m+i=n$ and $m=1$. By definition we have $\varepsilon_{g_2, g_3, \cdots, g_n}=\varepsilon_{g_2\cdots g_{n-1}, g_n}\circ \varepsilon_{g_2, \cdots g_{n-1}} F_{g_n}$. Then we are done by the following identity:
\begin{align*}
\varepsilon_{g_1, g_2\cdots g_n} \circ F_{g_1}\varepsilon_{g_2, g_3, \cdots, g_n}& =  \varepsilon_{g_1, g_2\cdots g_n} \circ F_{g_1} \varepsilon_{g_2\cdots g_{n-1}, g_n}\circ F_{g_1}\varepsilon_{g_2, \cdots g_{n-1}} F_{g_n}\\
&= \varepsilon_{g_1\cdots g_{n-1}, g_n}\circ \varepsilon_{g_1,g_2\cdots g_{n-1}}F_{g_n} \circ F_{g_1}\varepsilon_{g_2, \cdots g_{n-1}} F_{g_n}\\
&= \varepsilon_{g_1\cdots g_{n-1}, g_n}\circ \varepsilon_{g_1, \cdots, g_{n-1}}F_{g_n}\\
&=\varepsilon_{g_1, g_2, \cdots, g_n}.
\end{align*}
Here, the second equality uses (\ref{equA:2-cocycle}) and the third uses the induction hypothesis. We are done.
\end{proof}

\begin{rem}
Hideto Asashiba observes that a $G$-action on a small category $\mathcal{C}$ coincides with a pseudo-functor $X$ from $\mathcal{G}$ to the $2$-category formed by small categories such that $X(\ast)=\mathcal{C}$; compare \cite[p.135 and Definitions 2.1 and 6.1(3)]{Asa13}. Here, $\mathcal{G}$ denotes the category with only one object $\ast$ and the Hom-set $G$. For pseudo-functors, we refer to \cite[Section 7.5]{Bor}. Then the above result also follows from the general result in \cite{Pow}. \hfill $\square$
\end{rem}

\begin{lem}\label{lemA:twocases}
Let $g\in G$, $i, j\geq 0$ and $d\geq 1$. The following statements hold.
\begin{enumerate}
\item We have $\varepsilon_g^{(i+j)}=\varepsilon_{g^i, g^j}\circ F_{g^i}\varepsilon_g^{(j)}\circ \varepsilon_g^{(i)}F_g^j.$
\item Assume that $g^d=e$ and that $i+j\geq d$. Then we have $$\varepsilon_g^{(i+j-d)}\circ F_g^{i+j-d}u\circ F_g^{i+j-d}\varepsilon_g^{(d)}=\varepsilon_{g^i, g^j}\circ F_{g^i}\varepsilon_g^{(j)}\circ \varepsilon_g^{(i)}F_g^j.$$
\end{enumerate}
\end{lem}

\begin{proof}
(1) follows by applying Proposition \ref{propA:gen-2coc} twice. We mention that if $i=0$ or $j=0$, one needs  Lemma \ref{lemA:G-action}(2) and (3).

For (2), we have the following identity
\begin{align*}
&\varepsilon^{(i+j-d)}_g\circ {F_g}^{i+j-d}u\circ F_g^{i+j-d}\varepsilon_g^{(d)}\\
  =& {F}_{g^{i+j-d}} u\circ \varepsilon^{(i+j-d)}_g{F_{g^d}}\circ {F}_g^{i+j-d}{\varepsilon}_{g}^{(d)}\\
  =& {\varepsilon}_{g^{i+j-d}, g^d} \circ \varepsilon^{(i+j-d)}_g{F_{g^d}}\circ {F}_g^{i+j-d}{\varepsilon}_{g}^{(d)}\\
  = & \varepsilon^{(i+j)}_g.
 \end{align*}
 Here,  the first equality uses Lemma \ref{lemA:rule} and the assumption that $g^d=e$,  the second uses Lemma \ref{lemA:G-action}(3), and the last one is obtained by applying  Proposition \ref{propA:gen-2coc} twice. Now the required identity follows by combining the above identity with (1).
\end{proof}

\section{Strongly $\mathbf{K}$-standard categories}

We recall from \cite{CY} basic facts on strongly $\mathbf{K}$-standard additive categories, which plays a subtle role in the study of stable tilting objects in Section 7.

 Throughout, let $k$ be a field. All categories and functors are assumed to be $k$-linear.

Let $\mathcal{A}$ be a $k$-linear additive category. Denote by $\mathbf{K}^b(\mathcal{A})$ the bounded homotopy category, whose translation is denoted by $\Sigma$. For each object $M$ in $\mathcal{A}$ and integer $n$, we denote  by $\Sigma^n(M)$ the stalk complex concentrated on degree $-n$. Those complexes form a full subcategory  $\Sigma^n(\mathcal{A})$ of $\mathbf{K}^b(\mathcal{A})$. In particular, $\mathcal{A}$ is identified with $\Sigma^0(\mathcal{A})$.

For an additive endofunctor $F$ on $\mathcal{A}$, we denote by $\mathbf{K}^b(F)\colon \mathbf{K}^b(\mathcal{A})\rightarrow \mathbf{K}^b(\mathcal{A})$ its natural extension on complexes, which is a triangle functor with a trivial connecting isomorphism. Similarly, a natural transformation $\eta\colon F\rightarrow F'$ extends to a natural transformation $\mathbf{K}^b(\eta)\colon\mathbf{K}^b(F) \rightarrow \mathbf{K}^b(F') $ between triangle functors.

Recall that $Z(\mathcal{A})$ is the center of $\mathcal{A}$ and $Z_\vartriangle(\mathbf{K}^b(\mathcal{A}))$ is the triangle center of $\mathbf{K}^b(\mathcal{A})$. There is an algebra homomorphism
\begin{align}\label{equ:res}
Z_\vartriangle(\mathbf{K}^b(\mathcal{A}))\longrightarrow Z(\mathcal{A}), \quad \lambda\mapsto \lambda|_\mathcal{A},
\end{align}
where $\lambda|_\mathcal{A}$ denotes the restriction. This homomorphism admits a section, which sends $\mu\in Z(\mathcal{A})$ to $\mathbf{K}^b(\mu)$.

The following notions are introduced  in \cite[Sections 3 and 4]{CY}. By \cite[Lemmas~4.2 and 4.3]{CY}, the definition in (2) is equivalent to the original one \cite[Definition~4.1]{CY}.

\begin{defn}\label{defn:sstand}
(1) Let $(F, \omega)\colon \mathbf{K}^b(\mathcal{A})\rightarrow \mathbf{K}^b(\mathcal{A})$ be a triangle endofunctor. We say that $(F, \omega)$ is a \emph{pseudo-identity} provided that $F(X)=X$ for each complex $X$ and its restriction $F|_{\Sigma^n(\mathcal{A})}\colon \Sigma^n(\mathcal{A})\rightarrow \Sigma^n(\mathcal{A})$ equals the identity for each $n$.
(2) The additive category $\mathcal{A}$ is \emph{$\mathbf{K}$-standard}, provided that each pseudo-identity on $\mathbf{K}^b(\mathcal{A})$ is isomorphic to ${\rm Id}_{\mathbf{K}^b(\mathcal{A})}$, the genuine identity functor. If in addition, the homomorphism (\ref{equ:res}) is injective, then $\mathcal{A}$ is called \emph{strongly $\mathbf{K}$-standard}. \hfill $\square$
\end{defn}

 The following observation is due to \cite[Lemma 4.4]{CY}.

\begin{lem}\label{lem:appB}
Let $\mathcal{A}$ be a $\mathbf{K}$-standard category with two autoequivalences $F_1$ and $F_2$. The following statements hold.
\begin{enumerate}
\item[(1)] Assume that $(F, \omega)$ is a triangle autoequivalence on $\mathbf{K}^b(\mathcal{A})$ satisfying $F(\mathcal{A})\subseteq \mathcal{A}$. If $\mathcal{A}$ is idempotent complete,  then there is an isomorphism $(F, \omega)\simeq \mathbf{K}^b(F|_\mathcal{A})$ of triangle functors.
    \item[(2)] Assume further that $\mathcal{A}$ is strongly $\mathbf{K}$-standard. Then any natural isomorphism $\eta\colon \mathbf{K}^b(F_1)\rightarrow \mathbf{K}^b(F_2)$ is equal to $\mathbf{K}^b(\eta|_\mathcal{A})$. \hfill $\square$
\end{enumerate}
\end{lem}

The following is an immediate consequence.

\begin{cor}
Assume that $\mathcal{A}$ is strongly $\mathbf{K}$-standard and idempotent complete. Let $G$ be a group. Then there is a bijection between the sets of isoclasses
$$\{G\mbox{-actions on }\mathcal{A} \}{/\simeq} \; \longleftrightarrow  \;  \{\mbox{triangle } G\mbox{-actions on } \mathbf{K}^b(\mathcal{A}) \mbox{ fixing } \mathcal{A}\}{/\simeq}$$
sending a $G$-action  $\{F_g, \varepsilon_{g, h} |\; g, h\in G\}$ on $\mathcal{A}$ to    $\{\mathbf{K}^b(F_g), \mathbf{K}^b(\varepsilon_{g, h})|\; g, h\in G\}$, its natural extension  on $\mathbf{K}^b(\mathcal{A})$. \hfill $\square$
\end{cor}

Here, the action is said to \emph{fix} $\mathcal{A}$, provided that each  autoequivalence $F$ involved satisfying $F(\mathcal{A})=\mathcal{A}$, that is, $F(\mathcal{A})$ and $\mathcal{A}$ coincide up to the isomorphism closure in $\mathbf{K}^b(\mathcal{A})$.

For a finite dimensional algebra $A$, we denote by ${\rm proj}\mbox{-}A$ the category of finitely generated projective right $A$-modules. Examples of strongly $\mathbf{K}$-standard categories are ${\rm proj}\mbox{-}A$, provided that $A$ is \emph{triangular}, that is, its Gabriel quiver has no oriented cycles; see \cite[Proposition 4.6]{CY}.

\begin{prop}\label{prop:homot}
Let $A$ and $B$ be two finite dimensional algebras. Assume that there is a triangle equivalence $\mathbf{K}^b({\rm proj}\mbox{-}A)\rightarrow \mathbf{K}^b({\rm proj}\mbox{-}B)$. Then ${\rm proj}\mbox{-}A$ is (strongly) $\mathbf{K}$-standard if and only if so is ${\rm proj}\mbox{-}B$.
\end{prop}

\begin{proof}
The same argument as in \cite[Lemma 5.12]{CY} works in this situation. We mention that there is an analogue of \cite[Proposition 5.8 and Theorem 5.10]{CY} for the bounded homotopy category of projective modules.
\end{proof}

We do not know whether Proposition \ref{prop:homot} still holds if we replace the categories of projective modules by arbitrary additive categories.  On the other hand, we conjecture that ${\rm proj}\mbox{-}A$ is $\mathbf{K}$-standard for any finite dimensional algebra $A$. Indeed, by applying \cite[Sections 5 and 6]{CY}, this conjecture implies that any derived equivalence between finite dimensional algebras is standard. The latter is a well-known open question  in \cite[Section 3]{Ric}.
\end{appendix}

\vskip 10pt

\noindent {\bf Acknowledgements.}\quad  The authors are indebted to Henning Krause and Helmut Lenzing for their continued support. X.W. is very grateful to Hideto Asashiba for helpful discussions and to David Ploog for many useful comments. We thank the referee for many helpful suggestions.

This work is supported by the National Natural Science Foundation of China (No.s 11971398, 11971449 and 11801473), the Fundamental Research Funds for Central Universities of China (No.s 20720180002 and 20720180006),   and the Alexander von Humboldt Stiftung.

\bibliography{}

\vskip 15pt

 \noindent {\tiny  \noindent Jianmin Chen, Shiquan Ruan\\
School of Mathematical Sciences, \\
Xiamen University, Xiamen, 361005, Fujian, PR China.\\
E-mail: chenjianmin@xmu.edu.cn, sqruan@xmu.edu.cn\\}
\vskip 3pt

{\tiny \noindent    Xiao-Wu Chen \\
 Key Laboratory of Wu Wen-Tsun Mathematics, Chinese Academy of Sciences,\\
School of Mathematical Sciences, University of Science and Technology of China,\\
No. 96 Jinzhai Road, Hefei, 230026, Anhui, P.R. China.\\
E-mail: xwchen@mail.ustc.edu.cn, URL: http://home.ustc.edu.cn/$^\sim$xwchen}

\end{document}